\documentclass[a4paper,10pt]{amsart}
\usepackage[top=0.8in, bottom=0.8in, left=1.25in, right=1.25in]{geometry}
\usepackage{fancybox}
\usepackage{amsmath}
\usepackage{amssymb}
\usepackage{amsthm}
\usepackage{mathrsfs}
\usepackage{epstopdf}
\usepackage{epstopdf}
\usepackage{array} 
\usepackage{amscd}
\usepackage{color}
\usepackage{hyperref}
\usepackage{verbatim}

\usepackage[all]{xypic}
\usepackage{tikz}
\usepackage{tikz-cd}

\usepackage{amsfonts}
\usepackage{amssymb}
\usepackage{mathrsfs}

\DeclareFontFamily{U}{rsfs}{%
\skewchar\font127}
\DeclareFontShape{U}{rsfs}{m}{n}{%
<-6>rsfs5<6-8.5>rsfs7<8.5->rsfs10}{}
\DeclareSymbolFont{rsfs}{U}{rsfs}{m}{n}
\DeclareSymbolFontAlphabet
{\mathrsfs}{rsfs}
\DeclareRobustCommand*\rsfs{%
\@fontswitch\relax\mathrsfs}

\theoremstyle{plain}
\newtheorem{thm}{Theorem}[section]
\newtheorem{prop}[thm]{Proposition}
\newtheorem{lem}[thm]{Lemma}
\newtheorem{lemma}[thm]{Lemma}

\newtheorem{defi}[thm]{Definition}

\newtheorem{rmk}[thm]{Remark}
\newtheorem{cor}[thm]{Corollary}

\newtheorem{prop-defi}[thm]{Proposition-Definition}
\newtheorem{thm-defi}[thm]{Theorem-Definition}
\newtheorem{lem-defi}[thm]{Lemma-Definition}

\newtheorem{conj}[thm]{Conjecture}
\newtheorem{exam}[thm]{Example}
\newtheorem{example}[thm]{Example}

\newcommand{\vir}{\mathrm{vir}}
\newcommand{\rel}{\mathrm{rel}}
\newcommand{\red}{\mathrm{red}}
\newcommand{\std}{\mathrm{std}}

\newdimen\argwidth
\def\db[#1\db]{
 \setbox0=\hbox{$#1$}\argwidth=\wd0
 \setbox0=\hbox{$\left[\box0\right]$}
  \advance\argwidth by -\wd0
 \left[\kern.3\argwidth\box0 \kern.3\argwidth\right]}

\newcommand{\cC}{\mathcal{C}}

\newcommand{\eE}{\mathcal{E}}
\newcommand{\fF}{\mathcal{F}}
\newcommand{\gG}{\mathcal{G}}
\newcommand{\hH}{\mathcal{H}}

\newcommand{\oO}{\mathcal{O}}

\newcommand{\Hom}{\mathop{\rm Hom}\nolimits}

\newcommand{\dR}{\mathbf{R}}

\newcommand{\Hilb}{\mathop{\rm Hilb}\nolimits}

\newcommand{\Pic}{\mathop{\rm Pic}\nolimits}

\newcommand{\Chow}{\mathop{\rm Chow}\nolimits}

\newcommand{\id}{\textrm{id}}

\newcommand{\ch}{\mathop{\rm ch}\nolimits}

\newcommand{\td}{\mathop{\rm td}\nolimits}
\newcommand{\Ext}{\mathop{\rm Ext}\nolimits}

\newcommand{\rank}{\mathop{\rm rank}\nolimits}
\newcommand{\Coh}{\mathop{\rm Coh}\nolimits}

\newcommand{\ev}{\mathop{\rm ev}\nolimits}

\newcommand{\cneq}{\mathrel{\raise.095ex\hbox{:}\mkern-4.2mu=}}
\newcommand{\eqcn}{\mathrel{=\mkern-4.5mu\raise.095ex\hbox{:}}}

\newcommand{\Aut}{\mathop{\rm Aut}\nolimits}

\newcommand{\DT}{\mathop{\rm DT}\nolimits}

\newcommand{\GW}{\mathop{\rm GW}\nolimits}

\newcommand{\Sym}{\mathop{\rm Sym}\nolimits}

\newcommand{\End}{\mathop{\rm End}\nolimits}

\newcommand{\Ker}{\mathop{\rm Ker}\nolimits}

\newcommand{\BC}{{\mathbb{C}}}

\newcommand{\BE}{{\mathbb{E}}}
\newcommand{\BF}{{\mathbb{F}}}

\newcommand{\BH}{{\mathbb{H}}}

\newcommand{\BL}{{\mathbb{L}}}

\newcommand{\BQ}{{\mathbb{Q}}}
\newcommand{\BR}{{\mathbb{R}}}

\newcommand{\BZ}{{\mathbb{Z}}}

\newcommand{\CA}{{\mathcal A}}
\newcommand{\CB}{{\mathcal B}}
\newcommand{\CC}{{\mathcal C}}
\newcommand{\CE}{{\mathcal E}}

\newcommand{\CI}{{\mathcal I}}

\newcommand{\CO}{{\mathcal O}}
\newcommand{\CP}{{\mathcal P}}

\newcommand{\CU}{{\mathcal U}}
\newcommand{\CV}{{\mathcal V}}

\newcommand{\CX}{{\mathcal X}}

\newcommand{\CZ}{{\mathcal Z}}

\newcommand{\Fg}{{\mathfrak{g}}}

\newcommand{\Fq}{{\mathfrak{q}}}

\newcommand{\Ft}{{\mathfrak{t}}}

\newcommand{\fG}{{\mathfrak{G}}}

\newcommand\Kum{\mathrm{Kum}}
\newcommand{\pt}{{\mathsf{p}}}
\newcommand{\p}{{\mathbb{P}}}

\newcommand{\F}{{\mathsf{F}}}

\newcommand\Gr{\mathrm{Gr}}
\newcommand\pr{\mathrm{pr}}

\newcommand{\blangle}{\big\langle}
\newcommand{\brangle}{\big\rangle}

\renewcommand{\div}{\operatorname{div}}

\makeatletter
 
  \@addtoreset{equation}{section}
\makeatother

\newcommand{\Mbar}{{\overline M}}

\title[{Gopakumar-Vafa type invariants of holomorphic symplectic  4-folds}]
{Gopakumar-Vafa type invariants of \\ holomorphic symplectic 4-folds}
\date{}
\author{Yalong Cao}
\address{RIKEN Interdisciplinary Theoretical and Mathematical Sciences Program (iTHEMS), 2-1, Hirosawa, Wako-shi, Saitama, 351-0198, Japan}
\email{yalong.cao@riken.jp}
\author{Georg Oberdieck}
\address{University of Bonn, Institut f\"ur Mathematik}
\email{georgo@math.uni-bonn.de}
\author{Yukinobu Toda}
\address{Kavli Institute for the Physics and Mathematics of the Universe (WPI), The University of Tokyo Institutes for Advanced Study, The University of Tokyo, Kashiwa, Chiba 277-8583, Japan}
\email{yukinobu.toda@ipmu.jp}

\begin{document}
\maketitle
\begin{abstract}
Using reduced Gromov-Witten theory, we define new invariants
which capture the enumerative geometry of curves on
holomorphic symplectic 4-folds.
The invariants are analogous to the BPS counts of Gopakumar and Vafa for Calabi-Yau 3-folds, Klemm and Pandharipande for Calabi-Yau 4-folds,  
and Pandharipande and Zinger for Calabi-Yau 5-folds.

We conjecture that our invariants are integers and give a sheaf-theoretic interpretation in terms of reduced $4$-dimensional Donaldson-Thomas invariants of one-dimensional 
stable sheaves.
We check our conjectures for the product of two $K3$ surfaces and for the cotangent bundle of $\mathbb{P}^2$.
Modulo the conjectural holomorphic anomaly equation, we compute our invariants also for the Hilbert scheme of two points on a $K3$ surface.
This yields a conjectural formula for the number of isolated genus $2$ curves of minimal degree on a very general hyperk\"ahler $4$-fold of $K3^{[2]}$-type.
The formula may be viewed as a $4$-dimensional analogue of the classical Yau-Zaslow formula
concerning counts of rational curves on $K3$ surfaces.

In the course of our computations, we also derive a new closed formula for the Fujiki constants of the Chern classes of tangent bundles
of both Hilbert schemes of points on $K3$ surfaces and generalized Kummer varieties.
\end{abstract}

\setcounter{tocdepth}{1}
\tableofcontents

\section{Introduction}
\subsection{Gopakumar-Vafa invariants}
Gromov-Witten invariants of a smooth projective variety $X$ are defined
by integration over the virtual class \cite{BF, LT} of the moduli space $\Mbar_{g,n}(X,\beta)$ of genus $g$ degree $\beta \in H_2(X,\BZ)$ stable maps:
\begin{align}\label{intro GWinv}
\blangle \tau_{k_1}(\gamma_1) \cdots \tau_{k_n}(\gamma_n) \brangle^{\GW}_{g,\beta}
%\mathrm{GW}_{g, \beta}(\gamma_1, \ldots, \gamma_n)
=\int_{[\overline{M}_{g, n}(X, \beta)]^{\rm{vir}}}
\prod_{i=1}^n \mathrm{ev}_i^{\ast}(\gamma_i)\cdot\psi_i^{k_i}.
\end{align}
Here $\mathrm{ev}_i \colon \overline{M}_{g,n}(X, \beta)\to X$
is the evaluation map at the $i$-th marking, $\psi_i$ is the $i$-th cotangent line class,
and $\gamma_i \in H^{\ast}(X,\BQ)$ are cohomology classes.
Since $\overline{M}_{g,n}(X, \beta)$ is a Deligne-Mumford stack,
Gromov-Witten invariants are in general rational numbers, even if all $\gamma_i$ are integral.
Moreover the enumerative meaning of Gromov-Witten invariants is often not clear.
%because the moduli space parametrizes also maps with contracted components,  

For Calabi-Yau 3-folds, Gopakumar and Vafa \cite{GV} 
found explicit linear transformations
which transform the Gromov-Witten invariants
to a set of invariants (called {\em Gopakumar-Vafa invariants})
which they conjectured to be integers.
In an ideal geometry, where all curves are isolated, disjoint and smooth,
Gopakumar-Vafa invariants
should be the actual count of curves of given genus and degree.
The integrality of Gopakumar-Vafa invariants was proven recently in \cite{IP}.
A similar transformation of Gromov-Witten invariants into (conjectural) $\BZ$-valued invariants has been proposed for Calabi-Yau 4-folds by Klemm and Pandharipande \cite{KP},
and for Calabi-Yau 5-folds by Pandharipande and Zinger \cite{PZ}.
Universal transformations are expected in every dimension \cite{KP}.

Let $X$ be a holomorphic symplectic $4$-fold,
by which we mean a smooth complex projective 4-fold
which is equipped with a non-degenerate holomorphic 2-form $\sigma\in H^0(X,\Omega^2_X)$.
Since the obstruction sheaf has a trivial quotient, the ordinary Gromov-Witten invariants of $X$ vanish for all non-zero curve classes.
As a result, also all Klemm-Pandharipande invariants of $X$ vanish. %are hence not interesting.
Instead a reduced Gromov-Witten theory is obtained by Kiem-Li's cosection localization \cite{KiL}.
It is defined as in \eqref{intro GWinv} but by integration over the {\em reduced} virtual fundamental class:\footnote{We will only work with the reduced virtual class in this paper, hence we will denote it simply by $[ - ]^{\vir}$.}
\begin{equation}\label{intro vir cla}[\overline{M}_{g, n}(X, \beta)]^{\vir}\in A_{2-g+n}(\overline{M}_{g, n}(X, \beta)). \end{equation}

We are interested here in integer-valued invariants,
which underlie the (reduced) Gromov-Witten invariants \eqref{intro GWinv} of the holomorphic symplectic $4$-fold $X$. % and capture the enumerative geometry of its algebraic curves.
In genus $0$, all Gromov-Witten invariants can be reconstructed from the \textit{primary invariants},~i.e.~the integrals \eqref{intro GWinv} where all $k_i=0$.
Our proposal for the genus $0$ primary invariants is as follows:
\begin{defi}\emph{(Definition \ref{def of g=0 GV inv})}\label{intro def of g=0 GV inv}
For any $\gamma_1, \ldots, \gamma_n \in H^{\ast}(X,\BZ)$, 
we define the genus $0$ Gopakumar-Vafa invariant $n_{0, \beta}(\gamma_1, \ldots, \gamma_n) \in \BQ$ recursively by:
$$\blangle \tau_{0}(\gamma_1)\cdots\tau_{0}(\gamma_n) \brangle^{\GW}_{0,\beta}=\sum_{\begin{subarray}{c}k\geqslant   1, k|\beta  \end{subarray}}k^{n-3}\, n_{0, \beta/k}(\gamma_1, \ldots, \gamma_n). $$
\end{defi}
In fact, through a twistor space construction, this definition follows immediately from a similar definition on Calabi-Yau 5-folds given in \cite{PZ} (see~\S \ref{sect on def of gv} for more explanations). 

In genus 1, the situation is more complicated and does not follow from 5-fold geometry. 
Since the virtual dimension of \eqref{intro vir cla} is $1+n$, we require one marked point and an insertion $\gamma\in H^4(X,\mathbb{Z})$.
%to define genus 1 Gromov-Witten invariants. 
Because curves in imprimitive curve classes are very difficult to control in an ideal geometry (see Section~\ref{sect on heur arg})
we will restrict us to a primitive curve class (i.e.~where $\beta$ is not a multiple of a class in $H_2(X,\BZ)$).
\begin{defi}\emph{(Definition \ref{defn g1 GV primitive})} \label{intro defn g1 GV primitive}
Assume that $\beta \in H_2(X,\BZ)$ is primitive. For any $\gamma\in H^4(X, \mathbb{Z})$, we define
the genus $1$ Gopakumar-Vafa invariant $n_{1,\beta}(\gamma)\in \mathbb{Q} $ by
\[
\blangle \tau_0(\gamma) \brangle^{\GW}_{1,\beta}
= n_{1,\beta}(\gamma) - \frac{1}{24} \blangle \tau_0(\gamma) \tau_0(c_2(T_X)) \brangle_{0,\beta}^{\GW}.
\]
\end{defi}
In genus 2, the situation is even more complicated and attracting.
In fact, the appearance of genus 2 invariants is a new phenomenon that is not available on ordinary Calabi-Yau 4-folds and Calabi-Yau 5-folds.  
By the virtual dimension of \eqref{intro vir cla},
one expects a finite number of isolated genus $2$ curves.
The genus $2$ Gopakumar-Vafa invariant should be a count of these curves.
\begin{defi}\emph{(Definition \ref{defn g2 GV primitive})}\label{intro defn g2 GV primitive}
Assume that $\beta \in H_2(X,\BZ)$ is primitive. We define the genus $2$ Gopakumar-Vafa invariant $n_{2,\beta}\in \mathbb{Q}$ by
\[
\blangle \varnothing \brangle^{\GW}_{2,\beta}=n_{2,\beta}
- \frac{1}{24} n_{1,\beta}(c_2(X))
+ \frac{1}{2 \cdot 24^2} \blangle \tau_0(c_2(X)) \tau_0(c_2(X)) \brangle^{\GW}_{0,\beta}
+ \frac{1}{24} N_{\mathrm{nodal},\beta},
\]
%where $n_{1,\beta}(-)$ is given as in Definition \ref{intro defn g1 GV primitive} and
where $N_{\mathrm{nodal},\beta}\in \mathbb{Q}$ is the virtual count of rational nodal curves as defined in Eqn.~\eqref{Nnodal}.  
\end{defi}
Our first main conjecture is about the integrality of these definitions: 
\begin{conj}\emph{(Conjecture \ref{conj on integrality})}\label{intro conj on integrality}
With the notations as above, we have 
$$ n_{0,\beta}(\gamma_1,\ldots,\gamma_n), \,\,\, n_{1, \beta}(\gamma), \, \,\, n_{2, \beta} \in \mathbb{Z}. $$
%$$(i) \,\, n_{0,\beta}(\gamma_1,\ldots,\gamma_n)\in \mathbb{Z}, \,\,(ii) \,\, n_{1, \beta}(\gamma)\in \mathbb{Z}, \,\,(iii) \,\, n_{2, \beta} \in \mathbb{Z}. $$
\end{conj}

The definitions above are found via computations in 
an `ideal' geometry where we assume that algebraic curves behave in the expected way, see \S \ref{sect on heur arg}, \S \ref{subsec:ideal geometry primitive}.\footnote{Similar considerations in ideal geometries were taken by Klemm-Pandharipande on Calabi-Yau 4-folds \cite{KP} and Pandharipande-Zinger on Calabi-Yau 5-folds \cite{PZ}, though our case looks more complicated (see \S \ref{sect on heur arg} for details). }
We justify Conjecture~\ref{intro conj on integrality} in such an ideal case, which takes the whole \S \ref{sect on heuristic of g=1 GW/GV}, \S \ref{sect on heuristic of g=2 GW/GV}.

\subsection{GV/$\DT_4$ correspondence}
The second main theme of this paper is to give a sheaf theoretic interpretation of Gopakumar-Vafa invariants. This is motivated by
the parallel work of \cite{CMT1, CT2} on ordinary Calabi-Yau 4-folds.   

Let $M_\beta$ be the moduli scheme of one dimensional stable sheaves $F$ on $X$ with $\ch_3(F)=\beta$, $\chi(F)=1$. 
By \cite{KiP, Sav}, the ordinary
$\DT_4$ virtual class \cite{BJ, OT} (see also \cite{CL1}) of $M_{\beta}$ vanishes. 
By Kiem-Park's cosection localization \cite{KiP}, we instead have a (reduced) virtual class
\begin{align}\label{intro red vir class}[M_\beta]^{\vir}\in A_2(M_\beta,\mathbb{Q}). \end{align}
As usual, the virtual class depends on a choice of orientation \cite{CGJ, CL2}. More precisely, for each connected component of $M_{\beta}$, there 
are two choices of orientation which affect the virtual class by a sign (component-wise).
To define descendent invariants, consider the insertion operators:  
\begin{align*}\tau_i: H^{*}(X,\mathbb{Z})\to H^{*+2i-2}(M_\beta,\mathbb{Q}), \end{align*}
\begin{align*}\tau_i(\bullet):=(\pi_{M})_{\ast}\left(\pi_X^{\ast}(\bullet) \cup\ch_{3+i}(\mathbb{F}_{\mathrm{norm}})\right),
\end{align*}
where $\mathbb{F}_{\mathrm{norm}}$ is the normalized universal sheaf, i.e. $
\det(\pi_{M*}\mathbb{F}_{\mathrm{norm}})\cong \oO_{M_\beta}$.
As in Gromov-Witten theory, for any $\gamma_1, \ldots, \gamma_n \in H^{\ast}(X,\mathbb{Z})$ and $k_1,\ldots,k_n \in \mathbb{Z}_{\geqslant   0}$, we define $\DT_4$ invariants: 
\begin{align}\label{def of all DT4 invs}
\blangle\tau_{k_1}(\gamma_1),\ldots,\tau_{k_n}(\gamma_n) \brangle^{\DT_4}_{\beta}:=\int_{[M_{\beta}]^{\rm{vir}}}\prod_{i=1}^n\tau_{k_i}(\gamma_i)\in\mathbb{Q}.
\end{align}

Here is the second main conjecture of this paper, which gives a sheaf theoretic interpretation of our Gopakumar-Vafa invariants. 
\begin{conj}\emph{(Conjecture \ref{conj on DT4/GV})}\label{intro conj on DT4/GV}
For certain choice of orientation, the following holds. \\
When $\beta$ is an effective curve class, 
\begin{align}\label{g=0 part} \tag{i}
\blangle\tau_0(\gamma_1),\ldots,\tau_0(\gamma_n) \brangle^{\DT_4}_{\beta}=n_{0,\beta}(\gamma_1,\ldots,\gamma_n). \end{align}
When $\beta$ is a primitive curve class, 
\begin{align}\label{g=1 part} \tag{ii}
\blangle\tau_1(\gamma) \brangle^{\DT_4}_{\beta}=-\frac{1}{2}\blangle \tau_1(\gamma) \brangle^{\GW}_{0,\beta}-n_{1, \beta}(\gamma).
\end{align}
When $\beta$ is a primitive curve class, 
\begin{align}\label{g=2 part} \tag{iii}
-\blangle\tau_3(1) \brangle^{\DT_4}_{\beta}
-\frac{1}{12}\blangle\tau_1(c_2(X)) \brangle^{\DT_4}_{\beta}=n_{2, \beta}.
\end{align}
\end{conj}
As in Conjecture \ref{intro conj on integrality}, we verify these equalities in the ideal geometry (see \S \ref{sect on ideal DT4 comp0}, \S \ref{sect on ideal DT4 comp} and  also \S \ref{sec:embedded rational curve family} for details). An exception is the last equality involving genus 2 invariants, which we obtain indirectly through stable pair theory \cite{COT1} (see Remark \ref{rmk on dt4/pt}). 

Besides computations in the ideal geometry mentioned above, we study several examples and prove our conjectures in those cases.

\subsection{Verification of conjectures I: $K3\times K3$}
Let $X=S\times T$ be the product of two $K3$ surfaces. 
When the curve class $\beta \in H_2(S \times T, \BZ)$
is of non-trivial degree over both $S$ and $T$, then
the obstruction sheaf of the moduli space of stable maps has two linearly independent cosections,
which implies that the (reduced) Gromov-Witten invariants of $X$ in this class vanish. Therefore we always restrict ourselves to curve classes of form
\begin{equation}\label{intro curv cla}\beta\in H_2(S)\subseteq H_2(X). \end{equation} 
By Behrend's product formula \cite{Bproduct} (see Eqn.~\eqref{product formula}), we can easily compute all Gromov-Witten invariants and determine the Gopakumar-Vafa invariants as follows. 

\begin{thm}\emph{(Proposition \ref{prop on GW for prod})} \label{intro prop on GW for prod}
Let $\gamma,\gamma'\in H^{4}(X)$, $\alpha\in H^6(X)$ and let
$$\gamma=A_1\cdot 1\otimes \pt+D_1\otimes D_2+A_2\cdot \pt\otimes 1, \quad \gamma'=A'_1\cdot 1\otimes \pt+D'_1\otimes D'_2+A'_2\cdot \pt\otimes 1, $$
$$\alpha=\theta_1\otimes \pt+\pt\otimes \theta_2 $$
be their K\"unneth decompositions. 
Then we have
%the genus $0$ Gopakumar-Vafa invariants are:
\[
n_{0,\beta}(\gamma, \gamma') =(D_1\cdot\beta)\cdot (D_1'\cdot\beta)\cdot\int_T(D_2\cdot D_2')\cdot N_{0}\left(\frac{\beta^2}{2}\right),
\]
$$n_{0,\beta}(\alpha)=(\theta_1\cdot \beta)\,N_{0}\left(\frac{\beta^2}{2}\right). $$
If $\beta$ is primitive, we have
\begin{align*}
n_{1, \beta}(\gamma)= 24 A_2\,  N_1(\beta^2/2), \quad n_{2,\beta}= N_2\left( \frac{\beta^2}{2} \right), 
\end{align*}
where
\begin{align*}\sum_{l\in\mathbb{Z}}N_{0}(l)\, q^l&=\frac{1}{q} \prod_{n\geqslant   1}\frac{1}{(1-q^n)^{24}}, \\
\sum_{l \in \BZ} N_{1}(l)\,q^l &=\left(\frac{1}{q} \prod_{n\geqslant   1}\frac{1}{(1-q^n)^{24}}\right)\left(q \frac{d}{dq}G_2(q)\right), \\
\sum_{l\in\mathbb{Z}}N_{2}(l)\, q^l&=\left(\frac{1}{q} \prod_{n\geqslant   1}\frac{1}{(1-q^n)^{24}}\right) \left( 24 q \frac{d}{dq} G_2 - 24 G_2 - 1 \right), \end{align*}
%\begin{equation}\label{intro defi of N0}\sum_{l\in\mathbb{Z}}N_{0}(l)\, q^l=\frac{1}{q} \prod_{n\geqslant   1}\frac{1}{(1-q^n)^{24}},  \end{equation}
%\begin{equation}\label{intro equ py formula} \sum_{l \in \BZ} N_{1}(l)\,q^l =\left(\frac{1}{q} \prod_{n\geqslant   1}\frac{1}{(1-q^n)^{24}}\right) 
%\left(q \frac{d}{dq}G_2(q)\right),   \end{equation}
%\begin{equation}\label{intro defi of N2}\sum_{l\in\mathbb{Z}}N_{2}(l)\, q^l
%=\left(\frac{1}{q} \prod_{n\geqslant   1}\frac{1}{(1-q^n)^{24}}\right) \left( 24 q \frac{d}{dq} G_2 - 24 G_2 - 1 \right),\end{equation}
with Eisenstein series: 
$$G_2(q) = -\frac{1}{24} + \sum_{n \geqslant   1} \sum_{d|n} d q^n. $$
In particular, Conjecture \ref{intro conj on integrality} holds for $X=S\times T$. 
%where $n_{0,h}(S)$ $($$h\in \mathbb{Z}_{\geqslant  0}$$)$ fits into the generating series 
%\begin{equation}\label{equ on yz for}\sum_{h\geqslant  0}n_{0,h}(S)\, q^h=\prod_{n\geqslant   1}\frac{1}{(1-q^n)^{24}}.  \end{equation}
\end{thm}

On the Donaldson-Thomas side,
a main result of this paper is the explicit computation of all $\DT_4$ invariants of $X = S \times T$
for the classes \eqref{intro curv cla},
see Theorem \ref{thm:K3xK3 DT invariants} for the formulae.
We obtain a perfect match with our prediction:

\begin{thm}[Corollary~\ref{verify conjs}] 
%For one (of the two) orientations,
%With respect to a particular (canonical) choice of orientation,
Conjecture \ref{intro conj on DT4/GV} holds for $X = S \times T$ and all effective curve classes $\beta\in H_2(S,\mathbb{Z})\subseteq H_2(X,\mathbb{Z})$.
\end{thm}

Here, since the moduli space $M_{\beta}$ is connected, 
there are precisely two choices of orientation. 
We pick the one specified in Eqn.~\eqref{vir class StimesT}  (invariants for the other differ only by an overall sign).

Contrary to the case of Gromov-Witten invariants,
the computation of $\DT_4$ invariants on $S \times T$ is highly non-trivial.
In Theorem \ref{thm on vir clas}, we first identify the virtual class explicitly.
This expresses the $\DT_4$ invariants as tautological integrals on a (smooth) moduli space 
of one dimensional stable sheaves on the $K3$ surface $S$. By Markman's framework of monodromy operators \cite{Markman}, 
we then relate such integrals to tautological integrals on the Hilbert schemes of points on $S$ (see \S \ref{sect on des int on mod} and \S \ref{sect on transpot} for details).
Finally, we determine these integrals explicitly in \S \ref{sect on fujiki} and \S \ref{sect on des inte on hil}
by a combination of the universality result of Ellingsrud-G\"{o}ttsche-Lehn \cite{EGL}, constraints from Looijenga-Lunts-Verbitsky Lie algebra \cite{LL,Ver}
and known computations of Euler characteristics.

In particular, we found a remarkable closed formula 
for Fujiki constants of Chern classes of Hilbert schemes $S^{[n]}$ of points on $S$, which takes the following beautiful form (see also Proposition \ref{prop on fujiki on kumm} for the formula on generalized Kummer varieties):
\begin{thm}\emph{(Theorem \ref{thm on fujiki cons})}\label{intro thm on fujiki cons}
Let $S$ be a $K3$ surface. For any $k \geqslant   0$,
\[ \sum_{ n \geqslant   k} C( c_{2n-2k}(T_{S^{[n]}}))\,q^n 
=
\frac{ (2k)!}{ k! 2^{k}} \left( q \frac{d}{dq} G_2(q) \right)^k \prod_{n \geqslant   1} \frac{1}{(1-q^n)^{24}}. \]
\end{thm}
The right hand side, up to the combinatorical prefactor $(2k)! / (k! 2^k)$, is precisely the generating series of counts of genus $k$ curves on a $K3$ surface passing through $k$ generic points as considered by Bryan and Leung \cite{BL}. This suggests a relationship to the work of G\"ottsche on curve counting on surfaces \cite{Gott2},
which will be taken up in a follow-up work.

\subsection{Verification of conjectures II: $T^*\mathbb{P}^2$}
Let $T^{\ast} \p^2$ be the total space of the cotangent bundle on $\p^2$,
which is holomorphic symplectic.
Let $H \in H^2(T^{\ast} \p^2)$ be the pullback of the hyperplane class and use the identification $H_2(T^{\ast} \p^2, \BZ) \equiv \BZ$ given by taking the degree against $H$.
By Graber-Pandharipande's virtual localization formula \cite{GP}, we can compute all genus Gromov-Witten invariants (Proposition \ref{prop:TstartP2}) and determine the 
Gopakumar-Vafa invariants. 
\begin{prop}\emph{(Corollary~\ref{cor on inte on local p2})} 
\begin{align*}
n_{0,d}(H^2,H^2)&=
\left\{\begin{array}{rcl} 1   &\mathrm{if} \,\, d=1, \\ 
 -1  &\mathrm{if} \,\,  d=2,  \\
  0    & \,\,  \mathrm{otherwise}. 
\end{array} \right.  \\
n_{1,1}(H^2)&=0,  \quad n_{2,1}=0. 
\end{align*}
In particular, Conjecture \ref{intro conj on integrality} holds for $T^{\ast} \p^2$. 
\end{prop}
On the sheaf side, we can compute $\DT_4$ invariants for small degree curve classes. 
\begin{prop}\emph{(Proposition~\ref{prop on dt4 tp2})} 
For certain choice of orientation, we have 
$$\blangle\tau_0(H^2), \tau_0(H^2) \brangle^{\DT_4}_{1}=1, \quad \blangle\tau_0(H^2), \tau_0(H^2) \brangle^{\DT_4}_{2}=-1, \quad \blangle\tau_0(H^2), \tau_0(H^2) \brangle^{\DT_4}_{3}=0,$$
$$\blangle\tau_1(H^2) \brangle^{\DT_4}_{1}=-\frac{1}{2}, \quad \blangle\tau_1(H^2) \brangle^{\DT_4}_{2}=\frac{1}{2}, \quad \blangle\tau_1(H^2) \brangle^{\DT_4}_{3}=0, $$
$$\blangle\tau_2(H) \brangle^{\DT_4}_{1}=-\frac{1}{4}, \quad \blangle\tau_2(H)\brangle^{\DT_4}_{2 }=-\frac{1}{4}, \quad \blangle\tau_2(H) \brangle^{\DT_4}_{3}=0, $$
$$\blangle\tau_3(1) \brangle^{\DT_4}_{1}=-\frac{1}{8}, \quad \blangle\tau_3(1)\brangle^{\DT_4}_{2 }=\frac{1}{8}, \quad \blangle\tau_3(1) \brangle^{\DT_4}_{3}=0. $$
In particular, Conjecture \ref{intro conj on DT4/GV} \eqref{g=0 part} holds for all $d\leqslant 3$, and Conjecture \ref{intro conj on DT4/GV} \eqref{g=1 part}, \eqref{g=2 part} hold. 
\end{prop}

\subsection{Verification of conjectures III: $K3^{[2]}$}
Consider the Hilbert scheme $S^{[2]}$ of two points on a $K3$ surface $S$.
By a result of Beauville \cite{Beauville}, $S^{[2]}$ is irreducible hyperk\"ahler,
i.e. it is simply connected and the space of its holomorphic $2$-forms is spanned by a (unique) symplectic form.
Because the genus $0$ Gromov-Witten theory of $S^{[2]}$ is completely known by \cite{O1, O2, QuasiK3}
(see Theorem \ref{thm:Hilb genus 0} for the primitive case),
all genus $0$ Gopakumar-Vafa invariants are easily computed. 
For simplicity, we check the integrality conjecture in the following basic case (ref.~\S \ref{pf of inte of g0 for hilb2}):
\begin{thm} \label{thm:Hilb 2 genus 0 integrality}
Conjecture \ref{intro conj on integrality} holds for all effective curve classes on $S^{[2]}$ in genus $0$ and with one marked point.
\end{thm}
Higher genus Gromov-Witten invariants are more difficult to compute even for primitive curve classes. 
Nevertheless there are several conjectures on the structure of these invariants, including (i) a \textit{quasi-Jacobi form property}, and (ii) a \textit{holomorphic anomaly equation}
(see \cite[Conj.~A \&~C]{HilbHAE}, see also \cite{O3} for a progress report). Assuming these conjectures and using several explicit evaluations of Gromov-Witten invariants,
we obtain a complete computation of all genus $1$ and $2$ Gromov-Witten invariants of $S^{[2]}$ in primitive classes, see Theorem \ref{thm:Hilb genus 12}.
From this, all Gopakumar-Vafa invariants are computed in Theorems~\ref{thm: Hilb GV genus 1} and~\ref{thm: Hilb GV genus 2}.

With the help of a computer program, we obtain the following check of integrality:

\begin{thm}\emph{(Corollaries \ref{cor on chc int 1} and \ref{cor on chc int 2})}
Assume Conjectures A and C of \cite{HilbHAE}.
Then the genus 1 and 2 part of Conjecture \ref{intro conj on integrality} hold for all primitive curve classes $\beta\in H_2(S^{[2]},\mathbb{Z})$ satisfying $(\beta,\beta)\leqslant 100$, where $(-,-)$ is the Beauville-Bogomolov-Fujiki pairing as in \S \ref{sect on curv clas}.
\end{thm}

\subsection{A Yau-Zaslow type formula on $K3^{[2]}$}

A hyperk\"ahler variety is of $K3^{[2]}$-type if it is deformation-equivalent to the Hilbert scheme of $2$ points of a $K3$ surface $S$.
Given a primitive curve class $\beta \in H_2(S^{[2]},\BZ)$,
consider the very general deformation $(X,\beta')$ of a pair $(S^{[2]},\beta)$, where $\beta$ stays of Hodge type on all fibers.
By the deformation theory of hyperk\"ahler varieties, the
variety $X$ then has Picard rank $1$ and the algebraic classes in $H_2(X,\BZ)$ are generated by $\beta'$.
In particular $\beta'$ is irreducible.
In this case, it is natural to expect that curves in $(X,\beta')$ forms an ideal geometry in the sense of \S \ref{sect on heur arg}, \S \ref{subsec:ideal geometry primitive}.
In other words, after a generic deformation, our Gopakumar-Vafa invariants should give
enumerative information about curves in these hyperk\"ahler varieties of $K3^{[2]}$-type.

In genus $2$, this yields the following conjectural formula for the number of isolated (rigid) genus $2$ curves on a very general hyperk\"ahler variety of $K3^{[2]}$-type of minimal degree.
This may be viewed as a $4$-dimensional analogue of the classical Yau-Zaslow formula concerning counts of rational curves on $K3$ surfaces:
\begin{thm}\emph{(Theorem \ref{thm: Hilb GV genus 2})} \label{intro thm: Hilb GV genus 2} Assume Conjectures A and C of \cite{HilbHAE}.
For any hyperk\"ahler variety $X$ of $K3^{[2]}$-type
and primitive curve class $\beta \in H_2(X,\BZ)$, 
the genus $2$ Gopakumar-Vafa invariant $n_{2,\beta}$
is the coefficient determined by $\beta$ (see Definition~\ref{defn:beta coefficient})
of the quasi-Jacobi form
\begin{multline*} \widetilde{I}(y,q) = 
\frac{\Theta^2}{\Delta}\Bigg[ 
\frac{5}{384} \wp E_2^{3} + \frac{25}{6144} E_2^{4} + \frac{35}{384} \wp E_2^{2} - \frac{5}{512} E_2^{3} + \frac{5}{384} \wp E_2 E_4 + \frac{7}{3072} E_2^{2} E_4 \\
- \frac{71}{64} \wp E_2 + \frac{27}{512} E_2^{2} - \frac{47}{384} \wp E_4 + \frac{5}{4608} E_2 E_4 - \frac{13}{18432} E_4^{2} - \frac{1}{96} \wp E_6 \\
+ \frac{1}{1152} E_2 E_6 + \frac{9}{8} \wp - \frac{5}{32} E_2 - \frac{23}{1536} E_4 - \frac{5}{1152} E_6 + \frac{1}{8} \Bigg],
\end{multline*}
where the functions $\Theta,\Delta, \wp, E_i$ are defined in \S \ref{sect on qj form}.
\end{thm}

In genus $1$, it is convenient to encode the invariants in the \textit{genus 1 Gopakumar-Vafa class} 
\[ n_{1,\beta} \in H^4(X,\BQ) \]
which is defined by
\[ \forall\,\, \gamma \in H^4(X,\BQ): \quad \int_X n_{1,\beta} \cup \gamma = n_{1,\beta}(\gamma), \]
where $n_{1,\beta}(\gamma)$ is given in Definition \ref{intro defn g1 GV primitive}.
In an ideal geometry, $n_{1,\beta}$ is the class of the surface swept out by elliptic curves in class $\beta$.
Theorem~\ref{thm: Hilb GV genus 1} then yields a conjectural formula for this class.
We list the first values of the genus $1$ and $2$ Gopakumar-Vafa invariants of hyperk\"ahler varieties of $K3^{[2]}$-type in Table~\ref{bps_table} and Table~\ref{gvg2_table} below.
Since the deformation class of a pair $(X,\beta)$ where $\beta$ is a primitive curve class, only depends on the square $(\beta,\beta)$ (see \cite{O2}),
the Gopakumar-Vafa invariants only depend on $(\beta,\beta)$.

It is interesting to compare the enumerative significance of the listed invariants with the known geometry of curves on very general hyperk\"ahler 4-folds of $K3^{[2]}$-type with curve class $\beta$.
In the case $(\beta,\beta)=-5/2$, any curve in class $\beta$
is a line in a Lagrangian $\p^2 \subset X$, see \cite{HT}.
In particular, there are no higher genus curves, and indeed we observe the vanishing of the $g=1,2$ Gopakumar-Vafa invariants in this case.
Similarly, the case $(\beta,\beta)=-1/2$ corresponds to the exeptional curve class on $K3^{[2]}$
(the class of the exceptional curve of the Hilbert-Chow morphism $K3^{[2]} \to \mathrm{Sym}^{2}(K3)$),
and again there are no higher genus curves. The case $(\beta,\beta)=-2$ is similar, see \cite{HT}.
The first time we see elliptic curves is in case $(\beta,\beta)=0$,
which corresponds to the fiber class of a Lagrangian fibration $X \to \p^2$. Elliptic curves appear here in fibers over the discriminant.
The case $(\beta,\beta)=3/2$ corresponds to a very general Fano variety of lines on a cubic 4-fold,
with $\beta$ the minimal curve class (of degree $3$ against the Pl\"ucker polarization).
Since there are no cubic genus $2$ curves in a projective space (see also Example~\ref{example:Fano1}),
there are no genus $2$ curves in this class; again, this matches the vanishing observed in the table.
The case $(\beta,\beta)=2$ are the double covers of EPW sextics \cite{EPW}. %, and again we expect no genus $2$ curves.
The first time we should see isolated smooth genus $2$ curves is the case $(\beta,\beta)=11/2$,
which are precisely the Debarre-Voisin 4-folds \cite{DV}.
Here, the explicit geometry of curves has not been studied yet.
It would be very interesting to construct the expected $3465$ isolated smooth genus $2$ curves explicitly.
In fact, to the best of the authors' knowledge,
there exists so far no known example of a smooth isolated (rigid) genus $2$ curves on a hyperk\"ahler $4$-fold,
and this may be perhaps the simplest case.

\begin{table}[ht]
{\renewcommand{\arraystretch}{1.1}
\begin{tabular}{| c | c | c |}
\hline
\!$(\beta,\beta)$\!  & $a_{\beta}$   & $b_{\beta}$ \\
\hline
$-5/2$  & $0$ & $0$ \\
$-2$            & $0$ & $0$ \\
$-1/2$ & $ 0$ & $ 0$ \\
$0$ & $ 6$ & $ 1$ \\
$3/2$ & $ 105$ & $ 35/8$ \\
$2$ & $ 360$ & $ 30$ \\
$7/2$ & $ 3840$ & $ 40$ \\
$4$ & $ 9360$ & $ 300$ \\
$11/2$ & $ 74970$ & $ -6405/4$ \\
$6$ & $ 157080$ & $ -1540$ \\
$15/2$ & $ 1034496$ & $ -55224$ \\
$8$ & $ 1982820$ & $ -94570$ \\
$19/2$ & $ 11288760$ & $ -965720$ \\
$10$ & $ 20371680$ & $ -1702680$ \\
\hline
\end{tabular}
\quad
\begin{tabular}{| c | c | c |}
\hline
\!$(\beta,\beta)$\!  & $a_{\beta}$   & $b_{\beta}$ \\
\hline
$23/2$ & $ 103461120$ & $ -12187560$ \\
$12$ & $ 178607520$ & $ -21135240$ \\
$27/2$ & $ 826591920$ & $ -124077800$ \\
$14$ & $ 1378589520$ & $ -210090760$ \\
$31/2$ & $ 5903493120$ & $ -1077138720$ \\
$16$ & $ 9574935480$ & $ -1781067420$ \\
$35/2$ & $ 38376042111$ & $ -65957272227/8$ \\
$18$ & $ 60812926920$ & $ -13338391770$ \\
$39/2$ & $ 230147470080$ & $ -56902511160$ \\
$20$ & $ 357559991712$ & $ -90266652168$ \\
$43/2$ & $ 1286717384040$ & $ -359854419320$ \\
$22$ & $ 1965075202440$ & $ -560881363980$ \\
$47/2$ & $6762292992000$ & $-2110582343520$ \\
$24$ & $10172904142800$ & $-3237985250920$ \\
\hline
\end{tabular}}
\vspace{8pt}
\caption[]{The first coefficients of the genus $1$ Gopakumar-Vafa class\footnotemark
%\begin{minipage}{\linewidth}
$$n_{1,\beta} = \frac{1}{2} a_{\beta} h_{\beta}^2 + b_{\beta} c_2(T_X)$$
%\end{minipage}
for a hyperk\"ahler $4$-fold of $K3^{[2]}$-type with primitive curve class $\beta$
(see \S \ref{sect on curv clas} for the definition of the dual divisor $h_{\beta}$).
In an ideal geometry (ref.~\S \ref{subsec:ideal geometry primitive}), $n_{1,\beta}$ is the class of the surface swept out by the elliptic curves in class $\beta$.}
\label{bps_table} 
\end{table}

\begin{table}[ht]
{\renewcommand{\arraystretch}{1.1}
\begin{tabular}{| c | c|}
\hline
\!$(\beta,\beta)$\! & $n_{2,\beta}$ \\
\hline
$-5/2$  & $0$ \\
$-2$ & $0$ \\
$-1/2$ & $0$ \\
$0$ & $0$ \\
$3/2$ & $0$ \\
$2$ & $0$ \\
$7/2$ & $0$ \\
$4$ & $0$\\
$11/2$ & $ 3465$ \\
$6$ & $ 7920$ \\
$15/2$ & $ 153720$ \\
$8$ & $ 321300$ \\
$19/2$ & $ 3527370$ \\
$10$ & $ 6902280$ \\
\hline
\end{tabular}
\quad
\begin{tabular}{| c | c | c |}
\hline
\!$(\beta,\beta)$\!  & $n_{2,\beta}$ \\
\hline
$23/2$ & $ 55981800$ \\
$12$ & $ 104091120$ \\
$27/2$ & $ 691537770$ \\
$14$ & $ 1234210950$ \\
$31/2$ & $ 7087424400$ \\
$16$ & $ 12229093800$ \\
$35/2$ & $ 62706694050$ \\
$18$ & $ 105164743320$ \\
$39/2$ & $ 492018813720$ \\
$20$ & $ 805306494960$ \\
$43/2$ & $ 3490512517800$ \\
$22$ & $ 5593478602320$ \\
$47/2$ & $ 22715949849120$ \\
$24$ & $ 35731375344000$ \\
\hline
\end{tabular}}
\quad
\begin{tabular}{| c | c | c |}
\hline
\!$(\beta,\beta)$\!  & $n_{2,\beta}$ \\
\hline
$51/2$ & $ 137145316350735$ \\
$26$ & $ 212193639864360$ \\
$55/2$ & $ 775018459086480$ \\
$28$ & $ 1181532282033600$ \\
$59/2$ & $ 4129199523398880$ \\
$30$ & $ 6211686830906340$ \\
$63/2$ & $ 20865837137909400$ \\
$32$ & $ 31011424430679000$ \\
$67/2$ & $ 100506478032240210$ \\
$34$ & $ 147733008377317200$ \\
$71/2$ & $ 463428612330788160$ \\
$36$ & $ 674306145117002160$ \\
$75/2$ & $ 2052965259390710250$ \\
$38$ & $ 2959299345635755920$ \\
$79/2$ & $8765107896801841200$ \\
%$40$ & $12525573273890896800$ \\
\hline
\end{tabular}

\vspace{8pt}
\caption{The first genus 2 Gopakumar-Vafa invariants of a hyperk\"ahler $4$-fold of $K3^{[2]}$-type
in a primitive curve class $\beta$.}
\label{gvg2_table}
\end{table}

\subsection{Appendix}
In the appendix \S \ref{append}, we discuss several cases where we can extend the above GW/GV/$\DT_4$ correspondence to imprimitive curve classes. 

\subsection*{Notation and convention}
All varieties and schemes are defined over $\mathbb{C}$. 
For a morphism $\pi \colon X \to Y$ of schemes
and objects $\fF, \gG \in \mathrm{D^{b}(Coh(\textit{X\,}))}$ we will use
\[ \dR \hH om_{\pi}(\fF, \gG) := \dR \pi_{\ast} \dR \hH om_X(\fF, \gG). \]

A class $\beta\in H_2(X,\mathbb{Z})$ is called \textit{effective} if there exists a non-empty curve $C \subset X$ with class $[C] = \beta$. An effective class $\beta$ is called \textit{irreducible} if it is not the sum of two effective classes, and it is called \textit{primitive} if it is not a positive integer multiple of an effective class.

A holomorphic symplectic variety is a smooth projective variety
together with a non-degenerate holomorphic two form $\sigma\in H^0(X,\Omega^2_X)$. 
A holomorphic symplectic variety is \textit{irreducible hyperk\"ahler}
if $X$ is simply connected and $H^0(X, \Omega_X^2)$ is generated by a symplectic form.
A $K3$ surface is an irreducible hyperk\"ahler  variety of dimension $2$.
%simply connected smooth projective surface $S$ such that $K_S\cong \oO_S$.
%A hyperk\"ahler variety $X$ is a smooth projective variety which is simply connected and such that the space of holomorphic $2$-forms $H^0(X, \Omega_X^2)$ is generated by a symplectic form.

\footnotetext{Although there are fractional numbers in Table~\ref{bps_table}, the corresponding classes $n_{1,\beta}$ are integral, see Lemma~\ref{lemma:integrality}.}

\subsection*{Acknowledgement} 
We thank Chen Jiang, Young-Hoon Kiem, Sergej Monavari, Rahul Pandharipande and Hyeonjun Park for helpful discussions.
A special thanks goes to Luca Battistella for several useful discussions about the splitting of the Gromov-Witten virtual class in genus $2$ and 
to Sergej Monavari for his help on a localization computation.

Y. C. is partially supported by RIKEN Interdisciplinary Theoretical and Mathematical Sciences
Program (iTHEMS), World Premier International Research Center Initiative (WPI), MEXT, Japan, 
JSPS KAKENHI Grant Number JP19K23397 and Royal Society Newton International Fellowships Alumni 2020 and 2021. 
G.O. is partially supported by Deutsche Forschungsgemeinschaft (DFG) - OB 512/1-1. 
Y. T. is partially supported by World Premier International Research Center Initiative (WPI initiative), MEXT, Japan, and
Grant-in Aid for Scientific Research grant (No. 19H01779) from MEXT, Japan.

\section{Gopakumar-Vafa invariants}
Let $X$ be a holomorphic symplectic 4-fold
with symplectic form $\sigma \in H^0(X,\Omega^2_X)$.

In this section we first recall the definition of (reduced) Gromov-Witten invariants,
and then give our definition of Gopakumar-Vafa invariants.
In Section \ref{sect on heur arg}, we justify the definition 
by working in an ideal geometry of curves.

\subsection{Gromov-Witten invariants}
Let $\overline{M}_{g, n}(X, \beta)$
be the moduli space of $n$-pointed genus $g$ stable maps
to $X$ representing the non-zero curve class $\beta \in H_2(X,\BZ)$.
The moduli space $\Mbar_{g,n}(X,\beta)$ admits a perfect obstruction theory \cite{BF, LT}.
By the construction of \cite[\S 2.2]{MP_GWNL} the symplectic form $\sigma$ induces an everywhere surjective cosection of the obstruction sheaf.
By Kiem-Li's theory of cosection localization \cite{KiL} 
it follows that the standard virtual class as defined in \cite{BF,LT}
vanishes and instead there exists a reduced virtual fundamental class:
$$[\overline{M}_{g, n}(X, \beta)]^{\vir}\in A_{2-g+n}(\overline{M}_{g, n}(X, \beta)). $$
In this paper we will always work with the reduced virtual fundamental class
which we will hence simply denote by $[ - ]^{\vir}$.

Given cohomology classes $\gamma_i \in H^{\ast}(X)$ 
and integers $k_i \geqslant   0$ 
the (reduced) Gromov-Witten invariants of $X$ in class $\beta$ are defined by
\begin{align}\label{GWinv}
\blangle \tau_{k_1}(\gamma_1) \cdots \tau_{k_n}(\gamma_n) \brangle^{\GW}_{g,\beta}
%\mathrm{GW}_{g, \beta}(\gamma_1, \ldots, \gamma_n)
=\int_{[\overline{M}_{g, n}(X, \beta)]^{\rm{vir}}}
\prod_{i=1}^n \mathrm{ev}_i^{\ast}(\gamma_i)\cdot\psi_i^{k_i},
\end{align}
where $\mathrm{ev}_i \colon \overline{M}_{g,n}(X, \beta)\to X$
is the evaluation map at the $i$-th marking and $\psi_i$ is the $i$-th cotangent line class.
By the properties of the reduced virtual class, the integral \eqref{GWinv}
is invariant under deformations of the pair $(X,\beta)$ with preserve the Hodge type of the class $\beta$.
We call the invariant \eqref{GWinv} a \emph{primary Gromov-Witten invariant} if all the $k_i$ are zero.

\subsection{Relations}
We record several basic relations among genus $0$ Gromov-Witten invariants
which will be used later on in the text.
For the first reading, this section map be skipped.

\begin{lemma}\label{div equ on GW}
Let $D$ be a divisor on $X$ such that $d:= D\cdot \beta\neq 0$. Then 
$$\blangle \tau_1(\gamma) \brangle^{\GW}_{0,\beta}
=
\frac{1}{d^2} \blangle \tau_0(\gamma) \tau_0(D^2) \brangle^{\GW}_{0,\beta} -\frac{2}{d} \blangle \tau_0(\gamma \cdot D) \brangle^{\GW}_{0,\beta}.$$
\end{lemma}
\begin{proof}
By the divisor equation (e.g.~\cite[pp.~305]{CK})
\[ \blangle \tau_1(\gamma) \tau_0(D)^2 \brangle^{\GW}_{0,\beta} = d^2 \blangle \tau_1(\gamma) \brangle^{\GW}_{0,\beta} + 2 d \blangle \tau_0(\gamma\cdot D ) \brangle^{\GW}_{0,\beta}. \]
On the other hand, by rewriting $\psi_1$ in terms of boundary divisors 
and using the splitting formula for reduced virtual classes as in \cite[\S 7.3]{MPT}
one gets
\[ \blangle \tau_1(\gamma)  \tau_0(D)^2 \brangle^{\GW}_{0,\beta} = \blangle \tau_0(\gamma)\tau_0(D^2) \brangle^{\GW}_{0,\beta}. \qedhere \]
\end{proof}

\begin{lemma} \label{lemma:One relation}
For any $\gamma \in H^4(X)$, we have:
$\blangle \tau_1(\gamma) \brangle^{\GW}_{0,\beta}
=
\blangle \tau_2(1) \tau_0(\gamma) \brangle^{\GW}_{0,\beta}. $
\end{lemma}
\begin{proof}
Arguing as in Lemma~\ref{div equ on GW} we can express both sides in terms of primary Gromov-Witten invariants, which yields the result.
\begin{comment}
For the left hand side this was done in Lemma~\ref{div equ on GW}.
For the right hand side consider the invariant
\begin{equation} 
\blangle \tau_2(1) \tau_0(\gamma) \tau_0(D) \brangle^{\GW}_{0,\beta}. \label{invariant GW} 
\end{equation}
By rewriting $\psi_1$ in terms of boundary and then applying the string equation, we see \eqref{invariant GW} equals $\blangle \tau_1(1) \tau_0(\gamma D) \brangle^{\GW}_{0,\beta} - \blangle \tau_0(\gamma D) \brangle^{\GW}_{0,\beta}$.
On the other hand, by applying the divisor equation
we get that \eqref{invariant GW} also equals 
$d \blangle \tau_2(1) \tau_0(\gamma) \brangle^{\GW}_{0,\beta} +
\blangle \tau_1(D) \tau_0(\gamma) \brangle^{GW}_{0,\beta}$.
Combining both equalities yields
\[
d^2 \blangle \tau_1(\gamma) \brangle^{\GW}_{0,\beta} =
\blangle \tau_1(\gamma) \tau_0(D^2) \brangle^{\GW}_{0,\beta}
-2d \blangle \tau_0(\gamma D) \brangle^{\GW}_{0,\beta}.
\]
The claim now follows from by comparision with Lemma~\ref{div equ on GW}.
\end{comment}
\end{proof}

\begin{lemma} \label{lemma:Relation 2}
$\blangle \tau_3(1) \brangle_{0,\beta}^{\GW} = \blangle \tau_2(1) \tau_2(1) \brangle_{0,\beta}^{\GW}$.
%\[ \blangle \tau_1(c_2) \brangle_{0,\beta} = \blangle \tau_0(c_2), \tau_2(1) \brangle_{0,\beta}. \]
\end{lemma}
\begin{proof}
Let $D \in H^2(X)$ such that $d:= D \cdot \beta \neq 0$. Consider the following invariants:
\begin{alignat*}{2}
A_0 & =  \blangle \tau_3(1) \brangle_{0, \beta}^{\GW} & B_0 & = \blangle \tau_2(1) \tau_2(1) \brangle_{0,\beta}^{\GW} \\
A_1 & =  \blangle \tau_2(D) \brangle_{0, \beta}^{\GW} & B_1 & = \blangle \tau_2(1) \tau_1(D) \brangle_{0,\beta}^{\GW} \\
A_2 & =  \blangle \tau_1(D^2) \brangle_{0, \beta}^{\GW} & B_0 & = \blangle \tau_2(1) \tau_0(D^2) \brangle_{0,\beta}^{\GW} \\
A_3 & = \blangle \tau_0(D^3) \brangle_{0, \beta}^{\GW}  & & \\
C_2 & =  \blangle \tau_1(D) \tau_1(D) \brangle_{0, \beta}^{\GW} & C_3 & = \blangle \tau_1(D) \tau_0(D^2) \brangle_{0,\beta}^{\GW} \\
F & =  \blangle \tau_0(D^2) \tau_0(D^2) \brangle_{0, \beta}^{\GW}.
\end{alignat*}

Applying topological recursions to the invariants on the left then yields the following relations on the right:
\begin{alignat*}{2}
& \blangle \tau_3(1) \tau_0(D) \tau_0(D) \brangle_{0,\beta}^{\GW}: & \quad B_2 & = d^2 A_0 + 2d A_1 + A_2 \\
& \blangle \tau_2(D) \tau_0(D) \tau_0(D) \brangle_{0,\beta}^{\GW}: & \quad C_3 & = d^2 A_1 + 2d A_2 + A_3 \\
& \blangle \tau_1(D^2) \tau_0(D) \tau_0(D) \brangle_{0,\beta}^{\GW}: & \quad F & = d^2 A_2 + 2d A_3 \\
& \blangle \tau_2(1) \tau_0(D^3) \brangle_{0,\beta}^{\GW}: & \quad - A_3 & = d B_2 + C_3 \\
& \blangle \tau_1(D) \tau_0(D^2) \tau_0(D) \brangle_{0,\beta}^{\GW}: & \quad d A_3 & = d C_3 + F \\
& \blangle \tau_2(1) \tau_2(1) \tau_0(D) \brangle_{0,\beta}^{\GW}: & \quad 0 & = d B_0 + 2 B_1 \\
& \blangle \tau_2(1) \tau_1(D) \tau_0(D) \brangle_{0,\beta}^{\GW}: & \quad 0 & = d B_1 + C_2 + B_2 \\
& \blangle \tau_1(D) \tau_1(D) \tau_0(D) \brangle_{0,\beta}^{\GW}: & \quad 0 & = d C_2 + 2 C_3. 
\end{alignat*}
Putting all together (using the assistance of a computer) one finds:
\begin{equation} \blangle \tau_3(1) \brangle_{0,\beta}^{\GW} = A_0 = - \frac{8}{d^3} A_3 + \frac{6}{d^4} F = B_0 = \blangle \tau_2(1) \tau_2(1) \brangle_{0,\beta}^{\GW}. \label{tau31} 
\qedhere \end{equation}
%The second claim follows similarly (but a shorter calculation). One obtains:
%\[
%\blangle \tau_1(c_2) \brangle_{0,\beta}
%=
%\frac{1}{d^2} \Big( \blangle \tau_0(c_2) \tau_0(D^2) \brangle_{0,\beta} - 2d \blangle \tau_0( c_2 \cdot D ) \brangle_{0,\beta} \Big)
%=
%\blangle \tau_0(c_2) \tau_2(1) \brangle_{0,\beta}.
%\]
\end{proof}

\begin{lemma} \label{lemma:psi1 in terms of Tp}
Assume that all fibers of the universal curve $p : \CC \to \Mbar_{0,0}(X,\beta)$ are isomorphic to $\p^1$.
Let $\pi : \Mbar_{0,1}(X,\beta) \to \Mbar_{0,0}(X,\beta)$ be the forgetful morphism.
Then
\[ c_1(\omega_{\pi}) = \psi_1. \]
In particular, with $f : \CC \to X$ the universal map, we have
\begin{equation}\label{equ on des g0 gw ideal}
\blangle \tau_1(\gamma) \brangle^{\GW}_{0, \beta} 
= \int_{ p^{\ast}[ \Mbar_{0,0}(X,\beta) ]^{\vir} } f^{\ast}(\gamma)\,c_1(\omega_p).
\end{equation}
\end{lemma}
\begin{proof}
Let $\tilde{p} : \CC_1 \to \Mbar_{0,1}(X,\beta)$ be the universal curve and let $s : \Mbar_{0,1}(X,\beta) \to \CC_1$ be the universal section.
By definition, we have
\[ \psi_1 = s^{\ast}( c_1(\omega_{\tilde{p}}) ) = s^{\ast} c_1(\Omega_{\tilde{p}}) . \]
Recall that we have $\CC\cong \Mbar_{0,1}(X,\beta)$. Moreover, since $\CC \to \Mbar_{0,0}(X,\beta)$ parametrizes only smooth curves,
we have
\[ \CC_{1} \cong \CC \times_{\Mbar_{0,0}(X,\beta)} \CC. \]
Under this isomorphism, the section $s$ is identified with the diagonal morphism.
We have the fiber diagram
\[
\begin{tikzcd}
\CC \times_{\Mbar_{0,0}(X,\beta)} \CC \ar{r}{\tilde{\pi}} \ar{d}{\tilde{p}} & \CC \ar{d}{p} \\
\Mbar_{0,1}(X,\beta) \ar{r}{\pi} \ar[bend left]{u}{s} & \Mbar_{0,0}(X,\beta).
\end{tikzcd}
\]
Hence since $\tilde{\pi} \circ s = \id$, we have
\[ \psi_1 = s^{\ast} c_1(\Omega_{\tilde{p}}) = s^{\ast} \tilde{\pi}^{\ast} c_1(\Omega_{p}) = c_1(\Omega_{p}). \]
The second part follows since
\[ [\Mbar_{0,1}(X,\beta)]^{\vir} = \pi^{\ast} [\Mbar_{0,0}(X,\beta) ]^{\vir}. \qedhere \]
\end{proof}

\subsection{Definition of GV invariants}\label{sect on def of gv}
We consider the definition of Gopakumar-Vafa invariants.

In genus 0, by \cite{BL, MP_GWNL},
reduced Gromov-Witten invariants of $X$
are equal to the (ordinary) Gromov-Witten invariants in fiber classes of the twistor space
$\CX \to \p^1$ associated to the symplectic form $\sigma$
(alternatively, we can view $X$ embedded in a suitable $1$-parameter family of holomorphic symplectic $4$-folds such that the corresponding classifying map is transverse to the Noether-Lefschetz divisor defined by $\beta$).
The definition of genus $0$ Gopakumar-Vafa invariants for Calabi-Yau 5-folds proposed by Pandharipande and Zinger in \cite[Eqn.~(0.2)]{PZ}
hence can be viewed as a definition
for genus $0$ Gopakumar-Vafa invariants of $X$
as follows:
%
%In this way we can obtain a definition of genus 0 GV invariants of $X$ from the
%Using \cite[Eqn. (0.2)]{PZ} we obtain the following:
\begin{defi}\label{def of g=0 GV inv}
For any $\gamma_1, \ldots, \gamma_n \in H^{\ast}(X,\BZ)$, 
we define the genus $0$ Gopakumar-Vafa invariant $n_{0, \beta}(\gamma_1, \ldots, \gamma_n) \in \BQ$ by
$$\blangle \tau_{0}(\gamma_1)\cdots\tau_{0}(\gamma_n) \brangle^{\GW}_{0,\beta}=\sum_{\begin{subarray}{c}k\geqslant   1, k|\beta  \end{subarray}}k^{n-3}\, n_{0, \beta/k}(\gamma_1, \ldots, \gamma_n). $$
\end{defi}
The case of genus $1$ does {\em not} follow from the 5-fold geometry, 
since the virtual class of the moduli spaces differ by a factor of $(-1)^g \lambda_g$,
see \cite{MP_GWNL,O2}.
Instead we propose a definition of genus $1$ Gopakumar-Vafa invariants based on computations in an ideal geometry of curves in class $\beta$.
Because curves in imprimitive curve classes are very difficult to control,
we restrict hereby to the primitive case
(i.e.~to those $\beta$ which are not a multiple in $H_2(X,\BZ)$).
Consider the Chern classes of the tangent bundle of $X$:
\[ c_k(X) := c_k(T_X) \in H^{2k}(X,\BZ). \]
\begin{defi} \label{defn g1 GV primitive}
Assume that $\beta \in H_2(X,\BZ)$ is primitive. For any $\gamma\in H^4(X, \mathbb{Z})$, we define
the genus $1$ Gopakumar-Vafa invariant $n_{1,\beta}(\gamma)\in \mathbb{Q} $ by
\[
\blangle \tau_0(\gamma) \brangle^{\GW}_{1,\beta}
= n_{1,\beta}(\gamma) - \frac{1}{24} \blangle \tau_0(\gamma) \tau_0(c_2(T_X)) \brangle_{0,\beta}^{\GW}.
\]
\end{defi}
%The restriction to primitive curve classes above reflects the fact that the geometry of curves in imprimitive classes is very difficult to control.
Next we come to the genus $2$ Gopakumar-Vafa invariants.
Since the virtual dimension of the moduli space $\Mbar_{2,0}(X,\beta)$ is zero,
GV invariants are defined without cohomological constraints.
In other words, we expect that $n_{2,\beta}$ should be given by the enumerative count of genus $2$ curves in class $\beta$. For the definition 
we require the following invariant introduced in \cite{NO}:
\begin{equation} \label{Nnodal}
N_{\mathrm{nodal},\beta}:=
\frac{1}{2}\left[
\int_{[\overline{M}_{0,2}(X,\beta)]^{\vir}} (\ev_1 \times \ev_2)^{\ast}(\Delta_X) - \int_{[ \Mbar_{0,1}(X,\beta) ]^{\vir}} \frac{\ev_1^{\ast}(c(T_X))}{1-\psi_1}
\right], \end{equation}
where 
\begin{itemize}
\item $\Delta_X \in H^8(X \times X)$ is the class of the diagonal, and
%\item $\lambda_k$ is the $k$-th Chern class of the Hodge bundle $\BE \to \Mbar_{1,0}(X,\beta)$ (with fiber $H^0(C,\omega_C)$ at a point $[f : C \to X]$), and
\item $c(T_X) = 1 + c_2(T_X) + c_4(T_X)$ is the total Chern class of $T_X$.
\end{itemize}
The invariant $N_{\mathrm{nodal},\beta}$ is the expected number of rational nodal curves in class $\beta$ \cite[Prop.~1.2]{NO}\footnote{Here 
we use genus reduction to rewrite the term $N_X$ in \cite[\S 1.3]{NO} as the first term in Eqn.~\eqref{Nnodal}.}.
%since $24 \lambda_1 \in H^2(\Mbar_{1,1})$ is Poincar\'e dual to the class of a point the first term in \eqref{Nnodal} 
%is the Gromov-Witten count of elliptic curves of fixed $j$-invariant,
%and the second term in \eqref{Nnodal} is the excess contribution coming from genus $0$ curves. The first term also can be computed from genus $0$ Gromov-Witten invariants by the basic reduction:
%\begin{align*}
%\int_{[\overline{M}_{1,0}(X,\beta)]^{\vir}} 24 \lambda_1
%& = \int_{[\overline{M}_{1,1}(X,\beta)]^{\vir}} 24 \lambda_1 \ev_1^{\ast}(D_{\beta}) \\
%& = \int_{[\overline{M}_{0,3}(X,\beta)]^{\vir}} \ev_1^{\ast}(D_{\beta}) (\ev_2 \times \ev_3)^{\ast}(\Delta_X) \\
%& = \blangle \tau_0 \tau_0(\Delta_X) \brangle_{0,\beta}
%\end{align*}
%where 
%\begin{itemize}
%\item $D_{\beta} \in H^2(X,\mathbb{Q})$ is any divisor class with $\int_{\beta} D_{\beta} = 1$
%\end{itemize}
\begin{defi}\label{defn g2 GV primitive}
Assume that $\beta \in H_2(X,\BZ)$ is primitive. We define the genus $2$ Gopakumar-Vafa invariant $n_{2,\beta}\in \mathbb{Q}$ by
\[
\blangle \varnothing \brangle^{\GW}_{2,\beta}=n_{2,\beta}
- \frac{1}{24} n_{1,\beta}(c_2(X))
+ \frac{1}{2 \cdot 24^2} \blangle \tau_0(c_2(X)) \tau_0(c_2(X)) \brangle^{\GW}_{0,\beta}
+ \frac{1}{24} N_{\mathrm{nodal},\beta}.
\]
\end{defi}

\begin{rmk}
For primitive $\beta \in H_2(X,\BZ)$, we obtain the following:
\begin{gather*}
n_{0,\beta}(\gamma_1, \ldots, \gamma_n) = \blangle \tau_0(\gamma_1) \cdots \tau_0(\gamma_n) \brangle_{0,\beta}^{\GW}, \\
n_{1,\beta}(\gamma) = \blangle \tau_0(\gamma) \brangle^{\GW}_{1,\beta}
+ \frac{1}{24} \blangle \tau_0(\gamma) \tau_0(c_2(X)) \brangle^{\GW}_{0,\beta}, \\
n_{2,\beta}
= \blangle \varnothing \brangle^{\GW}_{2,\beta}
+ \frac{1}{24} \blangle \tau_0(c_2(X)) \brangle_{1,\beta}^{\GW}
+ \frac{1}{2 \cdot 24^2} \blangle \tau_0(c_2(X))\tau_0(c_2(X)) \brangle^{\GW}_{0,\beta}
- \frac{1}{24}  N_{\mathrm{nodal},\beta}.
\end{gather*}
It would be interesting to obtain a conceptual understanding for the form of these formulae.
\end{rmk}

As in the cases of Calabi-Yau 4-folds and 5-folds \cite{KP, PZ}, our first main conjecture concerns the integrality of the Gopakumar-Vafa invariants on holomorphic symplectic 4-folds.
\begin{conj}[Integrality]\label{conj on integrality}
We have 
$$n_{0,\beta}(\gamma_1,\ldots,\gamma_n), \,\,\, n_{1, \beta}(\gamma), \,\,\, n_{2, \beta} \in \mathbb{Z}. $$
\end{conj}

\subsection{Ideal geometry}\label{sect on heur arg}
We will justify our definition of Gopakumar-Vafa invariants
by working in an  `ideal' geometry where we assume curves on $X$ deform in families of expected dimensions and have expected genericity properties.
This discussion is inspired by the  `ideal' geometry of curves on Calabi-Yau 4-folds by \cite{KP} and on Calabi-Yau 5-folds by \cite{PZ}.
%However we will see notable differences.
Concretely, since the virtual dimension of $\Mbar_{g,0}(X,\beta)$ is $2-g$, we expect that:
\begin{quote}
Any genus $g$ curve moves in a smooth compact $(2-g)$-dimensional family.
\end{quote}
In particular, there are no curves of genus $g \geqslant   3$.

We discuss now the expected behaviour of the curves in these families.
%Compared to the case of CY 4-folds and CY 5-folds \cite{KP,PZ},
%we will see that the behaviour of curves in these families
%is more complicated, even in ideal situation.
We start with genus zero. Let $p : \CC^0_{\beta} \to S^0_{\beta}$ be a family of rational curves in class $\beta$ over a smooth $2$-dimensional surface $S^0_{\beta}$, fiberwise embedded in $X$. Then we can have the following behaviour:
\begin{enumerate}
\item[(i)] All the curves parametrized by $S^0_{\beta}$ can be reducible.

Reason: Let $\beta = \beta_1 + \beta_2$ and let $\CC^0_{\beta_i} \to S^0_{\beta_i}$ be a $2$-dimensional family of rational curves in class $\beta_i$.
Let $S^0_{\beta_1, \beta_2}$ be the preimage of the diagonal under the evaluation maps
\[ j_1 \times j_2 : \CC^0_{\beta_1} \times \CC^0_{\beta_2} \to X \times X. \]
Then $S^0_{\beta_1, \beta_2}$ is of expected dimension $3+3-4=2$, so by gluing the curves we can obtain a $2$-dimensional family of reducible rational curves in class $\beta$.

\item[(ii)] Given a generic curve $\CC^0_s:=p^{-1}(s) \subset X$ in the family, there exists another curve $\CC^0_{s'} \subset X$ in the family which meets it.

Reason: This follows by the same reasoning as in (i).

\item[(iii)] For finitely many $s \in S$, we expect the curve $\CC^0_s \subset X$ to be nodal\footnote{
A naive model for (ii,iii) would be for the image of $\CC_\beta^0$ in $X$ to be $(S \times \p^1) / ((s,0) \sim (gs, \infty))$, where $g \in \Aut(S)$ is an automorphism
with a finite number of fixed points (which lead to nodal curves).}.

Reason: The moduli space $\Mbar_{0,2}(X,\beta)$ is of expected dimension $4$, and hence the preimage of the diagonal under $\ev_1 \times \ev_2$ is of expected dimension $0$.\footnote{This is related to what is called the double point number.}

\item[(iv)] Even if all fibers of $\CC^0_{\beta} \to S_{\beta}$ are smooth $\p^1$'s, the natural morphism $j : \CC^0_{\beta} \to X$ is not necessarily an immersion.

(The differential $dj : T_{\CC^0_{\beta}} \to j^{\ast}(T_X)$ is expected to have a kernel in codimension $\geqslant   2$.)
\end{enumerate}

Similarly given a family $p : \CC^1_{\beta} \to S^1_{\beta}$ of elliptic curves in class $\beta$ over a smooth $1$-dimensional curve $S^1_{\beta}$, fiberwise embedded in $X$, all the curves parametrized by $S^1_{\beta}$ can be reducible. The argument is similar to (i) above, by considering the preimage of the diagonal under the evaluation maps
\[ j_1 \times j_2 : \CC^0_{\beta_1} \times \CC^1_{\beta_2} \to X \times X, \quad \mathrm{where}\,\, \beta=\beta_1+\beta_2, \]
where $\CC^0_{\beta_1}$ (resp. $\CC^1_{\beta_2}$) is a family of rational curves in class $\beta_1$ (resp. elliptic curves in class $\beta_2$).

The genus $2$ curves we expect to be smooth and finite. By dimension reasons they should be disjoint from elliptic curves, but can have finite intersection points with the family of rational curves. In the moduli space $\Mbar_{2,0}(X,\beta)$ we will hence also see genus $2$ curves with rational tails.

In summary, the geometry of curves is more complicated then for both CY 4-folds and CY 5-folds. Especially for imprimitive curve classes $\beta$, it becomes increasingly difficult to control.

%Since the virtual dimension is larger, they behave more complicated than the
%families of curves encounted on CY4 and CY5-folds.
\subsection{Ideal geometry: Primitive case}\label{subsec:ideal geometry primitive}
We make the following additional assumptions:
\begin{itemize}
\item $X$ is irreducible hyperk\"ahler,
\item the effective curve class $\beta \in H_2(X,\BZ)$ is primitive.
\end{itemize}
By the global Torelli for (irreducible) hyperk\"ahler varieties \cite{Ver, Huy}
(in fact, the local surjectivity of the period map is sufficient)
the pair $(X,\beta)$ is deformation equivalent (through a deformation with keeps $\beta$ of Hodge type)
to a pair $(X', \beta')$ where $\beta' \in H_2(X,\BZ)$ is irreducible.
Hence we may without loss of generality make the following stronger assumption:\footnote{The statement is false if we do not assume that $X$ is irreducible hyperk\"ahler, for example on $X$ the product of two $K3$ surfaces,
a class $\beta = (\beta_1, \beta_2)$ with both $\beta_i$ non-zero effective,
any deformation that keeps $\beta$ Hodge, will keep both $\beta_i$ Hodge. In particular $\beta$ stays reducible under deformations.}
\begin{itemize}
\item the effective curve class $\beta \in H_2(X,\BZ)$ is irreducible.
\end{itemize}
%assume that the class is irreducible $\beta$.
Under these assumptions our ideal geometry of curves simplifies to the following form:
\begin{enumerate}
\item
The rational curves in $X$ of class $\beta$
move in a proper 2-dimensional smooth family of embedded irreducible rational curves. Except for a finite number of rational nodal curves, the rational curves are smooth, with normal bundle $\CO_{\p^1} \oplus \CO_{\p^1} \oplus \mathcal{O}_{\mathbb{P}^{1}}(-2)$. 
\item
The arithmetic genus $1$ curves in $X$ of class $\beta$ move in a proper 1-dimensional smooth family of embedded irreducible genus 1 curves. Except for a finite number of rational nodal curves, the genus one curves are smooth elliptic curves. For the convenience of computations, we also assume the normal bundle of elliptic curves is $L\oplus L^{-1}\oplus \oO$, where $L$ is a generic degree zero line bundle.
%Finitely many of the curves are rational nodal.
\item
All genus two curves are smooth and rigid.
\item
There are no curves of genus $g\geqslant   3$.
\end{enumerate}

\begin{example} \label{example:Fano1}
Let $Y \subset \p^5$ be a very general smooth cubic 4-fold and let $F(Y) \subset \mathrm{Gr}(2,6)$ be the Fano variety of lines on $Y$. By a result of Beauville and Donagi \cite{BD}, $F(Y)$ is an irreducible hyperk\"ahler 4-fold, and together with its Pl\"ucker polarization it is the generic member of a locally complete family of polarized hyperk\"ahler varieties deformation equivalent to the second punctual Hilbert scheme of a $K3$ surface.
The algebraic classes in $H_2(F(Y),\BZ)$ are of rank $1$. Let $\beta$ be the generator
which pairs positively with the polarization (it is of degree $3$ with respect to the Pl\"ucker polarization).
The geometry of curves in class $\beta$ has been studied in \cite{OSY, NO, GK}.
The Chow variety of curves in class $\beta$ is given by
\[ \Chow_{\beta}(F(Y)) = S \cup \Sigma, \]
where $S \subset F(Y)$ is the smooth irreducible surface of lines of second type, and $\Sigma$ is a smooth curve parametrizing genus $1$ curves.
There are precisely $3780$ rational nodal curves corresponding to the intersection points $S \cap \Sigma$, and all other rational curves are isomorphic to $\p^1$. Moreover, there are no curves of genus $\geqslant   2$ in class $\beta$. We see that the curves in $F(Y)$ of class $\beta$ satisfy the requirements of the ideal geometry.
\end{example}

\subsection{Justification:~GV in genus $1$}\label{sect on heuristic of g=1 GW/GV}
For a given class $\gamma \in H^4(X,\BZ)$ let $\Gamma \subset X$ be a generic topological cycle
whose class is Poincar\'e dual to $\gamma$.
In an ideal geometry, the Gopakumar-Vafa invariant $n_{1,\beta}(\gamma)$ 
should be the (enumerative) number $n(\Gamma)$ of arithmetic genus $1$ curves in $X$ which are incident to $\Gamma$.
To derive an expression for it using Gromov-Witten invariants,
we start with the genus $1$ Gromov-Witten invariant:
\begin{equation} \blangle \tau_0(\gamma) \brangle^{\GW}_{1,\beta}, \label{g1 invariant} \end{equation}
where $\beta \in H_2(X,\BZ)$ is primitive.
Assuming the ideal geometry of Section~\ref{subsec:ideal geometry primitive}
we will analyze the contributions from genus $0$ and genus $1$ curves to it
(there are no contributions from genus $2$ curves since they never meet the cycle $\Gamma$). We show that the contribution from genus $1$ curves is precisely $n(\Gamma)$.
This will yield the expression for $n_{1,\beta}(\gamma)$.
%We study the contributions of rational curve families (RCF)  and elliptic curve families (ECF) to the genus $1$ GW invariants
%$\blangle \tau_0(\gamma) \brangle^{\GW}_{1,\beta}$ for a cycle $\gamma \in H^4(X)$. As the 4-cycle generically avoids 
%genus 2 curves, we do not need to consider their contributions. 

\subsubsection{Contribution from genus one curves} \label{subsubsec:contributions gv genus 1 from genus 1}
Let $p:\mathcal{C}^1_\beta \to S^1_\beta$ be a $1$-dimensional
family of elliptic curves of class $\beta$
as in Section~\ref{subsec:ideal geometry primitive}, and let $j : \mathcal{C}^1_\beta \to X$ be the evaluation map. %(which we assume to be an embedding).

Since $\Gamma$ (which represents $\gamma \in H^4(X,\BZ)$) is chosen generic,
it intersects $\CC^1_\beta$
in precisely $(\CC^1_\beta \cdot \gamma)$ many points.
Following Section~\ref{subsec:ideal geometry primitive}, we assume that the incident curves are smooth elliptic curves $E$ with normal bundle
$N_{E/X} = L \oplus L^{-1} \oplus \CO$.
We find the contribution of this family to the invariant \eqref{g1 invariant}
is
\[
(\CC^1_\beta \cdot \gamma) \int_{ [\Mbar_{1,1}(E,1)]^{\text{vir}}} \ev_1^{\ast}(\omega) = (\CC^1_\beta \cdot \gamma) = n(\Gamma),
\]
where $\omega \in H^2(E,\BZ)$ is the class of a point
and the trivial factor $H^1(E,N_{E/X}) = H^1(E,\CO_E) = \BC$ 
in the obstruction sheaf
does not appear because we used the reduced virtual fundamental class.
\begin{comment}
of degree $d \beta$ to be:
\[
(\CC^1_\beta \cdot \gamma) \int_{ [\Mbar_{1,1}(E,d)]^{\text{vir}}} \ev_1^{\ast}(\omega) = (\CC^1_\beta \cdot \gamma)\,\sigma(d),
\]
where the reduced obstruction theory accounts for removing the trivial factor $H^1(E,N_{E/X}) = H^1(E,\CO_E) = \BC$.
Summing over all elliptic curve families, we justify terms in RHS of \eqref{equ of pot def of g=1}.  
\end{comment}

\subsubsection{Contribution from genus zero curves}
\label{subsubsec:Contributions genus zero to GV genus 1}
Let $p : \CC^0_\beta \to S^0_\beta$ be a $2$-dimensional family
of embedded rational curves of class $\beta$ in $X$ parametrized by a smooth surface $S^0_{\beta}$. The generic fiber of $p$ is isomorphic to $\p^1$
but over finitely many points we can have a rational nodal curve.
The insertion $\Gamma$ intersects the divisor $\CC^0_{\beta}$ in a curve that we can assume maps to a curve in $S^0_{\beta}$.
In particular, it avoids the singular fibers.
For simplicity we may hence assume that there are no nodal fibers, and that this is the only family of rational curves in class $\beta$.
We will compute the contribution of this family to the genus $1$ GW invariant \eqref{g1 invariant}.

Under these assumptions, for any genus $1$ degree $\beta$ stable map $f : C \to X$, the source curve splits canonically as $C \cong E \cup \p^1$, where $E$ is an elliptic curve
glued to $\p^1$ at one point $p$. The map $f$ is of degree 0 on $E$, and of degree $\beta$ on $\p^1$.
Hence
\[ \Mbar_{1,0}(X,\beta) = \Mbar_{0,1}(X,\beta) \times \Mbar_{1,1}. \]
By comparing the obstruction theories on the level of virtual classes, we  get
\begin{align*}
 [ \Mbar_{1,0}(X,\beta) ]^{\text{vir}} 
& =  \left[ \frac{c( \BE^{\vee} \otimes \ev_1^{\ast}(T_X))}{1 - \lambda_1 - \psi_1} \right]_{3} \cap \left( [ \Mbar_{1,1} ] \times [ \Mbar_{0,1}(X,\beta) ]^{\text{vir}} \right) \\
& =  \left(\psi_1^3 - \psi_1^2 \lambda_1 + \ev_1^{\ast}(c_2(X)) (\psi_1 - \lambda_1)\right) \left( [ \Mbar_{1,1} ] \times [ \Mbar_{0,1}(X,\beta) ]\right),
\end{align*}
where $[ - ]_{d}$ denotes taking the degree $d$ part, $\BE \to \Mbar_{g}$ is the Hodge bundle over the moduli space of curves
(having fiber $H^0(C,\omega_C)$ over a point $[C]$) with first Chern class $\lambda_1 = c_1(\BE) \in A^1(\Mbar_{1,1})$,
and $\psi_1$ is the usual psi class on the moduli space $\Mbar_{0,1}(X,\beta)$.
In the last line we have used that the dimension of $\Mbar_{0,k}(X,\beta)$ is equal to the expected dimension, so
\begin{equation} [ \Mbar_{0,k}(X,\beta)]^{\vir} = [\Mbar_{0,k}(X,\beta) ]. \label{vir equal fundamental class} \end{equation} 
Finally, as we will do often, we have suppressed pullback maps along the projection to the factors.

Consider the forgetful morphism $\pi \colon \Mbar_{1,1}(X,\beta) \to \Mbar_{1,0}(X,\beta)$
which at the same time is the universal curve over the moduli space.
In particular, we have a decomposition 
$$\Mbar_{1,1}(X,\beta) = \CE \cup \CP,$$ 
where $\CE \to \Mbar_{1,1}$ and $\CP \to \Mbar_{0,1}(X,\beta)$ are the universal curves. Since $\pi$ is flat of relative dimension $1$, we have
%\begin{multline*}
\[
[ \Mbar_{1,1}(X,\beta) ] =
\pi^{\ast} [ \Mbar_{1,0}(X,\beta) ]
=
a_{\ast}\left( [ \Mbar_{1,2} ] \times [ \Mbar_{0,1}(X,\beta) ] \right) + b_{\ast}\left( [ \Mbar_{1,1} ] \times [ \Mbar_{0,2}(X,\beta) ] \right),
\]
%\end{multline*}
where $a : \CE \to \Mbar_{1,1}(X,\beta)$ and $b : \CP \to \Mbar_{1,1}(X,\beta)$ are the natural inclusions.
We find
\begin{equation}
\label{efgwrh233}
\begin{aligned}
\blangle \tau_0(\gamma) \brangle^{\GW}_{1,\beta}
& = 
\int_{ [ \Mbar_{1,1} ] \times [ \Mbar_{0,2}(X,\beta) ]^{\text{vir}}} \ev_2^{\ast}(\gamma) \pi^{\ast}( \psi_1^3 - \psi_1^2 \lambda_1 + \ev_1^{\ast}(c_2(X)) (\psi_1-\lambda_1) ) \\
& +
\int_{ [\Mbar_{1,2} ] \times  [\Mbar_{0,1}(X,\beta) ]^{\text{vir}}} \ev_1^{\ast}(\gamma) \pi^{\ast}( \psi_1^3 - \psi_1^2 \lambda_1 + \ev_1^{\ast}(c_2(X)) (\psi_1-\lambda_1) ),
\end{aligned}
\end{equation}
where for the second summand %(where the marked point of $\Mbar_{1,1}(X,\beta)$ is on the elliptic curve) 
the $\gamma$ is pulled back along the evaluation map $\ev_1 : \Mbar_{0,1}(X,\beta) \to X$
(since the map is constant on the elliptic curve).
%we used that the evaluation map at the marked point
%is just the evaluation map at the meeting point.

In the second term in Eqn.~\eqref{efgwrh233}, the integrand over the factor $\Mbar_{1,2}$ is pulled back from $\Mbar_{1,1}$; hence this term vanishes.
We conclude that Eqn.~\eqref{efgwrh233} is equal to:
\begin{align*}
& \quad \,\, \int_{\Mbar_{1,1}} (-\lambda_1) \int_{ [\Mbar_{0,2}(X,\beta)]^{\text{vir}} } \ev_2^{\ast}(\gamma) \pi^{\ast}(\psi_1^2 + \ev_1^{\ast}(c_2(X))) \\
& = -\frac{1}{24} \Big( \blangle \tau_2(1) \tau_0(\gamma) \brangle^{\GW}_{0,\beta} + \blangle \tau_0(c_2(X))\tau_0(\gamma) \brangle^{\GW}_{0,\beta} - \blangle \tau_1(\gamma) \brangle^{\GW}_{0,\beta} \Big) \\
& = - \frac{1}{24} \blangle \tau_0(c_2(X)) \tau_0(\gamma) \brangle^{\GW}_{0,\beta}.
\end{align*} 
Here in the second step, we used that $\pi^{\ast}(\psi_1) = \psi_1 - s_{\ast}(1)$, where $s : \Mbar_{0,1}(X,\beta) \to \Mbar_{0,2}(X,\beta)$ is the section,
so that $\pi^{\ast}(\psi_1^2) = \psi_1^2 - s_{\ast}(\psi_1)$,
and in the last step we used Lemma~\ref{lemma:One relation}.
%As any primitive curve class can be deformed to an irreducible class, the above computation gives all contributions to genus 1 GW invariants when $\beta$ is primitive. 

\subsubsection{Conclusion}
By the discussion above we have obtained that in the ideal geometry we have
\[
\blangle \tau_0(\gamma) \brangle^{\GW}_{1,\beta}
=
n(\Gamma) - \frac{1}{24} \blangle \tau_0(c_2(X)) \tau_0(\gamma) \brangle^{\GW}_{0,\beta}.
\]
Since $n(\Gamma)$ is the Gopakuma-Vafa invariant $n_{1,\beta}(\gamma)$ in the ideal geometry, this ends the justification for both Definition~\ref{defn g1 GV primitive} and
integrality of genus $1$ invariants in Conjecture \ref{conj on integrality}.

\subsection{Justification:~GV in genus $2$}\label{sect on heuristic of g=2 GW/GV}
In the ideal geometry of Section~\ref{subsec:ideal geometry primitive},
the genus two Gopakumar-Vafa invariant $n_{2,\beta}$ should be the
(enumerative) number of genus $2$ curves in the irreducible curve class $\beta$.
We hence make the ansatz
\begin{equation} \label{genus 2 integral}
\blangle \varnothing \brangle^{\GW}_{2,\beta} = 
\int_{[ \Mbar_{2,0}(X,\beta) ]^{\vir}} 1
=
n_{2,\beta} + \cdots,
\end{equation}
where the dots stand for the contributions from curves of genus $\leqslant 1$.
In this section we derive an expression for these lower genus contributions.

\subsubsection{Contribution from genus one curves}
We consider first the contributions from a $1$-dimensional family of elliptic curves $\mathcal{C}^1_\beta \to S^1_\beta$ parametrized by a smooth curve $S^1_{\beta}$, 
but with the \textit{additional assumption} that there are no nodal rational curves in the family.

For simplicity of notation we also assume that the family $\mathcal{C}^1_{\beta}$ parametrizes all curves in class $\beta$ (so there are no rational or genus $2$ curves).
We compute the invariant $\blangle \varnothing \brangle^{\GW}_{2,\beta}$
in this geometry.

Under the above assumption we have the isomorphism
\[ \Mbar_{2,0}(X,\beta) = \Mbar_{1,1}(X,\beta) \times \Mbar_{1,1}, \]
and with an argument parallel to Section~\ref{subsubsec:Contributions genus zero to GV genus 1}, the virtual class is:
\[ [\Mbar_{2,0}(X,\beta)]^{\vir} =
\left[ \frac{ c(\BE^{\vee} \otimes \ev_1^{\ast}(T_X)) }{1-\lambda_1-\psi_1} \right]_{3} \cap [\Mbar_{1,1}(X,\beta)]^{\vir} \times [\Mbar_{1,1}], \]
where $\psi_1$ is the cotangent line class on $\Mbar_{1,1}(X,\beta)$
and $\lambda_1 \in H^2(\Mbar_{1,1})$, both pulled back to the product via the projection to the factors. One obtains that:
\begin{align*}
\blangle \varnothing \brangle^{\GW}_{2,\beta}
& = \int_{[ \Mbar_{1,1}(X,\beta) ]^{\text{vir}} \times [ \Mbar_{1,1} ] }
-\lambda_1 \psi_1^2 - \lambda_1 \ev_1^{\ast}(c_2(T_X)) \\
%(\psi_1^{2} + \ev_1^{\ast}(c_2(X))) (\psi_1 - \lambda_1) \\
& = - \frac{1}{24} \int_{[\Mbar_{1,1}(X,\beta)]^{\text{vir}}} \psi_1^2 + \ev_1^{\ast}(c_2(X)).
\end{align*}
By our assumption there are no family of rational curves in class $\beta$, so that 
we have  $\psi_1 = \tau^{\ast}(\psi_1)$,
where $\tau : \Mbar_{1,1}(X,\beta) \to \Mbar_{1,1}$ is the forgetful morphism to the moduli space of stable curves,
and therefore $\psi_1^2=0$.
We conclude that
\[ \blangle \varnothing \brangle^{\GW}_{2,\beta} = -\frac{1}{24} \blangle \tau_0(c_2(X)) \brangle^{\GW}_{1,\beta} = -\frac{1}{24} n_{1,\beta}(c_2(X)). \]
%= - \frac{1}{24} \int_{[\Mbar_{1,1}(X,\beta)]} \ev_1^{\ast}(c_2(X))
%We get that the first term becomes 
%\[ \int_{ [ \Mbar_{1,1}(X,\beta) ]^{\text{vir}}} \tau^{\ast}(\psi_1^2) = 0. \]
In total hence we see that the family
$\mathcal{C}^1_\beta \to S^1_\beta$
contributes
$-\frac{1}{24} n_{1,\beta}(c_2(X))$ to the integral \eqref{genus 2 integral}.

Assume more generally that there are both rational and elliptic curves in class $\beta$,
but still no nodal rational curves.
Then by the discussion in Section~\ref{sect on heuristic of g=1 GW/GV}
and the above computation we have that
$-\frac{1}{24} n_{1,\beta}(c_2(X))$
is precisely the contribution from the elliptic curves to \eqref{genus 2 integral}.
Hence this contribution remains valid also in the presence of rational curves.

\subsubsection{Contribution from genus zero curves}
Let $p : \CC^0_\beta \to S^0_\beta$ be a family
of degree $\beta$ embedded rational curves in $X$ parametrized
by a smooth surface $S^0_{\beta}$. 
We assume that there are no curves of genus $1$ or $2$,
and that all rational curves parametrized by $S^0_{\beta}$ are smooth.
Since $\beta$ is irreducible, this means that all of them are isomorphic to $\p^1$.

By our assumption, we have an isomorphism of moduli spaces:
\[
M:= \Mbar_{2,0}(X,\beta) \cong \Mbar_{2,0}(\CC^0_{\beta}/S^0_{\beta}, 1),
\]
where the right hand side is the moduli space of genus $2$ degree $1$ stable maps to the fibers of $\CC^0_\beta \to S^0_\beta$.
In particular, we have a diagram:
\[
\begin{tikzcd}
C \ar[swap]{dr}{\pi} \ar{r}{\rho} \ar[bend left]{rr}{\tilde{f}} & q^{\ast} \CC^0_{\beta} \ar{r}{\tilde{q}} \ar{d}{\tilde{p}} & \CC^0_{\beta} \ar{r}{f_{\beta}} \ar{d}{p} & X \\
& M \ar{r}{q} & S^0_{\beta},
\end{tikzcd}
\]
where $C \to M$ is the universal curve over the moduli space (for both
$\Mbar_{2,0}(X,\beta)$ and $\Mbar_{2,0}(\CC^0_{\beta}/S^0_{\beta}, 1)$),
$\tilde{f} : C \to \CC^0_{\beta}$ is the universal map of
$\Mbar_{2,0}(\CC^0_{\beta}/S^0_{\beta}, 1)$,
and $q$ is the structure morphism to the base.
By definition the middle square is fibered.
The moduli space $\Mbar_{2,0}(\CC^0_{\beta}/S^0_{\beta}, 1)$ carries naturally a virtual fundamental class which we denote by
\[ [ M ]^{\text{rel}} := [\Mbar_{2,0}(\CC^0_{\beta}/S^0_{\beta}, 1) ]^{\vir} \in A_{6}(M). \]
We also denote the reduced virtual fundamental class of $\Mbar_{2,0}(X,\beta)$ by
\[  [M]^{\vir} := [ \Mbar_{2,0}(X,\beta) ]^{\vir} \in A_{0}(M). \]
Since $f_{\beta}$ is fiberwise an embedding we have 
the subbundle $T_p \subset f_{\beta}^{\ast}(T_X)$. Let
\[ N = f_{\beta}^{\ast}(T_X)/T_p \]
be the quotient, which is locally free of rank $3$.
The key to our discussion is the following comparision of virtual fundamental classes.

\begin{prop} \label{prop:Comparision virtual class}
We have
\[ [ M ]^{\vir} = e\Big( \tilde{p}_{\ast}\big( (R^1 \rho_{\ast} \CO_C) \otimes \tilde{q}^{\ast}N \big) \Big) \cap [M ]^{\rel}. \]
\end{prop}

%Since $f_{\beta}$ is fiberwise an embedding we have 
%the subbundle $T_p \subset f_{\beta}^{\ast}(T_X)$
%giving rise to the short exact sequence
%\[ 0 \to T_{p} \to f_{\beta}^{\ast}(T_X) \to f_{\beta}^{\ast}(T_X)/T_p \to 0, \]
%with quotient $f_{\beta}^{\ast}(T_X)/T_p$ locally free of rank $3$.
%Since $(R^1 \rho_{\ast} \CO_C)$ is fiberwise torsion of length $2$,
%\[ \tilde{p}_{\ast}\big( (R^1 \rho_{\ast} \CO_C) \otimes \tilde{q}^{\ast} N \big) \]
%is a vector bundle of rank $6$. Hence the above proposition makes sense.

For the proof we start with the two basic lemmata:
\begin{lemma} \label{lemma:pi star N 1}
 We have
\[
\pi_{\ast}( \tilde{f}^{\ast}N) \cong q^{\ast}(T_{S^0_\beta}).
\]
\end{lemma}
\begin{proof}
By the Cohomology and Base Change Theorem we have
\[ \rho_{\ast}(\CO_C) = \CO_{q^{\ast}\CC^0_{\beta}}. \]
Hence we find that
\begin{align*}
\pi_{\ast} \tilde{f}^{\ast} N
& = \tilde{p}_{\ast} \rho_{\ast} \rho^{\ast} \tilde{q}^{\ast}N  \\
& = \tilde{p}_{\ast}( \rho_{\ast}(\CO_C) \otimes \tilde{q}^{\ast}N ) \\
& = \tilde{p}_{\ast} \tilde{q}^{\ast}N  \\
& = q^{\ast} p_{\ast} N \\
& = q^{\ast}T_{S^0_{\beta}},
\end{align*}
where in the second equality we used that $N$ is locally free,
and in the forth equality we used flat base change. For the last step we used that $S^0_{\beta} = \Mbar_{0,0}(X,\beta)$ is smooth with
tangent bundle given by $p_{\ast}N$
(which at each point $s \in S^0_\beta$ has fiber $H^0(\CC^0_{\beta,s}, N_{\CC^0_{\beta,s}/X})$).
\end{proof}

\begin{lemma} \label{lemma:pi star N 2}
We have the exact sequence:
\[
0 \to \tilde{p}_{\ast}\big( (R^1 \rho_{\ast} \CO_C) \otimes \tilde{q}^{\ast}N \big) \to 
R^1 \pi_{\ast}( \tilde{f}^{\ast}N)
\to R^1 \tilde{p}_{\ast}( \tilde{q}^{\ast} N)
\to 0
\]
and $R^1 \tilde{p}_{\ast}( \tilde{q}^{\ast} N) = \CO_{M}$.
\end{lemma}
\begin{proof}
The first statement is just an application of the Leray-Serre spectral sequence for the composition $\pi = \tilde{p} \circ \rho$.
For the second statement, we have by flat base change that:
\[ 
R^1 \tilde{p}_{\ast}( \tilde{q}^{\ast} N) \cong q^{\ast}( R^1 p_{\ast} N ).
\]
By the existence of a global cosection, we have a surjection
$R^1 p_{\ast} N \to \CO_{S^0_{\beta}}$.
Since $p_{\ast} N$ is locally free of rank $2$, $R^1 p_{\ast} N$ is locally free of rank $1$, so using the cosection it is isomorphic to $\CO_{S^0_{\beta}}$.
\end{proof}

\begin{proof}[Proof of Proposition~\ref{prop:Comparision virtual class}]
The `standard' virtual tangent bundle\footnote{If $E^{\bullet} \to \BL_M$ is a perfect obstruction theory,
then the associated virtual tangent bundle is $T^{\vir}_M := (E^{\bullet})^{\vee}$.} 
of $\Mbar_{2,0}(X,\beta)$ relative to the Artin stack of prestable curves $\mathfrak{M}_2$
is by definition given by
\[ T_{\Mbar_{2,0}(X,\beta)/\mathfrak{M}_2}^{\std} = R \pi_{\ast} f^{\ast}(T_X), \]
where $f = f_{\beta} \circ \tilde{f} : C \to X$ is the universal map.
The reduced virtual tangent bundle is defined to be the cone:
\[
T_{\Mbar_{2,0}(X,\beta)/\mathfrak{M}_2}^{\vir} = \left( R \pi_{\ast} f^{\ast}T_X \right)^{\red}
:= \mathrm{Cone}( R \pi_{\ast} f^{\ast}(T_X) \xrightarrow{\mathsf{sr}_{\sigma}} \CO[-1] )[-1],
\]
where $\mathsf{sr}_{\sigma}$ is the semi-regularity map associated to the symplectic form $\sigma$, see \cite{MP_GWNL, MPT}.
The inclusion $T_p \subset f_{\beta}^{\ast}(T_X)$ induces a natural distinguished triangle:
\begin{equation} \label{eq:A}
R \pi_{\ast} \tilde{f}^{\ast} T_p \to \left( R \pi_{\ast} f^{\ast}T_X \right)^{\red} \to \left( R \pi_{\ast} \tilde{f}^{\ast}N \right)^{\red}. \end{equation}
where the third term is defined as the cone of the first map.
By Lemma~\ref{lemma:pi star N 2} and since the restriction of $\mathsf{sr}_{\sigma}$ to $\tilde{p}_{\ast}\big( (R^1 \rho_{\ast} \CO_C) \otimes \tilde{q}^{\ast}N \big)$ vanishes,
%to $R^1 \tilde{p}_{\ast}( \tilde{q}^{\ast} N)$ remains surjective,
we have
\begin{equation} h^{1}\left( ( R \pi_{\ast} \tilde{f}^{\ast}N )^{\red} \right)  = \tilde{p}_{\ast}\big( (R^1 \rho_{\ast} \CO_C) \otimes \tilde{q}^{\ast}N \big). \label{h1 eqn} \end{equation}
Similarly, the virtual tangent bundle of the perfect obstruction theory of $\Mbar_{2,0}(\CC^0_{\beta}/S^0_{\beta})$
fits into the distinguished triangle
\begin{equation} \label{eq:B} R \pi_{\ast}( \tilde{f}^{\ast} T_p ) \to T^{\vir}_{\Mbar_{2,0}(\CC^0_{\beta}/S^0_{\beta},1)/\mathfrak{M}_2} \to q^{\ast}(T_{S^0_\beta}). \end{equation}
By Lemma~\ref{lemma:pi star N 1} there exists a natural morphism
\[ q^{\ast}(T_{S^0_\beta}) \to \left( R \pi_{\ast} \tilde{f}^{\ast}N \right)^{\red}, \]
which induces an isomorphism in degree $0$ cohomology.
This morphism induces a morphism from the complex \eqref{eq:B} to the complex \eqref{eq:A},
and combining with Eqn.~\eqref{h1 eqn}, we obtain the 
%Combining \eqref{eq:A} and \eqref{eq:B} and using this morphism, we hence obtain the 
distinguished triangle:
\[
T^{\vir}_{\Mbar_{2,0}(\CC^0_{\beta}/S^0_{\beta},1)/\mathfrak{M}_2} \to
T_{\Mbar_{2,0}(X,\beta)/\mathfrak{M}_2}^{\vir}
\to \tilde{p}_{\ast}\big( (R^1 \rho_{\ast} \CO_C) \otimes \tilde{q}^{\ast}N \big)[-1].
\]
The claim now follows from the excess intersection formula.
\end{proof}
The moduli space $M$ decomposes naturally as the union
\[ M = M_1 \cup (M_2 / \BZ_2), \]
where
\begin{gather*} M_1 = \Mbar_{2,1} \times \Mbar_{0,1}(\CC^0_{\beta}/S^0_{\beta},1), \\
M_2 = \Mbar_{1,1} \times \Mbar_{1,1} \times \Mbar_{0,2}(\CC^0_{\beta}/S^0_{\beta},1),
\end{gather*}
and $\BZ_2$ acts by interchanging the two factors of $\Mbar_{1,1}$ and switching the markings on $\Mbar_{0,2}(\CC^0_{\beta}/S^0_{\beta},1)$.
The class $[ M ]^{\rel}$ is of dimension $6$, but the dimensions of $M_1$ and $M_2$ are $7$ and $6$ respectively.
In particular, there exists some class $\alpha \in A_{6}(M_1)$  such that
\[ [ M]^{\rel} = \xi_{1 \ast}(\alpha) + \frac{1}{2} \xi_{2 \ast}[M_2],  \]
where $\xi_i : M_i \to M$ are the natural (gluing) morphisms.\footnote{
The na\"ive splitting of the virtual class $[ \Mbar_{2,0}(X,\beta)]^{\vir}$ as the sum
\[
[ \Mbar_{2,1} \times \Mbar_{0,1}(X,\beta) ]
\left[ \frac{ c( \BE^{\vee} \otimes \ev_1^{\ast}(T_X)) }{1 - \psi_1 - \psi_1'} \right]_{7}
+
\frac{1}{2} [\Mbar_{1,1}^2 \times \Mbar_{0,2}(X,\beta)]^{\text{vir}}
\left[ \frac{ c( \BE_a^{\vee} \otimes \ev_1^{\ast}(T_X)) }{1 - \lambda_{a,1} - \psi_1} 
 \frac{ c( \BE_b^{\vee} \otimes \ev_2^{\ast}(T_X)) }{1 - \lambda_{b,1} - \psi_2} \right]_{6}
\]
does {\em not} hold. The long detour to the relative virtual class $[M]^{\rel}$ is necessary to decompose the virtual class!
}
By Proposition~\ref{prop:Comparision virtual class}, we find that:
\begin{equation} \label{genus 2 GW invariant first step}
\begin{aligned}
\blangle \varnothing \brangle^{\GW}_{2,\beta} 
& = \int_{[ M ]^{\rel}} e\Big( \tilde{p}_{\ast}\big( (R^1 \rho_{\ast} \CO_C) \otimes \tilde{q}^{\ast}N \big) \Big) \\
& = \int_{\alpha} \xi_1^{\ast} e\Big( \tilde{p}_{\ast}\big( (R^1 \rho_{\ast} \CO_C) \otimes \tilde{q}^{\ast}N \big) \Big) \\
& \quad + \frac{1}{2} \int_{ \Mbar_{1,1} \times \Mbar_{1,1} \times \Mbar_{0,2}(\CC^0_{\beta}/S^0_{\beta},1)}
\xi_2^{\ast} e\Big( \tilde{p}_{\ast}\big( (R^1 \rho_{\ast} \CO_C) \otimes \tilde{q}^{\ast}N \big) \Big).
\end{aligned}
\end{equation}
These two terms are analyzed as follows:

\begin{lemma} \label{lemma:genus 2 eval 1} We have the vanishing
\[ e\Big( \xi_1^{\ast} \tilde{p}_{\ast}\big( (R^1 \rho_{\ast} \CO_C) \otimes \tilde{q}^{\ast}N \big) \Big) = 0. \]
\end{lemma}
\begin{proof}
Let $C \to M$ be the universal curve as before, and let $C' \to M_1$ be its pull back along $\xi_1 : M_1 \to M$.
There exists a natural decomposition $C' = R \cup_q Z$ where $R$ is the pullback of the universal curve over $\Mbar_{0,1}(X,\beta)$ and $Z$ is the pullback of the universal curve from $\Mbar_{2,1}$. The curves $R$ and $Z$ are glued along the marked points $v : M_1 \to C$.
In particular, we have the diagram
\[
\begin{tikzcd}
C' \ar{r} \ar{d}{\rho'} & C \ar{d}{\rho} \\
\xi_1^{\ast} q^{\ast} \CC^0_{\beta} \ar{r}{\tilde{\xi}_1} \ar{d}{\tilde{p}'} & q^{\ast} \CC^0_{\beta} \ar{d}{\tilde{p}} \\
M_1 \ar{r}{\xi_1} \ar[bend left=85]{uu}{v} \ar[bend left]{u}{x} & M,
\end{tikzcd}
\]
where $x=\rho' \circ v$ is the image of the gluing point.
Applying $\rho'_{\ast}$ to the normalization exact sequence
\[ 0 \to \CO_{C'} \to \CO_R \oplus \CO_Z \to \CO_{v} \to 0 \]
shows that
\[ \xi_1^{\ast} R^1 \rho_{\ast}(\CO_C) = R^1 \rho'_{\ast} \CO_C = R^1 \rho'_{\ast}( \CO_Z ) = x_{\ast}( \pr_1^{\ast} \BE^{\vee} ), \]
where $\pr_1 : M_1 \to \Mbar_{2,1}$ is the projection and %$x = \rho' \circ v$ is the image of the gluing point and 
$\BE \to \Mbar_{2,1}$ is the Hodge bundle (pulled back to the product).
We obtain that:
\[ \xi_1^{\ast} \tilde{p}_{\ast}\big( (R^1 \rho_{\ast} \CO_C) \otimes \tilde{q}^{\ast}N \big) \cong 
\tilde{\ev}_1^{\ast}(N) \otimes \mathrm{pr}_1^{\ast}(\BE^{\vee}), \]
where $\tilde{\ev}_1 = \tilde{q} \circ \tilde{\xi}_1 \circ x : M_1 \to \CC^0_{\beta}$ is the evaluation map.

Using the defining exact sequence $0 \to T_p \to f_{\beta}^{\ast}(T_X) \to N \to 0$ and that $\tilde{\ev}_1^{\ast}(T_p)$ is isomorpic to the
cotangent line bundle of $\Mbar_{0,1}(\CC^0_{\beta}/S^0_{\beta},1)$ at the marking, i.e. $\tilde{\ev}_1^{\ast}(T_p) \cong \BL_{p_1}^{\vee}$,
we obtain the exact sequence
\[ 0 \to \BE^{\vee} \otimes \BL_{p_1}^{\vee} \to \BE^{\vee} \otimes \ev_1^{\ast}(T_X) \to \tilde{\ev}_1^{\ast}(N) \otimes \mathrm{pr}_1^{\ast}(\BE^{\vee}) \to 0, \]
where $\ev_1 : \Mbar_{0,1}(\CC^0_{\beta}/S^0_{\beta},1) \cong \Mbar_{0,1}(X,\beta) \to X$ is the evaluation map to $X$
and we surpressed the pullbacks by the projection to the factors.
We conclude that
\begin{align*} 
e\Big( \xi_1^{\ast} \tilde{p}_{\ast}\big( (R^1 \rho_{\ast} \CO_C) \otimes \tilde{q}^{\ast}N \big) \Big) & = 
\left[ \frac{c( \BE^{\vee} \otimes \ev_1^{\ast}(T_X) )}{ c( \BE^{\vee} \otimes \BL_{p_1}^{\vee} ) } \right]_{6} \\
& =
\left[ \frac{ c( \BE^{\vee})^{4} + 2 \ev_1^{\ast}(c_2(T_X)) c(\BE^{\vee}) (1-\lambda_1) }{ (1-\psi_1)^2 - \lambda_1 (1 - \psi_1) + \lambda_2 } \right]_{6},
\end{align*}
where in the second equality we used the splitting principle
and the Mumford relation
\begin{equation} c(\BE)\,c(\BE^{\vee}) = (1 + \lambda_1 + \lambda_2) (1 - \lambda_1 + \lambda_2) = 1 + 2 \lambda_2 - \lambda_1^2 + \lambda_2^2 = 1. \label{eq:Mumford Relation} \end{equation}
Now a straightforward computation (using that $\Mbar_{0,1}(X,\beta)$ is of dimension $3$ and the Mumford relation, and which may be performed by a computer program)
shows that this degree $6$ component vanishes.
%\begin{lemma} We have
%\[ c( \BE^{\vee} \otimes \ev_1^{\ast}(T_X)) = 
%(1-\lambda_1 + \lambda_2)^2 \Big( (1-\lambda_1 + \lambda_2)^2 + 2 c_2(T_X) (1 - \lambda_1) \Big). \]
%\end{lemma}
%\begin{proof}
%Let $c_i = \ev^{\ast} c_i(T_X)$. And note that $c_4=c_2^2=0$.
%We assume $\BE^{\vee} = L_a \oplus L_b$ where $c_1(L_a) = 1+a$ and $c_1(L_b) = 1+b$. Then
%\begin{align*}
%c(\BE^{\vee} \otimes T_X)
%& = ((1+a)^4 + c_2 (1+a)^2) ( (1+b)^4 + c_2 (1+b)^2 ) \\
%& = c(\BE^{\vee})^4 + c_2\,c(\BE^{\vee})^2 ( (1+a)^2 + (1+b)^2 ) \\
%& = c(\BE^{\vee})^4 + c_2\, c(\BE^{\vee})^2 ( 2 + 2 (a+b) + (a+b)^2 - 2ab ) \\
%& = c(\BE^{\vee})^4 + c_2\, c(\BE^{\vee})^2 ( 2 - 2 \lambda_1 + \lambda_1^2 - 2 \lambda_2 ) \\
%& = c(\BE^{\vee})^2 ( c(\BE^{\vee})^2 + 2 c_2 (1-\lambda_1)),
%\end{align*}
%where we used the Mumford relation
%\[ c(\BE) c(\BE^{\vee}) = (1 + \lambda_1 + \lambda_2) (1 - \lambda_1 + \lambda_2) = 1 + 2 \lambda_2 - \lambda_1^2 + \lambda_2^2 = 1. \qedhere \]
%\end{proof}
%
%
%The computer code reads for SageMath:
%S = PolynomialRing(QQ, 'l1, l2, c2, psi')
%l1,l2, c2, psi = S.gens()
%rels1=[ l1**4, 2*l2-l1**2]
%rels2=[ psi**4, c2*psi**2, c2**2 ]
%I1=S.ideal( rels1 + rels2 )
%I=S.ideal( I1.groebner_basis() )
%Sq=QuotientRing(S, I)
%
%l1, l2, c2, psi = Sq.gens()
%a=(1 -l1 + l2)**4 + 2*c2*(1-l1+l2)**2*(1-l1)
%b=sum( (-1)**k*( -2*psi+psi**2 - l1*(1-psi) + l2 )**k for k in range(0,10 ))
%
%f=a*b
%fl=f.lift()
%l1,l2, c2, psi = S.gens()
%
%print("flift= ", fl)
%C1=fl.coefficient({ c2:1, psi:1})
%C2=fl.coefficient({ c2:0, psi:3})
%
%print(C1)
%print(C2)
\end{proof}
\begin{lemma}
\[ \int_{ \Mbar_{1,1} \times \Mbar_{1,1} \times \Mbar_{0,2}(\CC^0_{\beta}/S^0_{\beta},1)}
\xi_2^{\ast} e\Big( \tilde{p}_{\ast}\big( (R^1 \rho_{\ast} \CO_C) \otimes \tilde{q}^{\ast}N \big) \Big) =
\frac{1}{24^2} \blangle \tau_0(c_2(X)) \tau_0(c_2(X)) \brangle^{\GW}_{0,\beta}. \]
\end{lemma}
\begin{proof}
Let $C' \to M_2$ be the pullback of the universal curve $C \to M$ to $M_2$.
We have a decomposition $C' = R \cup E_1 \cup E_2$, where $R$ is the universal $2$-pointed genus $0$ curve, and the $E_i$ are the universal genus $1$ curves.
Let
\[ x_1, x_2 : M_2 \to \xi_2^{\ast} q^{\ast} \CC^0_{\beta} \]
be the image of the marked points under the evaluation map $\rho' : C' \to \xi_2^{\ast} q^{\ast} \CC^0_{\beta}$.
We have
\[ \xi_2^{\ast} R^1 \rho_{\ast} \CO_C = R^1 \rho'_{\ast}(\CO_{C'}) = x_{1 \ast}( \BE_1^{\vee} ) \oplus x_{2 \ast}(\BE_2^{\vee}), \]
where $\BE_i = \pr_i^{\ast}(\BE)$ are the Hodge bundles pulled-back from the first or second copy of $M_2$.
We argue as in Lemma~\ref{lemma:genus 2 eval 1}, that is first we have
\[ \tilde{\ev}_i = \tilde{q} \circ \tilde{\xi}_2 \circ x_i. \]
Then with $\pi_i : \Mbar_{0,2}(\CC^0_{\beta}/S^0_{\beta},1) \to \Mbar_{0,1}(\CC^0_{\beta}/S^0_{\beta},1)$
the morphism that forgets all but the $i$-th marking we have that
\[ \tilde{\ev}_i^{\ast}(T_p) = \pi_i^{\ast} \tilde{\ev}^{\ast}(T_p) = \pi_i^{\ast}( \BL_{p_i}^{\vee} ). \]
(Here, we need the precompose with the forgetful morphism because the two markings can lie on a bubble
in which case the tangent space to a marking maps with zero to the tangent space of the image point; by precomposing with the forgetful map, 
we contract the bubbles).
As in Lemma~\ref{lemma:genus 2 eval 1} we then obtain that
\begin{equation} \label{eq:eval}
\begin{aligned}
& \quad \, \, \int_{ \Mbar_{1,1} \times \Mbar_{1,1} \times \Mbar_{0,2}(\CC^0_{\beta}/S^0_{\beta},1)}
\xi_2^{\ast} e\Big( \tilde{p}_{\ast}\big( (R^1 \rho_{\ast} \CO_C) \otimes \tilde{q}^{\ast}N \big) \Big) \\
& = \int_{\Mbar_{1,1} \times \Mbar_{1,1} \times \Mbar_{0,2}(\CC^0_{\beta}/S^0_{\beta},1)}
\frac{ c( \BE_1^{\vee} \otimes \ev_1^{\ast}(T_X) ) c( \BE_2^{\vee} \otimes \ev_1^{\ast}(T_X) ) }{ (1 - c_1(\BE_1) - \pi_1^{\ast}(\psi_1)) (1-c_1(\BE_1) - \pi_2^{\ast}(\psi_2)) }.
\end{aligned}
\end{equation}
For $i=1,2$ and $(\lambda, \psi, \ev) := (c_1(\BE_i), \pi_i^{\ast}(\psi_i), \ev_i)$, we have
\begin{align*}
& \quad\, \frac{ c( \BE^{\vee} \otimes \ev^{\ast}(T_X)) }{1 - \lambda - \psi} \\
& = ((1-\lambda)^4 + \ev^{\ast}(c_2(X)) (1-\lambda)^2) (1 + \lambda + \psi + 2 \lambda \psi + \psi^2 + 3 \lambda \psi^2 + \psi_1^3 + 4 \lambda \psi^3) \\
& =
\lambda \Big( (1+ \ev^{\ast}c_2(X)) ( 1 + 2 \psi + 3 \psi^2 + 4 \psi^3 )
+ (-4 - 2 \ev^{\ast} c_2(X))(1 + \psi + \psi^2 + \psi^3 ) \Big)
+ (\cdots) \\
& = \lambda \Big( -3 - 2 \psi - \psi^2 - \ev^{\ast}(c_2(X)) \Big) + (\cdots),
\end{align*}
where $(\cdots)$ are terms that are not multiples of $\lambda$.

Using this and Eqn.~\eqref{vir equal fundamental class}, the term \eqref{eq:eval} becomes:
\begin{align} \label{eq:evesdfs}
\frac{1}{24^2} \int_{[ \Mbar_{0,2}(X,\beta) ]^{\vir}} ( \ev_1^{\ast}(c_2(X)) + \pi_1^{\ast}(\psi_1)^2  )( \ev_2^{\ast}(c_2(X)) + \pi_2^{\ast}(\psi_2)^2 ).
\end{align}
On $\Mbar_{0,2}(X,\beta)$ we have
\[ \psi_1 = \pi_1^{\ast}(\psi_1) + s_{\ast}(1), \]
where $s : \Mbar_{0,1}(X,\beta) \to \Mbar_{0,2}(X,\beta)$ is the canonical section,
and therefore
\[ \pi_1^{\ast}(\psi_1)^2 = \psi_1^2 - s_{\ast}(\psi_1). \]
Applying Lemma~\ref{lemma:One relation}, we find that:
\[
\int_{[ \Mbar_{0,2}(X,\beta) ]^{\vir}} \pi_1^{\ast}(\psi_1)^2 \ev_2^{\ast}(c_2(X)) = \blangle \tau_2(1) \tau_0(c_2(X)) \brangle^{\GW}_{0,\beta} - \blangle \tau_1(c_2(X)) \brangle^{\GW}_{0,\beta}
= 0.
\]
With a similar reasoning, using Lemma~\ref{lemma:Relation 2}, we also get that:
\[
\int_{[ \Mbar_{0,2}(X,\beta) ]^{\vir}}\pi_1^{\ast}(\psi_1)^2 \pi_2^{\ast}(\psi_2)^2 = \blangle \tau_2(1) \tau_2(1) \brangle^{\GW}_{0,\beta} - \blangle \tau_3(1) \brangle^{\GW}_{0,\beta} = 0.
\]
Inserting both these vanishings into Eqn.~\eqref{eq:evesdfs} concludes the claim.
\end{proof}

Inserting the two lemmata above into Eqn.~\eqref{genus 2 GW invariant first step}, the whole computation collpases into the following simple evaluation:
\[
\blangle \varnothing \brangle^{\GW}_{2,\beta}  = 
\frac{1}{2 \cdot 24^2} \blangle \tau_0(c_2(X)) \tau_0(c_2(X)) \brangle^{\GW}_{0,\beta}. 
\]
We hence conclude that the family
$\CC^0_\beta \to S^0_\beta$ of rational curves with only smooth fibers
contributes 
$\blangle \tau_0(c_2(X)) \tau_0(c_2(X)) \brangle^{\GW}_{0,\beta} / (2 \cdot 24^2)$ to the Gopakumar-Vafa invariant $n_{2,\beta}$.

\subsubsection{Conclusion and contribution from nodal rational curves}\label{sect las ste}
Consider an ideal geometry of curves as in Section~\ref{subsec:ideal geometry primitive} without any additional assumptions.
We expect the contributions from genus $0$ and genus $1$ curves
to the invariant
$\blangle \varnothing \brangle^{\GW}_{2,\beta}$ to be as discussed above,
plus a correction term coming from the nodal rational curves.
This correction term should be local, and hence a multiple of the expected number of nodal rational curves $N_{\text{nodal},\beta}$.
We hence make the ansatz:
\begin{equation} \label{ansatz}
\blangle \varnothing \brangle^{\GW}_{2,\beta} = 
n_{2,\beta} -\frac{1}{24} n_{1}(c_2(X))
+ \frac{1}{2 \cdot 24^2} \blangle \tau_0(c_2(X)) \tau_0(c_2(X)) \brangle^{\GW}_{0,\beta} + a N_{\text{nodal},\beta}
\end{equation}
for a constant $a \in \BQ$ independent of $(X,\beta)$.

We determine now $a$ with a test calculation.
Let $X$ be the Fano variety of lines on a very general cubic 4-fold, and let $\beta \in H_2(X,\BZ)$ be the minimal effective curve class.
As we will see in Section~\ref{sec:Hilb K3} we have the evaluations (assuming the conjectural holomorphic anomaly equation):
\begin{gather*}
\blangle \varnothing \brangle^{\GW}_{2,\beta} =  - 11445/128, \\
n_{1,\beta}(c_2(X)) = 5985, \\
\blangle \tau_0(c_2(X)) \tau_0(c_2(X)) \brangle^{\GW}_{0,\beta} = 2835.
\end{gather*}
Moreover, by \cite[Thm.~1.3]{NO}, we have
\[ N_{\text{nodal}, \beta} = 3780. \]
Since there are no genus $2$ curves on $X$ in class $\beta$ (see \cite{NO}) we set
\[ n_{2,\beta} = 0. \]
Inserting this into Eqn.~\eqref{ansatz} yields:
\[ a = \frac{1}{24}. \]
This conclude the justification of Definition~\ref{defn g2 GV primitive}.
While the last step (i.e.~\S \ref{sect las ste}) requires two assumption (locality of the contribution of nodal rational curves, and the holomorphic anomaly equation),
the remainder of the paper yields plenty of numerical support for this definition.

\section{Donaldson-Thomas invariants}
For a holomorphic symplectic $4$-fold, we
define (reduced) Donaldson-Thomas invariants ($\DT_4$ invariants for short) of one dimensional stable sheaves.
We then use them to give a sheaf theoretic approach to Gopakumar-Vafa invariants defined in the previous section.
In the last section we justify the definition
by computations in the ideal geometry of curves.

\subsection{Definitions}
Let $M_\beta$ be the moduli scheme of one dimensional stable sheaves $F$ on $X$ with $[F]=\beta$, $\chi(F)=1$. Such moduli spaces 
are independent of the choice of polarization (e.g.~\cite[Rmk.~1.2]{CMT1}) and are used in \cite{CMT1, CT2} 
to give sheaf theoretic interpretation of Gopakumar-Vafa type invariants 
of ordinary Calabi-Yau 4-folds \cite{KP}. 
We also refer to \cite{CMT2, CT1, CT3, CT4} for related conjectures and computations,
%\begin{GO} Why is it necessary to cite all of these, e.g. those on zero-dimensional DT or the DT/PT correspondence? Do we need it anywhere? \end{GO}
which build on the works of virtual class constructions \cite{BJ, OT} (see also \cite{CL1}). 

Parallel to Gromov-Witten theory, the ordinary virtual class of $M_\beta$ vanishes \cite{KiP, Sav}. For a choice of ample divisor $H$, one can define 
a reduced virtual class due to Kiem-Park \cite[Def.~8.7, Lem.~9.4]{KiP}: 
\begin{align}\label{red vir class}[M_\beta]^{\vir}\in A_2(M_\beta,\mathbb{Q}), \end{align}
depending on the choice of orientation \cite{CGJ, CL2}. 
To define descendent invariants, we need insertions:  
\begin{align*}\tau_i: H^{*}(X,\mathbb{Z})\to H^{*+2i-2}(M_\beta,\mathbb{Q}), \end{align*}
\begin{align*}\tau_i(\bullet):=(\pi_{M})_{\ast}\left(\pi_X^{\ast}(\bullet) \cup\ch_{3+i}(\mathbb{F}_{\mathrm{norm}})\right),
\end{align*}
where $\mathbb{F}_{\mathrm{norm}}$ is the normalized universal sheaf, i.e. $
\det(\pi_{M*}\mathbb{F}_{\mathrm{norm}})\cong \oO_{M_\beta}$ (ref.~\cite[\S1.4]{CT2}).

\begin{defi}\label{def DT4 inv}
For any $\gamma_1, \ldots, \gamma_n \in H^{\ast}(X)$ and $k_i \in \mathbb{Z}_{\geqslant   0}$ the $\DT_4$ invariants are  defined by
\begin{align}\label{def of all DT4 invs}
\blangle\tau_{k_1}(\gamma_1),\ldots,\tau_{k_n}(\gamma_n) \brangle^{\DT_4}_{\beta}:=\int_{[M_{\beta}]^{\rm{vir}}}\prod_{i=1}^n\tau_{k_i}(\gamma_i)\in\mathbb{Q}.
\end{align}
\end{defi}
%Our main focus are the following primary/descendent/higher descendent $\DT_4$ invariants:    
%\begin{align*}  
%\blangle\tau_0(\gamma_1),\ldots,\tau_0(\gamma_n) \brangle^{\DT_4}_{\beta}&:=\int_{[M_{\beta}]^{\rm{vir}}}\prod_{i=1}^n\tau_0(\gamma_i), \,\,\, 
%\mathrm{where}\,\, \gamma_i \in H^{m_i}(X, \mathbb{Z}), \\ 
%\blangle\tau_1(\gamma) \brangle^{\DT_4}_{\beta}&:=\int_{[M_{\beta}]^{\rm{vir}}}\tau_1(\gamma),  \,\,\,
%\mathrm{where}\,\, \gamma\in H^4(X, \mathbb{Z}), \\
%\blangle\tau_2(\theta) \brangle^{\DT_4}_{\beta}&:=\int_{[M_{\beta}]^{\rm{vir}}}\tau_2(\theta),  \,\,\, \mathrm{where}\,\, \theta\in H^2(X, \mathbb{Z}), \\
%\blangle\tau_3(1) \brangle^{\DT_4}_{\beta}&:=\int_{[M_{\beta}]^{\rm{vir}}}\tau_3(1). \end{align*}

\subsection{Conjectures}
As in \cite{CMT1, CT2}, we propose the following sheaf theoretic interpretation of
all genus Gopakumar-Vafa invariants: 
\begin{conj}\label{conj on DT4/GV}
For certain choice of orientation, the following equalities hold. \\
When $\beta$ is an effective curve class, 
\begin{align}\label{g=0 part} \tag{i}
\blangle\tau_0(\gamma_1),\ldots,\tau_0(\gamma_n) \brangle^{\DT_4}_{\beta}=n_{0,\beta}(\gamma_1,\ldots,\gamma_n). \end{align}
When $\beta$ is a primitive curve class, 
\begin{align}\label{g=1 part} \tag{ii}
\blangle\tau_1(\gamma) \brangle^{\DT_4}_{\beta}=-\frac{1}{2}\blangle \tau_1(\gamma) \brangle^{\GW}_{0,\beta}-n_{1, \beta}(\gamma).
\end{align}
When $\beta$ is a primitive curve class, 
\begin{align}\label{g=2 part} \tag{iii}
-\blangle\tau_3(1) \brangle^{\DT_4}_{\beta}
-\frac{1}{12}\blangle\tau_1(c_2(X)) \brangle^{\DT_4}_{\beta}=n_{2, \beta}.
\end{align}
\end{conj}
By Proposition \ref{div equ on GW}, $\blangle\tau_1(\gamma) \brangle^{\GW}_{0,\beta}$ can be deduced by $g=0$ primary Gromov-Witten invariants. 
Therefore these formulae determine 
%$g=0$ (and $g=1, 2$ for primitive classes) 
all genus Gopakumar-Vafa invariants from primary and descendent $\DT_4$ invariants, which give a sheaf theoretic interpretation for them.
\begin{rmk}\label{rmk on dt4/pt}
The way we write down Conjecture \ref{conj on DT4/GV} \eqref{g=2 part} is indirect. By \cite[App.~A]{COT1}, 
the LHS of \eqref{g=2 part} is equal to stable pair invariant $P_{-1,\beta}$ which is conjecturally the same as genus 2 Gopakumar-Vafa invariants \cite[Conj.~1.10]{COT1}.
We believe there is also a formula relating $\blangle\tau_2(\theta) \brangle^{\DT_4}_{\beta}$ to genus 2 Gopakumar-Vafa invariants, which we haven't found so far.
\end{rmk}
\begin{rmk}
Our conjecture implicitly includes the independence of $\DT_4$ invariants on the choice of ample divisor in defining 
reduced virtual classes \eqref{red vir class}.  
\end{rmk}

\subsection{Justification: Primary $\DT_4$ invariants}\label{sect on ideal DT4 comp0}

For Conjecture \ref{conj on DT4/GV} \eqref{g=0 part}, we consider the case $\gamma_1,\gamma_2\in H^4(X,\mathbb{Z})$ for simplicity.
These two 4-cycles (generically) cut out finite number of rational curves and miss high genus curves. 

As in \cite[\S1.4]{CMT1}, any one dimensional stable sheaf $F$ with $[F]=\beta$ is $\oO_C$ for some rational curve $C$. Their moduli space $M_\beta$
is identified with the moduli space $S^0_{\beta}$ of rational curves and 
\begin{align}\label{vir class in ideal g=0}[M_\beta]^{\vir}=[S^0_\beta],  \end{align}
for some choice of orientation. After imposing the primary insertion, we have 
\begin{align*}\int_{[M_\beta]^{\vir}}\tau_0(\gamma_1)\,\tau_0(\gamma_2)=\int_{S^0_\beta}p_*(f^*\gamma_1)\cdot p_*(f^*\gamma_2),  \end{align*}
where $p:\mathcal{C}^0_{\beta}\to S^0_{\beta}$ is the total space of rational curve family (RCF) of class $\beta$ and $f: \mathcal{C}^0_{\beta}\to X$ is the evaluation map. 
Therefore Conjecture \ref{conj on DT4/GV} \eqref{g=0 part} is confirmed in this ideal setting as both sides of the equation are (virtually) enumerating rational curves of class $\beta$ incident to cycles dual to $\gamma_1$ and $\gamma_2$.

\subsection{Justification: Descendent $\DT_4$ invariants}\label{sect on ideal DT4 comp}

For Conjecture \ref{conj on DT4/GV} \eqref{g=1 part}, as we put the incident condition with one 4-cycle $\gamma$ in $\blangle\tau_1(\gamma) \brangle^{\DT_4}_{\beta}$ which generically does not intersect genus 2 curves, so we only need to consider the contributions from RCF and ECF (elliptic curve family).

(1) For any RCF of class $\beta$, we have an embedding $i:\mathcal{C}^0_\beta\hookrightarrow S^0_{\beta}\times X$ fitting into the diagram: 
\begin{align}\label{RCF diagr}
\xymatrix{
\mathcal{C}^0_\beta  \ar[r]^{i\quad} \ar[dr]_{p} \ar@/_32pt/[drr]_{f}   & S^0_{\beta}\times X \ar[d]_{\pi_S} \ar[dr]^{\pi_X} & \\
  & S^0_{\beta}  & X. }
 \end{align}
By Grothendieck-Riemann-Roch (GRR) formula, we have 
\begin{align}\label{GRR equ}\ch(i_*\oO_{\mathcal{C}^0_\beta})=i_*(\td^{-1}(N_{\mathcal{C}^0_\beta/S^0_{\beta}\times X})). \end{align}
Obviously $\mathbb{F}_{\mathrm{norm}}=\oO_{\mathcal{C}^0_\beta}$, and therefore   
\begin{align*}\tau_1(\gamma)&=\pi_{S*}(\ch_4(\oO_{\mathcal{C}^0_\beta})\cdot \pi_X^*\gamma) \\
&=-\frac{1}{2}\pi_{S*}(i_*c_1(\omega_p)\cdot \pi_X^*\gamma) \\
&=-\frac{1}{2}\pi_{S*}(i_*(c_1(\omega_p)\cdot f^*\gamma)) \\ 
&=-\frac{1}{2}p_*(c_1(\omega_p)\cdot f^*\gamma), \end{align*}
where $\omega_p$ is the relative cotangent bundle of $p$. 

Combining with Eqn.~\eqref{vir class in ideal g=0}, we see RCF in class $\beta$ contributes to $\blangle\tau_1(\gamma) \brangle^{\DT_4}_{\beta}$ by 
\begin{align}\label{cont to tau1 from RCF}\int_{[M_\beta]^{\vir}}\tau_1(\gamma)=-\frac{1}{2}\int_{S^0_{\beta}}p_*(c_1(\omega_p)\cdot f^*\gamma). \end{align}
%When RCF are embedded, from the computation of genus 0 descendent GW invariants \eqref{equ on des g0 gw ideal}, the above \eqref{cont to tau1 from RCF}
%equals to the contribution of this RCF to $-\frac{1}{2}n_{0, \beta}(\gamma;\psi)$ (see Definition \ref{def of g=0 m inv}), 
%therefore justifying the first term in RHS of \eqref{equ on impri DT4/GV}. 
%\begin{YC}We need the old computations of multiple cover formula of g=0 descendent GW on embedded RCF to conclude \eqref{cont to tau1 from RCF}. \end{YC}
As $\beta$ is primitive, we may deform it to the irreducible case where RCF consists of smooth rational curves (except at some finite number of fibers of nodal curves 
which can be ignored by insertion $\gamma\in H^4(X)$). By Lemma \ref{lemma:psi1 in terms of Tp},
%As genus 0 descendent GW invariants can be written as genus 0 primary GW invariants (ref.~Lemma~\ref{div equ on GW}) 
%whose insertion can avoid the place where $j : \CC^0_{\beta} \to X$ is not necessarily an immersion, 
the RHS of Eqn.~\eqref{cont to tau1 from RCF} is equal to $-\frac{1}{2}\blangle \tau_1(\gamma) \brangle^{\GW}_{0,\beta}$.
This justifies the first term in the RHS of Conjecture \ref{conj on DT4/GV} \eqref{g=1 part}.

(2) Next we consider the contribution from ECF. Let 
$p:\mathcal{C}^1_{\beta}\to S^1_{\beta}$ be the total space of ECF of class $\beta$ and $j: \mathcal{C}^1_{\beta}\to X$ be the evaluation map.
The insertion $\gamma\in H^4(X)$ (generically) intersects $\mathcal{C}^1_{\beta}$ in a finite number of points.
We may assume $\mathcal{C}^1_{\beta}=E\times S^1_{\beta}$ is the product, $p$ is the projection and $j$ is an embedding 
in our computations. 
We further assume $E$ is smooth with normal bundle  
$L\oplus L^{-1}\oplus \oO$ for a generic degree zero line bundle $L$ on $E$.

\begin{lem}
Let $p:\mathcal{C}^1_{\beta}\to S^1_{\beta}$ be a one dimensional family of smooth elliptic curves $E$ on $X$ with normal bundle $N_{E/X}=L\oplus L^{-1}\oplus \oO$ for a generic $L\in \Pic^0(E)$. Then
any one dimensional stable sheaf $F$ supported on this family is scheme theoretically supported on a fiber of $p$.
\end{lem}
\begin{proof}
By \cite[Lem.~2.2]{CMT1}, we know $F$ is scheme theoretically supported on $\mathrm{Tot}_E(L\oplus L^{-1})$ for a fiber $E$ of $p$. 
By \cite[Prop.~4.4]{HST}, $F$ is scheme theoretically supported on the its zero section, so we are done.
\end{proof}
By the above lemma, there exists a morphism 
\begin{align}\label{map from M to S1}M_{\beta}\to S^1_{\beta}, \end{align}
whose fiber over $\{E\}$ is the moduli space $M_{1,1}(E)$ of stable bundles on $E$ with rank $1$ and $\chi=1$.
%Note here we have a determinant map$$\det: M_{k,1}(E)\stackrel{\cong}{\to} E. $$
Note that $M_{1,1}(E)\cong E$. 
A family version of such isomorphism gives 
\begin{align}\label{equ on moduli from beta/k}M_{\beta}\cong \mathcal{C}^1_{\beta}, \end{align}
such that the virtual class satisfies
\begin{align}\label{vir class in ideal g=1}[M_\beta]^{\vir}=[\mathcal{C}^1_{\beta}],  \end{align}
for certain choice of orientation. 

\iffalse
By this lemma, for any $k\in \mathbb{Z}_{\geqslant  1}$ such that $k|\beta$, on ECF of class $\beta/k$, there exists a morphism 
\begin{align}\label{map from M to S1}M_{\beta}\to S^1_{\beta/k}, \end{align}
whose fiber over $\{E\}$ is the moduli space $M_{k,1}(E)$ of stable bundles on $E$ with rank $k$ and $\chi=1$.
Note here we have a determinant map
$$\det: M_{k,1}(E)\stackrel{\cong}{\to} E. $$
A family version of such isomorphisms gives 
\begin{align}\label{equ on moduli from beta/k}M_{\beta}\cong \mathcal{C}^1_{\beta/k}, \end{align}
such that the virtual class satisfies
\begin{align}\label{vir class in ideal g=1}[M_\beta]^{\vir}=[\mathcal{C}^1_{\beta/k}],  \end{align}
for certain choice of orientation. 
\fi
Next we compute the descendent insertion. In the following diagram:
% \oO_{\Delta_{S^1_{\beta}}}  \ar[d]_{}
\begin{align*} 
\xymatrix{   \mathbb{F} \ar[d]_{ }    & \mathbb{E} \ar[d]_{ }  &  \\
\mathcal{C}^1_{\beta}\times X & \mathcal{C}^1_{\beta}\times \mathcal{C}^1_{\beta} \ar[r]^{\bar{p}=(p,p)} \ar[l]_{\bar{j}=(\id,j)}  & S^1_{\beta}\times S^1_{\beta}, }\end{align*}
a universal one dimensional sheaf $\mathbb{F}$ can be chosen as 
$$\mathbb{F}=\bar{j}_*\mathbb{E}, \quad \mathbb{E}:=\oO_{\bar{p}^*(\Delta_{S^1_{\beta}})}(\Delta_{\mathcal{C}^1_{\beta}}), $$
where we treat $\Delta_{\mathcal{C}^1_{\beta}}$ as a divisor of $\bar{p}^*(\Delta_{S^1_{\beta}})$ via 
$$\Delta_{\mathcal{C}^1_{\beta}}=\big\{(x,x)\,|\,x\in \mathcal{C}^1_{\beta}\big\}\hookrightarrow \big\{(x,p^{-1}p(x))\,|\,x\in \mathcal{C}^1_{\beta}\big\}=\bar{p}^*(\Delta_{S^1_{\beta}}). $$ 
It is straightforward to check that $\mathbb{F}$ is normalized. 

Below, we use notations from the following diagram  
\begin{align*} 
\xymatrix{      & \mathcal{C}^1_{\beta}\times X \ar[d]_{\pi_X} \ar[dl]_{\pi_C}  & \mathcal{C}^1_{\beta}\times\mathcal{C}^1_{\beta}  \ar[d]_{\pi_2} \ar[dr]^{\pi_1}   
\ar[l]_{\bar{j}=(\id,j)} & \\
\mathcal{C}^1_{\beta}  & X     &  \mathcal{C}^1_{\beta} \ar[l]_{j} & \mathcal{C}^1_{\beta}. } \end{align*}
The GRR formula gives 
\begin{align}\label{equ on ch4}
\ch_4(\bar{j}_*\mathbb{E})=
\bar{j}_*\left(\frac{1}{2}\ch_1(\mathbb{E})\cdot\pi_2^*c_1(\mathcal{C}^1_{\beta})+\ch_2(\mathbb{E})\right).
\end{align}
Therefore, we have 
\begin{align}\label{equ on tau1}
\tau_1(\gamma)&=\pi_{C*}(\ch_4(\bar{j}_*\mathbb{E})\cdot \pi_X^*\gamma) \\ \nonumber
&=\pi_{C*}\bar{j}_*\left(\left(\frac{1}{2}\ch_1(\mathbb{E})\cdot\pi_2^*c_1(\mathcal{C}^1_{\beta})+\ch_2(\mathbb{E})\right)\cdot\bar{j}^*\pi_X^*\gamma \right) \\\nonumber
&=\frac{1}{2}\pi_{1*}\left(\ch_1(\mathbb{E})\cdot\pi_2^*c_1(\mathcal{C}^1_{\beta})\cdot \pi_{2}^*j^*\gamma \right)+\pi_{1*}\left(\ch_2(\mathbb{E})\cdot\pi_{2}^*j^*\gamma \right) \\
\nonumber
&=\pi_{1*}\left(\ch_2(\mathbb{E})\cdot\pi_{2}^*j^*\gamma \right), 
\end{align}
where the last equality is because $\dim_{\mathbb{C}}\mathcal{C}^1_{\beta}=2$ and $c_1(\mathcal{C}^1_{\beta})\cdot j^*\gamma\in H^6(\mathcal{C}^1_{\beta})=0$.

From the exact sequence in $\Coh(\mathcal{C}^1_{\beta}\times \mathcal{C}^1_{\beta})$: 
$$0\to \oO_{\bar{p}^*(\Delta_{S^1_{\beta}})}\to \oO_{\bar{p}^*(\Delta_{S^1_{\beta}})}(\Delta_{\mathcal{C}^1_{\beta}})\to \oO_{\Delta_{\mathcal{C}^1_{\beta}}}(\Delta_{\mathcal{C}^1_{\beta}})\to 0, $$
we obtain 
\begin{align}\label{equ on ch2E}
\ch_2(\mathbb{E})&=\ch_2\left(\oO_{\bar{p}^*(\Delta_{S^1_{\beta}})}\right)+\ch_2\left(\oO_{\Delta_{\mathcal{C}^1_{\beta}}}(\Delta_{\mathcal{C}^1_{\beta}})\right) \\ \nonumber
&=\bar{p}^*\ch_2(\oO_{\Delta_{S^1_{\beta}}})+[\Delta_{\mathcal{C}^1_{\beta}}] \\ \nonumber
&=-\frac{1}{2}\bar{p}^*(\Delta_{S^1_{\beta}})_*(c_1(S^1_{\beta}))+[\Delta_{\mathcal{C}^1_{\beta}}], 
\end{align}
where $\Delta_{S^1_{\beta}}: S^1_{\beta}\to S^1_{\beta}\times S^1_{\beta}$ denotes the diagonal embedding and 
we use GRR formula for the map $\Delta_{S^1_{\beta}}$ in the last equation.

Combining Eqns.~\eqref{equ on tau1}, \eqref{equ on ch2E}, we obtain 
\begin{align}\label{equ on result of tau1}
\tau_1(\gamma)=-\frac{1}{2}\pi_{1*}\left(\bar{p}^*(\Delta_{S^1_{\beta}})_*(c_1(S^1_{\beta}))\cdot\pi_{2}^*j^*\gamma\right)+\pi_{1*}\left([\Delta_{\mathcal{C}^1_{\beta}}]\cdot\pi_{2}^*j^*\gamma \right)=j^*\gamma,
\end{align}
where we note that $\bar{p}^*(\Delta_{S^1_{\beta}})_*(c_1(S^1_{\beta}))$ is some multiple of the fiber class of $\bar{p}$, 
so the first term in above vanishes. Therefore, ECF of class $\beta$ contributes to $\blangle\tau_1(\gamma) \brangle^{\DT_4}_{\beta}$ by 
$$\int_{[M_\beta]^{\vir}}\tau_1(\gamma)=\int_{\mathcal{C}^1_{\beta}}j^*\gamma, $$
which gives exactly the genus 1 GV invariant $n_{1,\beta}(\gamma)$ for primitive $\beta$ as  
they are (virtually) enumerating elliptic curves of class $\beta$ incident to the cycle dual to $\gamma$.
\begin{rmk}\label{rmk on higher k}
For a general curve class $\beta$ and any $k\geqslant   1$ such that $k|\beta$, one can similarly show that any elliptic curve family $\mathcal{C}^1_{\beta/k}$ of class $\beta/k$ contributes to $\blangle\tau_1(\gamma) \brangle^{\DT_4}_{\beta}$ by 
$$\int_{\mathcal{C}^1_{\beta/k}}j^*\gamma=n_{1,\beta/k}(\gamma). $$ 
 Therefore, all elliptic curve families contribute to $\blangle\tau_1(\gamma) \brangle^{\DT_4}_{\beta}$ by $\sum_{k|\beta}n_{1,\beta/k}(\gamma)$.
\end{rmk}

\iffalse
For general $k\geqslant   1$ such that $k|\beta$, the contribution of ECF of class $\beta/k$ to $\blangle\tau_1(\gamma) \brangle^{\DT_4}_{\beta}$ can be similarly obtained. Under the 
isomorphism \eqref{equ on moduli from beta/k}, the 
normalized universal sheaf has $K$-theory class (see also \cite[\S1.5,~pp.~8]{CT2}):
$$[\oO_{\bar{p}^*(\Delta_{S^1_{\beta/k}})}(\Delta_{\mathcal{C}^1_{\beta/k}})]+(k-1)[\oO_{\bar{p}^*(\Delta_{S^1_{\beta/k}})}]. $$
Going through \eqref{equ on ch4}, \eqref{equ on tau1}, \eqref{equ on ch2E}, \eqref{equ on result of tau1}, we obtain 
$$\int_{[M_\beta]^{\vir}}\tau_1(\gamma)=\int_{\mathcal{C}^1_{\beta/k}}j^*\gamma. $$
To sum up, all ECF contribute to $\blangle\tau_1(\gamma) \brangle^{\DT_4}_{\beta}$ by 
\begin{align}\label{cont to tau1 from ECF}\sum_{k|\beta}\int_{\mathcal{C}^1_{\beta/k}}j^*\gamma
=\sum_{k|\beta}n_{1,\beta/k}(\gamma),  \end{align}
where the above equality holds because both sides are (virtually) enumerating elliptic curves of class $\{\beta/k\}_{k\geqslant  1, k|\beta}$ incident to the cycle dual to $\gamma$.
Putting together \eqref{cont to tau1 from RCF}, \eqref{cont to tau1 from ECF} with certain choice of sign, we verify Conjecture \ref{conj on DT4/GV} \eqref{g=1 part} and 
\eqref{equ on impri DT4/GV} in this ideal setting. 
\fi

\section{The embedded rational curve family} \label{sec:embedded rational curve family}
As a first illustration of the general case,
we work out here all Gromov-Witten, Gopakumar-Vafa and Donaldson-Thomas invariants
for a family of smooth irreducible rational curves {\em globally embedding} in a holomorphic symplectic 4-fold.
We will see that the global embedding assumption forces already almost all of our invariants to vanish.

\subsection{Setting}\label{sect on sett}
Let $X$ be a holomorphic symplectic 4-fold with symplectic form $\sigma \in H^0(X,\Omega_X^2)$.
Consider a family $p : \CC \to S$ of embedded rational curves in the irreducible curve class $\beta \in H_2(X,\BZ)$ parametrized by a smooth surface $S$. 

We make the following assumptions:
\begin{itemize}
\item[(i)] All fibers of $p$ are non-singular (isomorphic to $\p^1$).
\item[(ii)] The evaluation map $j : \CC \to X$ is a (global) embedding.
\item[(iii)] %(For notational simplicity) 
All curves in class $d \beta$ for all $d \geqslant   1$ are unions of curves of the family $\CC \to S$.
\end{itemize}

Let $\sigma \in H^0(X,\Omega_X^2)$ be the holomorphic symplectic form.
Since the pullback $j^{\ast}(\sigma) \in H^0(\CC, \Omega_{\CC}^2)$ vanishes on $T_{p}$, there exists
a $2$-form $\alpha \in H^0(S, \Omega_S^2)$ such that 
\[ p^{\ast}(\alpha) = j^{\ast}(\sigma). \]
If $\alpha$ vanishes at a point $s \in S$, then for every point $x$ in the fiber $\CC_s := p^{-1}(s)$ 
the form $\sigma$ vanishes on the image of $T_{\CC,x} \to T_{X,j(x)}$.
Since $\sigma_{j(x)}$ is non-degenerate, it can only vanish on a subspace of at most half the dimension of $T_{X,j(x)}$,
so this is impossible. Hence $\alpha$ does not vanish.
We conclude that $S$ is a \textit{holomorphic symplectic} surface, hence either an abelian or a $K3$ surface.

Moreover, consider the sequence
\[ 0 \to T_{\CC} \to j^{\ast}(T_X) \to N_{\CC/X} \to 0. \]
The form $\sigma' = \sigma|_{\CC} \in H^0(\CC, j^{\ast} \Omega_X^2)$ is non-degenerate;
so the vanishing $\sigma'(T_{p}, T_{\CC}) = 0$ implies that we have an isomorphism
\[ \sigma' : T_{p} \xrightarrow{\cong} N_{\CC/X}^{\vee}. \]

\begin{exam}\label{eg on hilb2}
Let $S^{[2]}$ be the Hilbert scheme of two points on a holomorphic symplectic surface $S$.
The Hilbert-Chow map from $S^{[2]}$ to the second symmetric product of $S$:
$$\pi: S^{[2]}\to \Sym^2(S) $$
is a resolution of singularity \cite{F}, whose exceptional divisor $D$
fits into the Cartesian diagram
\begin{align*} \xymatrix{
D \ar[d]_{p}  \ar[r]^{j \quad \,\,\, }    & S^{[2]} \ar[d]^{\pi}   \\
S    \ar[r]^{\Delta \quad \,\,\, }   & \Sym^2(S),  } \quad \quad
\end{align*}
where $\Delta$ is the diagonal embedding and $p: D\to S$ is a $\mathbb{P}^1$-bundle.
The pair $(S^{[2]}, \beta := j_{\ast}[ D_s ] )$ satisfies the assumptions (i-iii) for the family $D \to S$.
\end{exam}

\subsection{Gromov-Witten invariants}  \label{subsubsec:Embedded rational curve family}
In the setting (i-iii), we have the following computation of Gromov-Witten invariants.
In genus $0$, one has the following description:
\begin{lemma} \label{lemma: genus 0 embedded rat} 
%Assume (i-iii) above. 
For any $\gamma_1, \ldots, \gamma_n \in H^{\ast}(X)$, we have
\[ \blangle \tau_0(\gamma_1) \cdots \tau_0(\gamma_n) \brangle^{\GW}_{0,d \beta} = d^{-3+n} \int_{S} \prod_{i=1}^{n} p_{\ast}( j^{\ast}(\gamma_i)). \]
\end{lemma}
\begin{proof}
By condition (iii) the evaluation map factors as
\[
\ev: \Mbar_{0,n}(X,d\beta) \xrightarrow{\rho} \underbrace{\CC \times_S \cdots \times_S \CC}_{n \text{ times }} \xrightarrow{j \times \cdots \times j} X^n.
\]
Since $\Mbar_{0,n}(X,d\beta)$ is of virtual dimension $2+n = \dim( \CC \times_S \cdots \times_S \CC)$ we have
\[ 
\ev_{\ast}[ \Mbar_{0,n}(X,d\beta) ]^{\vir} = a_d [\CC \times_S \cdots \times_S \CC].
\]
By restriction to a fiber and using the Aspinwall-Morrison formula (see~e.g.~\cite[Prop.~7(i)]{O1} for our context), we have
\[ a_d = d^{-3+n}. \]
Consider the fiber diagram
\[
\begin{tikzcd}
\CC \ar[swap]{d}{p} & \CC^{\times_S n} \ar{l}{\pi_n} \ar{d}{\pi} \ar{r}{j^{\times n}} & X^n \\
S & \CC^{\times_S (n-1)} \ar{l}{p}
\end{tikzcd}
\]
where $\pi_n$ and $\pi$ are the projections to the $n$-th and the first $(n-1)$-factors respectively,
and $p$ is the structure morphism. We obtain:
\begin{align*}
\int_{ \CC^{\times_S n} } (j^{\times n})(\gamma_1 \otimes \cdots \otimes \gamma_n) 
& = \int_{ \CC^{\times_S n} }\pi^{\ast}\big( (j^{ \times (n-1)})^{\ast}(\gamma_1 \otimes \cdots \otimes \gamma_{n-1}) \big) \pi_n^{\ast} j^{\ast}(\gamma_n) \\
& = \int_{ \CC^{\times_S (n-1)} } \big( (j^{\times (n-1)} )^{\ast}(\gamma_1 \otimes \cdots \otimes \gamma_{n-1}) \big) p^{\ast}( p_{\ast}( j^{\ast}(\gamma_n) ) )  \\
& = \int_{S} \prod_{i=1}^{n} p_{\ast}( j^{\ast}(\gamma_i) ),
\end{align*}
where we used that
$\pi_{\ast} \pi_n^{\ast}(j^{\ast} \gamma_n) = p^{\ast} p_{\ast}(j^{\ast}\gamma_n)$
and then induction in the last step.
The claim follows by putting these two statements together.
\end{proof}
In genus $1$ and $2$, we have: 
\begin{lemma} \label{lemma:GW ERCF genus 12}
%Assuming (i-iii) above.
For any $\gamma \in H^4(X,\BZ)$ and $d \geqslant   1$, we have
\[
\blangle \tau_0(\gamma) \brangle^{\GW}_{1,d \beta} = \blangle \varnothing \brangle^{\GW}_{2,d \beta} = 0.
\]
\end{lemma}
\begin{proof}
Under our assumptions we have an isomorphism of moduli spaces
\[ \Mbar_{1,1}(X,\beta) \cong \Mbar_{1,1}(\CC, d F) \cong \Mbar_{1,1}(\CC/S, d), \]
where $\Mbar_{1,1}(\CC, d F)$ is the moduli space of stable maps to the (total space of) $\CC$ of degree $d$ times the fiber class $F$,
and $\Mbar_{1,1}(\CC/S, d)$ is the moduli space of stable maps into fibers of $\CC\to S$.
By comparing the perfect-obstruction theories of the first two moduli spaces one finds that:
\[
\blangle \tau_0(\gamma) \brangle^{\GW}_{1,d \beta}
=
\int_{ [ \Mbar_{1,1}(\CC, d F) ]^{\text{vir}} }  \ev_1^{\ast}(j^{\ast}(\gamma))\,e( \CV),
\]
where the fiber of the bundle $\CV$ at a point $[f : \Sigma \to \CC,p_1] \in \Mbar_{1,1}(\CC, dF)$
is the kernel of the semiregularity map
$H^1(\Sigma, f^{\ast}(N_{\CC/X})) \to H^1(\Sigma, \omega_{\Sigma}) = \BC$.

Similarly, the virtual classes of the latter two moduli spaces are related by
\[ [ \Mbar_{1,1}(\CC, dF) ]^{\text{vir}} = [ \Mbar_{1,1}(\CC/S, d) ]^{\text{vir}} \cdot e( \BE^{\vee} \otimes p^{\ast}(T_S)). \]
Since $S$ is symplectic, we have: $e( \BE^{\vee} \otimes p^{\ast}T_S) = c_2(T_S) - \lambda_1 c_1(T_S) = c_2(T_S)$. Hence
\begin{align*}
\blangle \tau_0(\gamma) \brangle^{\GW}_{1,d \beta} 
&=\int_{ [ \Mbar_{1,1}(\CC, d F) ]^{\text{vir}} } \ev_1^{\ast}(j^{\ast}(\gamma))\,p^{\ast}( c_2(T_S)) e( \CV ) \\
&= \int_{ [ \Mbar_{1,1}(\CC, d F) ]^{\text{vir}} } \ev_1^{\ast}(j^{\ast}(\gamma)\, p^{\ast} c_2(T_S)) e( \CV ) \\
&= 0,
\end{align*}
where in the last step we used that $j^{\ast}(\gamma)\, p^{\ast} c_2(T_S) = 0 \in H^{\ast}(\CC)$ for dimension reasons.

The case of genus $2$ is similar (using the Mumford relation \eqref{eq:Mumford Relation}).
%the class is of codimension $4$ (whereas $\dim \CC = 3$)
\end{proof}

We also have the following vanishing:
\begin{lemma} \label{lemma:vanishing c2}
For any $\gamma \in H^4(X,\BQ)$ and $d \geqslant  1$ we have:
\[
\blangle\tau_0(\gamma) \tau_0(c_2(X))  \brangle^{\GW}_{0,d \beta} = 0.
\]
\end{lemma}
\begin{proof}
Consider the invariant $\blangle \tau_1(\gamma) \brangle^{\GW}_{0,\beta}$.
By Lemma~\ref{lemma:psi1 in terms of Tp}, we have
\begin{equation} \label{YYY1}
\blangle \tau_1(\gamma) \brangle^{\GW}_{0,\beta} = \int_{\CC}j^{\ast}(\gamma)\,c_1(\omega_p).
\end{equation}
Applying Lemma~\ref{div equ on GW} to the divisor $[ \CC ] \in H^2(X,\BZ)$
which satisfies $[\CC] \cdot \beta = -2$, we also have:
%Then $D \cdot \beta = D \cdot j_{\ast}(F) = -2$ since $N_{\CC/X} = T_{p}^{\ast}$.
%Assuming that there are no other contributions (except from the family $p: \CC\to S$), 
\begin{equation}  \label{YYY2}
\blangle \tau_1(\gamma) \brangle^{\GW}_{0, \beta} 
= \frac{1}{4} \blangle \tau_0(\gamma)\tau_0([\CC]^2) \brangle^{\GW}_{0 , \beta} + \blangle \tau_0(\gamma \cdot [ \CC ]) \brangle^{\GW}_{0, \beta}. \end{equation} 
%By Lemma~\ref{lemma: genus 0 embedded rat} and 
Since $N_{\CC/X} \cong T_{p}^{\vee} = \omega_p$, we have
\[
\blangle \tau_0(\gamma \cdot [ \CC ]) \brangle^{\GW}_{0, \beta}
=
\int_{\CC} j^{\ast}(\gamma \cdot [ \CC ]) 
%= \int_{\CC} j^{\ast}(\gamma) c_1(N_{\CC/X})
= \int_{\CC} j^{\ast}(\gamma)\,c_1(\omega_p).
\]
Comparing Eqn.~\eqref{YYY1} and Eqn.~\eqref{YYY2}, we conclude with the help of Lemma~\ref{lemma: genus 0 embedded rat} that:
\begin{equation} \label{van11}
0=\blangle \tau_0(\gamma)\tau_0([\CC]^2) \brangle^{\GW}_{0 , \beta} = \int_{S} p_{\ast}( j^{\ast}(\gamma)) \cdot p_{\ast}( c_1(T_p)^2 ).
\end{equation}
The pair of short exact sequences
\begin{gather*}
0 \to T_C \to j^{\ast} T_X \to T_p^{\vee} \to 0, \\
0 \to T_p \to T_C \to p^{\ast}(T_S) \to 0
\end{gather*}
shows that
\[ j^{\ast}(c(X)) = 1 + p^{\ast}(c_2(S)) - c_1(T_p)^2 \]
and hence
\begin{equation} j^{\ast}(c_2(X)) = p^{\ast}(c_2(S)) - c_1(T_p)^2. \label{pullback c2} \end{equation}
Inserting into Eqn.~\eqref{van11}, we find
\[
\int_{S} p_{\ast}( j^{\ast}(\gamma)) \cdot p_{\ast}( j^{\ast}(c_2(X)) ) = 0.
\]
By Lemma~\ref{lemma: genus 0 embedded rat} this implies the claim (for all $d \geqslant 1$).
\end{proof}

We will also require the following evaluation. 
\begin{lemma} \label{lemma:descendent GW on ERCF}
For any $\gamma \in H^4(X,\BQ)$, we have
\begin{equation}
\label{equ on des g0 gw ideal}\blangle \tau_1(\gamma) \brangle^{\GW}_{0, d \beta} 
%= \frac{1}{d^3} \int_{\CC} j^{\ast}(\gamma)\,c_1(N_{\CC/X})
= \frac{1}{d^3} \int_{\CC} j^{\ast}(\gamma)\,c_1(\omega_p). 
\end{equation}
\end{lemma}
\begin{proof}
By Lemma~\ref{div equ on GW} 
applied to $D = [ \CC ] \in H^2(X,\BZ)$ (which satisfies $D \cdot \beta = -2$) we have:
%Then $D \cdot \beta = D \cdot j_{\ast}(F) = -2$ since $N_{\CC/X} = T_{p}^{\ast}$.
%Assuming that there are no other contributions (except from the family $p: \CC\to S$), 
\begin{equation*} 
\blangle \tau_1(\gamma) \brangle^{\GW}_{0, d \beta} 
= \frac{1}{4 d^2} \blangle \tau_0(\gamma)\tau_0([\CC]^2) \brangle^{\GW}_{0 , d \beta} + \frac{1}{d} \blangle \tau_0(\gamma \cdot [ \CC ]) \brangle^{\GW}_{0, d \beta}. \end{equation*} 
By Eqn.~\eqref{van11} and Lemma~\ref{lemma: genus 0 embedded rat} the first term vanishes.
And by Lemma~\ref{lemma: genus 0 embedded rat} again we get:
% and since $N_{\CC/X} \cong T_{p}^{\ast} = \omega_p$, we get
\[
\blangle \tau_0(\gamma \cdot [ \CC ]) \brangle^{\GW}_{0, d \beta}
=
\frac{1}{d^2} \int_{\CC} j^{\ast}(\gamma \cdot [ \CC ]) 
%= \int_{\CC} j^{\ast}(\gamma) c_1(N_{\CC/X})
= \frac{1}{d^2} \int_{\CC} j^{\ast}(\gamma) c_1(\omega_p). \qedhere
\]
\end{proof}

\subsection{Gopakumar-Vafa invariants}
We compute all $g\geqslant1$ Gopakumar-Vafa invariants in the setting (i-iii). % we have the following computation of Gromov-Witten invariants.

%We make the following additional assumption:
%\begin{itemize}
%\item[(iv)] $\CC = \p(V)$ for a rank $2$ vector bundle $V \to S$ which is self-dual, i.e. $V \cong V^{\vee}$.
%\end{itemize}
%
%\begin{lemma} \label{lemma"vanishing c2}
%Assume conditions (i-iv). 
%We have $p_{\ast} j^{\ast} c_2(T_X) = 0 \in H^{\ast}(S)$.
%\end{lemma}
%\begin{proof}
%%After a base change by a finite morphism $S' \to S$ we may assume that the $\p^1$-bundle $\CC \to S$ has a section, so $\CC = \p(V)$ for a rank $2$ vector bundle $V$ on $S$.
%The pair of short exact sequences
%\begin{gather*}
%0 \to T_C \to j^{\ast} T_X \to T_p^{\vee} \to 0, \\
%0 \to T_p \to T_C \to p^{\ast}(T_S) \to 0
%\end{gather*}
%shows that
%\[ j^{\ast}(T_X) = 1 + p^{\ast}(c_2(S)) - c_1(T_p)^2. \]
%With $\xi = c_1(\CO_{\p(V)}(1))$, we have
%\[ c_1(T_p) = 2 \xi + p^{\ast} c_1(V), \]
%so that (using $\xi^2 + \xi p^{\ast}c_1(V) + p^{\ast} c_2(V)=0$) we find
%\begin{equation} p_{\ast}( c_1(T_p)^2 ) = p_{\ast}( 4 \xi^2 + 2 \xi p^{\ast} c_1(V) - p^{\ast}(c_1(V)^2) ) = -2 c_1(V) = 0. \label{pushforward vanishing} \end{equation}
%where the vanishing comes from the self-duality.
%\end{proof}

\begin{lemma} 
%Assume conditions (i-iv).
For any $\gamma \in H^4(X,\BZ)$, we have
\[ n_{1,\beta}(\gamma) = n_{2,\beta} = 0. \]
\end{lemma}
\begin{proof}
By Lemmata~\ref{lemma:GW ERCF genus 12} and \ref{lemma:vanishing c2}
%, Lemma~\ref{lemma: genus 0 embedded rat} 
and the definition of Gopakumar-Vafa invariants it suffices to show that $N_{\mathrm{nodal},\beta}$ vanishes.
Since $\Mbar_{0,2}(X,\beta) = \CC \times_S \CC$ we have
\begin{align*}
N_{\mathrm{nodal},\beta}
& =
\frac{1}{2}\left[
\int_{ \CC \times_S \CC} (j \times j)^{\ast}(\Delta_X) - \int_{[ \Mbar_{0,1}(X,\beta) ]^{\vir}} \psi_1^3 + \ev_1^{\ast}(c_2(X)) \psi_1
\right]. \quad \quad
\end{align*}
To evaluate the first term we use that the preimage of the diagonal under $j \times j : \CC \times_S \CC \to X \times X$
is equal to $\CC$ and that the refined intersection has an excess bundle which is an extension of $T_S$ and $T_p^{\vee}$.
For the second term we use Eqn.~\eqref{pullback c2} and that by Lemma~\ref{lemma:psi1 in terms of Tp}
we have $\psi_1 = -c_1(T_p)$ under the isomorphism $\Mbar_{0,1}(X,\beta) \cong \CC$.
With this the above becomes:
\begin{align*}
%N_{\mathrm{nodal},\beta}
\quad \quad \quad \quad \quad \ \ 
& =
\frac{1}{2} \left[ \int_{\CC} e(T_S) c_1(T_p^{\vee}) - \int_{\CC} (-c_1(T_p))^3 + (p^{\ast}(c_2(S)) - c_1(T_p)^2) (-c_1(T_p)) \right] \\
& =
\frac{1}{2} \left[ -2 e(S) + 2 e(S) \right]= 0,
\end{align*}
where we used $\psi_1 = -c_1(T_p)$ under the isomorphism $\Mbar_{0,1}(X,\beta) \cong \CC$ by Lemma~\ref{lemma:psi1 in terms of Tp}.
\end{proof}

\subsection{$\DT_4$ invariants} \label{subsec:Embedded rational family DT}

\begin{lemma} \label{lemma:DT4 ERCF}
In the setting (i-iii), for certain choice of orientation, we have
\begin{align*}
\blangle\tau_0(\gamma_1),\cdots,\tau_0(\gamma_n) \brangle^{\DT_4}_{\beta}&=\int_{S}\,\prod_{i=1}^n(p_*j^*\gamma_i), \\
\blangle\tau_1(\gamma)\brangle^{\DT_4}_{\beta}&=-\frac{1}{2}\int_{\CC} j^*(\gamma)\, c_1(\omega_p), \\ 
 \blangle\tau_2(\theta)\brangle^{\DT_4}_{\beta}&=\frac{1}{12}\int_{\CC}j^*(\theta)\, c_1(\omega_p)^2-\frac{1}{12}\int_{\CC}j^*(c_2(X)\cdot\theta),  \\
 \blangle\tau_3(1)\brangle^{\DT_4}_{\beta}&=\frac{1}{24}\int_{\CC} j^*(c_2(X))\, c_1(\omega_p).
\end{align*}
Moreover, all $\DT_4$ invariants vanish in curve class $d \beta$ for $d>1$.
\end{lemma}
%\begin{GO}This still needs to be integrated, and then chhecked that GV/DT relation holds.\end{GO}
%\begin{GO}The minus sign appears since $c_1(\omega_p) = - c_1(T_p) = -c_1(D)$.\end{GO}
\begin{proof}
The computation is essentially done in \S \ref{sect on ideal DT4 comp}. 
By \cite[Lem.~2.2]{CMT1}, any one dimensional stable sheaf in class $d\beta$ is scheme theoretically supported on a fiber of $p: \CC\to S$. 
Therefore 
\begin{align}\label{equ on iso of mod on fiber class of HC}M_{\beta}\cong S,\quad M_{d\beta}=\emptyset, \,\,\mathrm{if}\,\,d\geqslant   2. \end{align}
And their virtual classes satisfy
\begin{equation}\label{equ on vir clas on erf}[M_{\beta}]^{\vir}=\pm[S],\quad [M_{d\beta}]^{\vir}=0, \,\,\mathrm{if}\,\,d\geqslant   2. \end{equation}
Under the isomorphism \eqref{equ on iso of mod on fiber class of HC} and the commutative diagram 
\begin{align*} 
\xymatrix{
\mathcal{C}  \ar[r]^{i=(p,j)\quad} \ar[dr]_{p} \ar@/_32pt/[drr]_{j}   & S\times X \ar[d]_{\pi_S} \ar[dr]^{\pi_X} & \\
  & S  & X, }
 \end{align*}
the normalized universal family $\mathbb{F}_{\mathrm{norm}}$ is $i_*\oO_{\CC}$. 
By Grothendieck-Riemann-Roch formula, 
\begin{align}\label{ch charc}
\ch(i_*\oO_{\mathcal{C}})&=i_*(\td^{-1}(N_{\mathcal{C}/S\times X})) \\ \nonumber
&=i_*\left(1-\frac{1}{2}c_1(\omega_p)+\frac{1}{12}\left(c_1(\omega_p)^2-j^*c_2(X)\right)+\frac{1}{24}c_1(\omega_p)\cdot j^*c_2(X)\right). 
\end{align}
Therefore 
\begin{align*} 
\tau_k(\gamma)&=\pi_{S*}\left(\pi_{X}^*\gamma \cup \ch_{k+3}(i_*\oO_{\mathcal{C}})\right) \\
&=\pi_{S*}\left(\pi_{X}^*\gamma \cup i_*\left[\td^{-1}(N_{\mathcal{C}/S\times X})\right]_{\deg_{\mathbb{C}} k}\right) \\
&=p_*\left(j^*\gamma \cup \left[\td^{-1}(N_{\mathcal{C}/S\times X})\right]_{\deg_{\mathbb{C}} k}\right) . 
\end{align*}
Combining with Eqns.~\eqref{equ on vir clas on erf}, \eqref{ch charc}, we are done. 
\end{proof}
To sum up, combining Lemmata \ref{lemma: genus 0 embedded rat}--\ref{lemma:DT4 ERCF}, we obtain:
\begin{thm}
%Choose plus sign in \eqref{equ on vir clas on erf} as an orientation. For any $d\geqslant  1$, we have 
%$$\blangle\tau_0(\gamma_1),\cdots,\tau_0(\gamma_n) \brangle^{\DT_4}_{d\beta}=n_{0,d\beta}(\gamma_1,\ldots,\gamma_n), $$
%$$\blangle\tau_1(\gamma)\brangle^{\DT_4}_{d\beta}= ,$$
Conjecture \ref{conj on integrality} and Conjecture \ref{conj on DT4/GV} hold in the setting specified in \S \ref{sect on sett}. 
\end{thm}

\section{Tautological integrals on moduli spaces of sheaves on $K3$ surfaces}
In this section, we compute several tautological integrals on moduli spaces of one dimensional stable sheaves on $K3$ surfaces.
These will be used in Section \ref{sect on prod} to compute $\DT_4$ invariants on the product of $K3$ surfaces, though 
they are interesting in their own right.

\subsection{Fujiki constants}\label{sect on fujiki}
The second cohomology $H^2(M, \BZ)$ of an irreducible hyperk\"ahler variety
carries a integral non-degenerate quadratic form
\[ \mathsf{q}: H^2(M, \mathbb{Z}) \to \mathbb{Z}, \]
called the Beauville-Bogomolov-Fujiki form.
By the following result of Fujiki \cite{Fuji} (and its generalization in \cite{GHJ}) 
it controls the intersection numbers of products of divisors against
Hodge cycles which stay Hodge type on all deformations of $M$:

 %\cite[Cor.~23.17]{GHJ}.
\begin{thm}\label{fujiki result}$($\cite{Fuji}, \cite[Cor.~23.17]{GHJ}$)$
%Let $M$ be a hyperk\"ahler variety of dimension $2n$. 
Assume $\alpha\in H^{4j}(M,\mathbb{C})$ is of type $(2j, 2j)$ on all small deformation of $M$. Then there exists a unique constant $C(\alpha)\in\mathbb{C}$ depending only on $\alpha$ and called the Fujiki constant of $\alpha$ such that
for all $\beta \in H^2(M,\BC)$ we have
\begin{equation}\label{fukdefn}\int_{M}\alpha\cdot\beta^{2n-2j}=C({\alpha})\cdot \mathsf{q}(\beta)^{n-j}. \end{equation}
\end{thm}
In this section, we consider the Hilbert scheme $S^{[n]}$ of $n$-points of a $K3$ surface $S$,
which by the work of Beauville \cite{Beauville} is irreducible hyperk\"ahler.
We will prove a closed formula for the Fujiki constants of all Chern classes of its tangent bundle.
%which seems to be new.  %s on Hilbert schemes $S^{[n]}$ of $n$-points of a K3 surface $S$. 

For $k \geqslant   2$ even, we define the classical Eisenstein series
\begin{equation}\label{eisen} G_k(q) = - \frac{B_k}{2 \cdot k} + \sum_{n \geqslant   1} \sum_{d|n} d^{k-1} q^n, \end{equation}
where $B_k$ are Bernoulli numbers, i.e. $B_2=\frac{1}{6}$, $B_4=-\frac{1}{30}$, $\cdots$.  For example, we have
\[ G_2(q) = -\frac{1}{24} + \sum_{n \geqslant   1} \sum_{d|n} d q^n. \]
\begin{thm}\label{thm on fujiki cons} 
Let $S$ be a $K3$ surface. For any $k \geqslant   0$,
\[ \sum_{ n \geqslant   k} C( c_{2n-2k}(T_{S^{[n]}}))\,q^n 
=
\frac{ (2k)!}{ k! 2^{k}} \left( q \frac{d}{dq} G_2(q) \right)^k \prod_{n \geqslant   1} (1-q^n)^{-24}. \]
\end{thm}
The first coefficients are listed in Table~\ref{Fujiki_table}.
Remarkablely, the right hand side in Theorem~\ref{thm on fujiki cons} is
up to the prefactor $(2k)! / (k! 2^k)$ precisely the generating series of counts of genus $k$ curves on a $K3$ surface passing through $k$ generic points \cite{BL}.
This suggests a relationship to the work of G\"ottsche on curve counting on surfaces \cite{Gott2}.
The proof presented below uses similar ideas as in \cite{Gott2},
but we could not directly deduce it from there. 
The relationship to curve counting on $K3$ surfaces will be taken up in a follow-up work.

\begin{table}[ht]
{\renewcommand{\arraystretch}{1.2}\begin{tabular}{| c | c | c | c | c | c | c | c |}
\hline
    & $c_0$ & $c_2$ & $c_4$ & $c_6$ & $c_8$ & $c_{10}$ & $c_{12}$ \\
\hline
$S^{[0]}$ & $1$ & & & & & &  \\
$S^{[1]}$ & $1$ & $24$ & & & & & \\
$S^{[2]}$ & $3$ & $30$ & $324$ & & & & \\
$S^{[3]}$ & $15$ & $108$ & $480$ & $3200$ & & & \\
$S^{[4]}$ & $105$ & $630$ & $2016$ & $5460$ & $25650$ & & \\
$S^{[5]}$ & $945$ & $5040$ & $13500$ & $26184$ & $49440$ & $176256$ & \\
$S^{[6]}$ & $10395$ & $51030$ & $122220$ & $198300$ & $266490$ & $378420$ & $1073720$ \\
\hline
\end{tabular}}
\vspace{8pt}
\caption[]{The first non-trivial Fujiki constants of the Chern classes $c_k:=c_k(T_{S^{[n]}})$
of Hilbert schemes of points on a $K3$ surface.
The modularity of Theorem~\ref{thm on fujiki cons} appears in the diagonals, e.g.
the cases $k=0,1$ are the functions:
\begin{minipage}{\linewidth}
%\begin{gather*}
\begin{equation}\prod_{n \geqslant   1} (1-q^n)^{-24} = 1 + 24 q + 324 q^2 + 3200 q^2 + \cdots.   \nonumber \end{equation}
\begin{equation}\left( q \frac{d}{dq} G_2(q) \right)  \prod_{n \geqslant   1} (1-q^n)^{-24} =q + 30q^{2} + 480q^{3} + 5460q^{4} + \cdots. \nonumber \end{equation}
%\end{gather*}
\end{minipage} 
}
%\end{gather*}
%\end{minipage}
\label{Fujiki_table}
\vspace{-25pt}
\end{table}

\begin{proof}
Let $L \in \Pic(S)$ be a line bundle on an arbitrary surface $S$.
Consider the series
\[ \Phi_{S,L} = \sum_{n=0}^{\infty} q^n \int_{S^{[n]}} c(T_{S^{[n]}})\,e^{c_1(L_n)}, \]
where we let $L_n = \det( L^{[n]} ) \otimes \det(\CO_S^{[n]})^{-1}$.
Since the integrand is multiplicative, by \cite[Prop 3]{Nieper} (which immediately follows from \cite{EGL}), there exists
power series $\mathsf{A},\mathsf{B},\mathsf{C},\mathsf{D}$ in $q$ such that for any surface $S$ and line bundle $L$, we have
\[
\Phi_{S,L}
=
\exp\left(
c_1(L)^2 \mathsf{A} + c_1(L) \cdot c_1(S) \mathsf{B} + c_1(S)^2 \mathsf{C} + c_2(S) \mathsf{D} \right).
\]
The G\"ottsche formula
\[ \Phi_{S,0} = \sum_{n \geqslant   0} q^n \int_{S^{[n]}} c_{2n}(T_{S^{[n]}})
=
\prod_{ n \geqslant   1} \frac{1}{(1-q^n)^{e(S)}}
\]
shows then that $\mathsf{C}=0$ and provides an explicit expression for $\mathsf{D}$. Hence we have
\[
\Phi_{S,L}
=
\prod_{n \geqslant   1} (1-q^n)^{-e(S)}
\exp\left(
c_1(L)^2 \mathsf{A} + c_1(L) c_1(S) \mathsf{B} \right).
\]
%where $A = A_1(q) t + A_2(q) t^2 + \ldots \, $.
Replacing $L$ by $L^{\otimes t}$ for $t \in \BZ$ shows that
\begin{equation} \label{tseries}
\sum_{n=0}^{\infty} q^n \int_{S^{[n]}} c(T_{S^{[n]}})\,e^{t c_1(L_n)}
=
\prod_{n \geqslant   1} (1-q^n)^{-e(S)}
\exp\left(
c_1(L)^2 t^2 \mathsf{A} + c_1(L) c_1(S) t \mathsf{B} \right).
\end{equation}
Since both sides are power series with coefficients which are polynomials in $t$, and the equality holds for all $t \in \BZ$, we find that Eqn.~\eqref{tseries} also holds for $t$, a formal variable.
%it is true for all $t \in \BZ$
%it is already true for $t$ a formal variable.
We write $\Phi_{S,L}(t)$ for the series \eqref{tseries}.
%This holds for $t \in \BZ$, but since this 
%for $t \in \BZ$ and then, since both sides are power series with c
We argue now in two steps.

\vspace{5pt} \noindent
\textbf{Step 1: Specialization to $K3$ surfaces.}
Let $S$ be a $K3$ surface. Since $S^{[n]}$ is holomorphic symplectic, its odd Chern classes vanish. Together with Eqn.~\eqref{fukdefn} and $\mathsf{q}(L_n) = c_1(L)^2$ this gives
\begin{align}
\Phi_{S,L}(t)
& = \sum_{n \geqslant   0} q^n \sum_{k=0}^{n} \frac{1}{(2k)!} \int_{S^{[n]}} c_{2n-2k}(T_{S^{[n]}}) c_1(L_n)^{2k} t^{2k} \notag \\
%& = \sum_{k \geqslant   0} \frac{t^{2k}}{(2k)!} \sum_{n \geqslant   k} \int_{S^{[n]}} c_{2n-2k}(\Tan) c_1(L_n)^{2k}  \\
& = \sum_{k \geqslant   0} \frac{t^{2k} \big( c_1(L)^2 \big)^k}{(2k)!}
\sum_{n \geqslant   k} C(c_{2n-2k}(T_{S^{[n]}})) q^n \notag \label{ssss}.
\end{align}
On the other hand, we have
\[ \Phi_{S,L}(t) =
\prod_{n \geqslant   1} (1-q^n)^{-24} \cdot \exp( c_1(L)^2 t^2 \mathsf{A} ) .\]
By taking the $t^{2k}$ coefficient we obtain that
\begin{equation}\label{equ on fujik hilbn k3}
\sum_{n \geqslant   k} C(c_{2n-2k}(T_{S^{[n]}}))\,q^n
=
\frac{(2k)!}{k!} \mathsf{A}(q)^k \prod_{n \geqslant   1} (1-q^n)^{-24}.
\end{equation}

\vspace{5pt} \noindent
\textbf{Step 2: Specialization to abelian surfaces.}
Let $A$ be an abelian surface with a line bundle $L \in \Pic(A)$ such that $c_1(L)^2 \neq 0$.
Let $\sigma : A^{[n]} \to A$ be the sum map,
and let 
\[ \Kum_{n-1}(A) = \sigma^{-1}(0_A) \]
be the generalized Kummer variety of dimension $2n-2$.
We have the fiber diagram
\[
\begin{tikzcd}
A \times \Kum_{n-1}(A) \ar{r}{\nu} \ar{d}{\mathrm{pr}_1} & A^{[n]} \ar{d}{\sigma} \\
A \ar{r}{n \times } & A.
\end{tikzcd}
\]
In particular $\nu$ is \'etale (of degree $n^4$), which implies that
\[ \nu^{\ast} T_{A^{[n]}} \cong \mathrm{pr}_1^{\ast}(T_A) \oplus \mathrm{pr}_2^{\ast}( T_{\Kum_{n-1}(A)} ). \]
Since the Chern classes of an abelian surface vanish,
we obtain
\[ \nu^{\ast} c( T_{A^{[n]}} ) = \mathrm{pr}_2^{\ast} c( T_{\Kum_{n-1}(A)} ). \]
Moreover one has (see \cite[Eqn.~(2)]{Nieper}) that
\[ \nu^{\ast}(L_n) = \mathrm{pr}_1^{\ast}(L ^{\otimes n}) \otimes \left( L_n|_{\Kum_{n-1}(A)} \right). \]
We obtain that
\begin{align*}
\int_{A^{[n]}} c_1(L_n)^2 c_{2n-2}(T_{A^{[n]}}) 
& = \frac{1}{n^4} \int_{A \times \Kum_{n-1}(A)}
c_1( \nu^{\ast}(L_n) )^2 c_{2n-2}( \nu^{\ast} T_{\Kum_{n-1}(A)}) \\
& = \frac{1}{n^2} (c_1(L)^2) \cdot \int_{\Kum_{n-1}(A)} c_{2n-2}( T_{\Kum_{n-1}(A)} ) \\
& = \frac{1}{n^2} (c_1(L)^2)\,e( \Kum_{n-1}(A) ). 
\end{align*}
Using that $e(\Kum_{n-1}(A)) = n^3 \sum_{d|n} d$ (ref.~\cite{GS}) and Eqn.~\eqref{tseries}, we conclude that
\begin{align*}
(c_1(L)^2) \cdot \mathsf{A}(q) =
[\Phi_{A,L}(t)]_{t^2}
&=\frac{1}{2} \sum_{n \geqslant   0} q^n \int_{A^{[n]}} c_1(L_n)^2 c_{2n-2}(T_{A^{[n]}}) \\
&=\frac{(c_1(L)^2)}{2} \sum_{n \geqslant   1} n \sum_{d|n} d q^n,
\end{align*}
where $[-]_{t^2}$ denotes the coefficient of $t^2$ term.
Hence 
$$\mathsf{A}(q) = \frac{1}{2} q \frac{d}{dq} G_2(q). $$ 
Combining with Eqn.~\eqref{equ on fujik hilbn k3}, we are done.
%The series $A,B$ can be computed by usings Bott's residue formula for the pairs $(S,L) \in \{  (\p^2, \CO_{\p^2}(1)), (\p^2, \CO(2)) \}$.
\end{proof} 
For completeness we also state the Fujiki constants of Chern classes
of the second known infinite family of hyperk\"ahler varieties,
the generalized Kummer varieties.
\begin{prop}\label{prop on fujiki on kumm}
For any $k \geqslant   0$ and abelian surface $A$, we have
\[
\sum_{n \geqslant   k} C( c_{2n-2k}(T_{\Kum_n(A)} ))\,q^{n+1}
=
\frac{ (2k)!}{ (k+1)! 2^k }
\left( q \frac{d}{dq} \right)^2\left(  q \frac{d}{dq} G_2 \right)^{k+1}.
\]
\end{prop}
\begin{proof}
Using the universality \eqref{tseries} and the value of $\mathsf{A}(q)$ we computed above, one concludes that for any line bundle $L$ on $A$, we have:
\begin{align*}
& \frac{(2k)!}{k! 2^k} \left( q \frac{d}{dq} G_2 \right)^k ( c_1(L)^2 )^k \\
= & \sum_{n \geqslant   0} q^n \int_{A^{[n]}} c_{2n-2k}(T_{A^{[n]}}) c_1(L_n)^{2k} \\
= & \sum_{n \geqslant   0} q^n \frac{1}{n^4} \int_{A \times \Kum_{n-1}(A)} c_{2n-2k}( \nu^{\ast}(T_{A^{[n]}}) )
c_1( \nu^{\ast}L_n)^{2k} \\
= & \sum_{n \geqslant   0} q^n \frac{1}{n^2} \binom{2k}{2} \left( \int_A c_1(L)^2 \right) \mathsf{q}(L_n|_{\Kum_{n-1}(A)})^{k-1} C( c_{2n-2k}(T_{\Kum_{n-1}(A)})).
\end{align*}
Using $\mathsf{q}(L_n|_{\Kum_{n-1}(A)}) = c_1(L)^2$ we conclude the claim.
\end{proof}

\begin{rmk}
It is remarkable that all Fujiki constants of $c_k(T_X)$ for $X \in \{ S^{[n]}, \Kum_{n}(A)\}$ are positive integers.
By the software package `bott' of J. Song \cite{bott}, the same can be checked numerically for arbitrary products of Chern classes of the tangent bundle $($up to $n \leqslant 10$$)$. We also refer to \cite{CJ,J} for some general results on positivity of Todd classes of hyperk\"ahler varieties, and to \cite{OSV} for a discussion on positivity of Chern (character) numbers.
This suggests the question whether all 
(non-trivial) Fujiki constants of products of Chern classes on irreducible hyperk\"ahler varieties positive.
This question was raised independently and then studied in \cite{BS, Sawon}.
\end{rmk}

\subsection{Descendent integrals on the Hilbert scheme}\label{sect on des inte on hil}
We now turn to integrals over descendents on Hilbert schemes, which are defined
for $\alpha \in H^{\ast}(S)$ and $d \geqslant   0$ by
\[ \fG_d(\alpha):= \pi_{\Hilb*}( \pi_S^{\ast}(\alpha) \ch_d(\CO_{\CZ}) ) \in H^{\ast}(S^{[n]}), \]
where $\pi_{\Hilb}, \pi_S$ are projections from $S^{[n]} \times S$ to the factors.
We prove the following evaluations:
\begin{prop}\label{prop on des on hilb1}
Let $\pt\in H^4(S)$ be the point class. Then 
\[ \sum_{n \geqslant   0} q^n \int_{S^{[n]}} c( T_{S^{[n]}} ) \fG_2(\pt) =  \left(\frac{1}{24}+G_2 \right)\prod_{n \geqslant   1} (1-q^n)^{-24}. % \tag{i}
\]
\end{prop}
\begin{proof}
This is a special case of \cite{QS}, but we can give a direct argument.
For any surface $S$ and $K$-theory class $x \in K(S)$ with $\ch_0(x) = \ch_1(x) = 0$ consider the series
\[ \Phi_{S,x} = \sum_{n \geqslant   0} q^n \int_{S^{[n]}} c( T_{S^{[n]}} ) e^{\ch_2(x^{[n]})}. \]
By \cite{EGL} and since we know the answer for $x=0$, there exists a series $\mathsf{A}(q)$ such that
\[ \Phi_{S,x} = \prod_{n \geqslant   1} (1-q^n)^{-24}  \exp\left( \ch_2(x) \mathsf{A} \right). \]
Setting $x = t \CO_{\pt}$, we in fact get the equality of
\[
\Phi_{S,t} := \sum_{n \geqslant   0} q^n \int_{S^{[n]}} c( T_{S^{[n]}} ) e^{\ch_2(\CO_{\pt}^{[n]}) t}
=
\prod_{n \geqslant   1} (1-q^n)^{-24} \exp\left( \mathsf{A} t \right).
\]

\noindent
\emph{Case 1: $K3$ surfaces.} By GRR and taking the $t^1$-coefficient, one finds that
\begin{equation}\label{equ on k3 com sigma2}
\sum_{n \geqslant   0} q^n \int_{S^{[n]}} c( T_{S^{[n]}} ) \fG_2(\pt)
=
[\Phi_{S,t}]_{t^1}  = 
\mathsf{A}(q) \prod_{n \geqslant   1} (1-q^n)^{-24}.
\end{equation}

\noindent
\emph{Case 2: Abelian surfaces.} For an abelian surface $A$, similar as before, we have 
\begin{align*}
\int_{A^{[n]}} c_{2n-2}( T_{A^{[n]}} ) \ch_2( \CO_{\pt}^{[n]} ) 
&=
\frac{1}{n^4} \int_{A \times \Kum_{n-1}(A)} \nu^{\ast}( \ch_2( \CO_{\pt}^{[n]} ) ) c_{2n-2}(T_{\Kum_{n-1}(A)})  \\
&=
\frac{e( \Kum_{n-1}(A) ) }{n^3}
= \frac{1}{n^3} \left( n^3 \sum_{d|n} d \right)
= \sum_{d|n} d.
\end{align*}
Here we used that $\nu^{\ast} \ch_2( \CO_{\pt}^{[n]} )|_{A \times pt} = n \pt$.
(To see the last statement, consider the diagram
\[
\begin{tikzcd}
A \times \Kum_{n-1}(A) \times A \ar{d} \ar{r}{\nu \times \id} & A^{[n]} \times A \ar{d}{\pi} \\
A \times \Kum_{n-1}(A) \ar{r}{\nu} \ar{d}{\mathrm{pr}_1} & A^{[n]} \ar{d}{\sigma} \\
A \ar{r}{n \times } & A.
\end{tikzcd}
\]
Let $\CZ \subset A^{[n]} \times A$ be the universal subscheme,
and let $\CZ_{\Kum} \subset \Kum_{n-1}(A) \times A$ be its restriction to the Kummer.
Inside $A \times \Kum_{n-1} \to A$, we have an equality of subschemes:
\[ (\nu \times \id)^{-1}(\CZ) = m_{13}^{-1}(Z_{\Kum}), \]
where $m_{13}$ is the addition map on the outer factors. 
Restricting to $A \times pt \times A$ we find that $m_{13}|_{A \times pt \times A}^{\ast}(n \pt) = n \Delta_{A}$. Then the claim follows from the definition).
We hence obtain that
\[ \mathsf{A}(q) = [\Phi_{A,t}]_{t^1} = 
\sum_{n \geqslant   0}q^n \int_{A^{[n]}} c_{2n-2}( T_{A^{[n]}} ) \ch_2( \CO_{\pt}^{[n]} ) 
= \frac{1}{24} + G_2(q).  \] 
Combining with Eqn.~\eqref{equ on k3 com sigma2}, we are done. 
\end{proof} 

\begin{prop} \label{lem on vanishing}
%For any divisor $D \in H^2(S,\BQ)$ we have
\begin{align}
\int_{S^{[d]}}
c_{2d-2}(S^{[d]})
\fG_3(D) & = 0  \,\,\, \text{ for all divisors } D \in H^2(S,\BQ) \tag{i},
 \\
\sum_{n \geqslant   0} q^n \int_{S^{[n]}} c( T_{S^{[n]}} ) \fG_4(1) & =\left( -20 G_2^2 - 2 G_2 - 5/3 G_4 - 1/24 \right) \prod_{n \geqslant   1} (1-q^n)^{-24}, \tag{ii}
\end{align}
where $G_k$ is given in \eqref{eisen}. 
\end{prop}
\begin{proof}
Recall that for any hyperk\"ahler variety $X$, the
Looijenga-Lunts-Verbitsky Lie algebra
$\Fg(X)$ is isomorphic to $\mathrm{so}(H^2(X,\BQ) \oplus U_{\BQ})$, where $U = \binom{0\ 1}{1\ 0}$ is the hyperbolic lattice \cite{LL, Ver0, Ver1}. The degree $0$ part of the Lie algebra splits as
$\Fg_0(X) = \BQ h \oplus \mathfrak{so}(H^2(X,\BQ))$ where
$h$ is the degree grading operator.
Looijenga and Lunts show that for the natural action of $\Fg(X)$ on cohomology, the subliealgebra $\mathfrak{so}(H^2(X,\BQ))$ acts by derivations. In other words, if $\Ft \subset \mathfrak{so}(H^2(X))$ 
is a maximal Cartan, we have a decomposition
\[ H^{\ast}(X) = \bigoplus_{\lambda : \Ft \to \BZ} V_{\lambda}, \]
which is multiplicative, i.e. $V_{\lambda} \cdot V_{\mu} \subset V_{\lambda + \mu}$.
Here $\lambda$ runs over all weights of the torus and $V_{\lambda}$ is the corresponding eigenspace.

For a Hilbert scheme, let $\delta = c_1( \CO_S^{[n]})$ and recall the natural decomposition\[ H^2(S^{[d]}) = H^2(S) \oplus \BQ \delta. \]
We consider the subliealgebra $\mathfrak{so}(H^2(S)) \subset \mathfrak{so}(H^2(X,\BQ))$ and for a Cartan $\Ft' \subset \mathfrak{so}(H^2(S))$ the associated decomposition
\[ H^{\ast}(S^{[d]}) = \bigoplus_{\mu : \Ft' \to \BZ} V_{\mu}. \]
Since Chern classes are monodromy invariant, they lie in $V_{0}$; see \cite{LL} for a discussion.
If $D=0$ there is nothing to prove. Otherwise, we can choose the Cartan $\Ft'$ such that $D$ lies in a non-zero eigenspace of the action of $\mathfrak{so}(H^2(S))$ on $H^2(S,\BQ)$.
Since the map 
$\gamma \mapsto \pi_{\Hilb*}\left(\ch_d(\oO_\mathcal{Z}) \pi_S^{\ast}(\gamma) \right)$ is equivariant with respect to the action
of  $\mathfrak{so}(H^2(S))$ on $H^{\ast}(S^{[d]})$ and $H^{\ast}(S)$ respectively, we conclude that also
\[ \fG_3(D) = \pi_{\Hilb*}\left(\ch_3(\oO_\mathcal{Z}) \pi_S^{\ast}(\gamma) \right) \]
is of non-trivial weight with respect to $\Ft'$, i.e. lies in $V_{\mu}$ for $\mu \neq 0$.

By multiplicativity of the decomposition, it follows that the integrand $c_{2d-2}(S^{[d]}) \cdot \fG_3(D)$ is of non-zero weight, hence its integral must be zero. This proves (i).

For part (ii) we start with the vanishing from part (i): For any divisor $W \in H^2(S,\BQ)$ we have
\[ \int_{S^{[n]}} c_{2d-2}(S^{[n]}) \fG_3(W) = 0. \]
%We use now the Looijenga-Lunts-Verbitsky Lie algebra.
In the notation of \cite[Eqn.~(36)]{NOY}
consider the element $h_{F \delta} = F \wedge \delta$ in $\mathfrak{so}(H^2(X,\BQ))$ for some $F \in H^2(S, \BQ)$.
Since the integrated degree $0$ part of the LLV algebra acts as ring isomorphisms
and preserves the Chern classes (see \cite{LL}), we have
\begin{align}\label{invariance} \int_{S^{[n]}} c_{2d-2}(S^{[n]}) e^{t\cdot h_{F \delta}}( \fG_3(W) )
&= \int_{S^{[n]}} e^{t \cdot h_{F \delta}} ( c_{2d-2}(S^{[n]}) \fG_3(W) ) \\ \nonumber
&= \int_{S^{[n]}} c_{2d-2}(S^{[n]}) \fG_3(W) = 0. \end{align}
By \cite[Prop.~4.4]{NOY}, we have that
\[
h_{F \delta}( \fG_3(D_1) ) = - \fG_2(F) \fG_2(W) - \blangle F, W \brangle \fG_2(1) \fG_2(\pt) - \blangle F, W \brangle \fG_4( 1 ).
\]
Taking the derivative $\frac{d}{dt}|_{t=0}$ of Eqn.~\eqref{invariance}, we find that:
\[
\blangle F, W \brangle \int_{S^{[n]}} c(T_{S^{[n]}}) \fG_4(1)
=
- \int_{S^{[n]}} c(T_{S^{[n]}}) \fG_2(F) \fG_2(W) - n \blangle F, W \brangle \int_{S^{[n]}} c(T_{S^{[n]}}) \fG_2(\pt).
\]
Note that by Theorem~\ref{thm on fujiki cons}, we have
\[
\sum_{n \geqslant   0} q^n \int_{S^{[n]}} c(T_{S^{[n]}}) \fG_2(F) \fG_2(W) = \blangle F, W \brangle \prod_{n \geqslant   1} (1-q^n)^{-24} q \frac{d}{dq} G_2(q).
\]
Moreover, by considering $\log$-derivative, one has 
\begin{equation}\label{diffeqn}  q \frac{d}{dq} \prod_{n \geqslant   1} (1-q^n)^{-24} = \prod_{n \geqslant   1} (1-q^n)^{-24} ( 1 + 24 G_2 ). \end{equation}
Therefore we find as desired
\begin{align}  \sum_{n \geqslant   0} q^n \int_{S^{[n]}} c(T_{S^{[n]}}) \fG_4(1)
& = - \prod_{n \geqslant   1} (1-q^n)^{-24} \left( 2 q \frac{d}{dq} G_2(q) + (\frac{1}{24} + G_2) (1 + 24 G_2) \right) \label{sstwet} \\
& = - \prod_{n \geqslant   1} (1-q^n)^{-24} ( 5/3 G_4 + 20 G_2^2 + 2 G_2 + 1/24 ), \notag
\end{align}
where we used the following Ramanujan differential equation \cite[pp. 49, Prop.~15]{1-2-3}
\[ q \frac{d}{dq} G_2(q) = - 2 G_2(q)^2 + \frac{5}{6} G_4(q). \qedhere \]
%The claim then follows by a straightforward calculation from the first part and Theorem~\ref{thm on fujiki cons}.
\end{proof}

\subsection{Descendent integrals on moduli spaces of $1$-dimensional sheaves}\label{sect on des int on mod}
Let $\beta \in H_2(S,\BZ)$ be an effective curve class and let
$M_{S,\beta}$ be the moduli space of one dimensional stable sheaves $F$ on $S$ with $[F]=\beta$ and $\chi(F)=1$.
By a result of Mukai \cite{M}, $M_{S,\beta}$ is a smooth projective holomorphic symplectic variety of dimension $\beta^2+2$. 
Let $\BF$ be the normalized universal family, i.e. which satisfies $\det \dR \pi_{M\ast} \BF = \CO_{M_{S,\beta}}$.
For $\alpha \in H^{\ast}(S)$, we define the descendents
\[ \sigma_d(\alpha) = \pi_{M\ast}( \pi_S^{\ast}(\alpha) \ch_d(\BF) ). \]
We have the following evaluations:
%\begin{prop}
%\begin{align}
%\int_{M_{S,\beta}} c_{2d-2}(T_{M_{S,\beta}}) \sigma_3(1)\sigma_1(\pt) & = 0,  \tag{i} \\
%\int_{M_{S,\beta}} c_{2d-2}(T_{M_{S,\beta}}) \sigma_3(1)\sigma_2(D) & = , \tag{ii} \\
%\int_{M_{S,\beta}} c_{2d-2}(T_{M_{S,\beta}}) \sigma_2(D)\sigma_1(\pt) & = , \tag{iii}
%\end{align}
%\end{prop}
%We will also need the following:
\begin{prop} \label{prop:desc M_Sbeta} Let $\beta \in H_2(S,\BZ)$ be an effective curve class. For the point class $\pt\in H^4(S)$ and $D \in H^2(S)$, we have
\begin{align}
\int_{M_{S,\beta}} c( T_{M_{S,\beta}} ) \sigma_2(\pt) & = N_1\left( \frac{\beta^2}{2} \right),  \tag{i} \\
\int_{M_{S,\beta}} c( T_{M_{S,\beta}} ) \sigma_3(D) & = - (D \cdot \beta) N'\left( \frac{\beta^2}{2} \right), \tag{ii} \\
\int_{M_{S,\beta}} c( T_{M_{S,\beta}} ) \sigma_4(1) & = - N'\left( \frac{\beta^2}{2} \right), \tag{iii}
\end{align}
where $N_{1}(l)$, $N'(l)$ for all $l \in \BZ$ are defined by the generating series
\begin{align}\label{equ py formula} 
\sum_{l \in \BZ} N_{1}(l)\,q^l &=
\left(\frac{1}{q} \prod_{n\geqslant   1}\frac{1}{(1-q^n)^{24}}\right) 
\left(q \frac{d}{dq}G_2(q)\right) \\  
&=1+30q+480q^2+5460q^3+ 49440 q^4 + 378420 q^5 +2540160 q^6+ \cdots , \nonumber \\ \nonumber
\sum_{l \in \BZ} N'(l)\,q^l & =
\left(\frac{1}{q} \prod_{n\geqslant   1}\frac{1}{(1-q^n)^{24}}\right)
\left( q \frac{d}{dq} G_2 + G_2 + \frac{1}{24} \right) \\ \nonumber
& = 2 + 57q + 880q^{2} + 9735q^{3} + 86160q^{4} + 646850q^{5} + 4269888 q^6+\cdots. \notag
\end{align}
\end{prop}

\subsection{Transport of integrals to Hilbert schemes}\label{sect on transpot}
For the proof of Proposition~\ref{prop:desc M_Sbeta}
we will use the general framework of monodromy operators of Markman \cite{Markman} (see also \cite{OUniversality}) to
transport the integrals to the Hilbert schemes.

Consider the Mukai lattice, which is the lattice $\Lambda = H^{\ast}(S,\BZ)$ endowed with the Mukai pairing
\[ \blangle x , y \brangle := - \int_S x^{\vee} y, \]
where, if we decompose an element $x \in \Lambda$ according to degree as $(r,D,n)$, we write $x^{\vee} = (r,-D,n)$.
%We will also write
%\[ \rk(x) = r, \quad c_1(x) = D, \quad v_2(x) = n. \]
Given a sheaf or a complex of sheaves $E$ on $S$, its Mukai vector is defined by
\[ v(E) := \sqrt{\td_S} \cdot \ch(E) \in \Lambda. \]
Let $M(v)$ be a proper smooth moduli space of stable sheaves on $S$ with Mukai vector $v \in \Lambda$ (where stability is with respect to some fixed polarization).
We assume that there exists a universal family $\BF$ on $M(v) \times S$.
If it does not exist, everything below can be made to work by working with the Chern character $\ch(\BF)$
of a quasi-universal family, see \cite{Markman} or \cite{OUniversality}.
Let $\pi_M, \pi_S$ be the projections to $M(v)$ and $S$.
One has the Mukai morphism $\theta_{\BF} : \Lambda \to H^2(M(v))$
defined by
\[ \theta_{\BF}(x) = \left[ \pi_{M \ast}\left( \ch(\BF) \cdot \sqrt{\td_S} \cdot x^{\vee} \right) \right]_{\deg = 2}, \]
where $[ - ]_{\deg = k}$ stands for extracting the degree $k$ component
and where (as we will also do below) have suppressed the pullback maps from the projection to $S$.
The morphism restricts to an lattice isometry
\begin{equation} \theta_{\BF} : v^{\perp} \to H^2(M(v),\BZ) \label{lattice iso} \end{equation}
where on the right we consider the Beauville-Bogomolov-Fujiki form.
Define the universal class
\[ u_v = \exp\left( \frac{ \theta_{\BF}(v) }{\blangle v,v \brangle} \right) \ch(\BF) \sqrt{\td_S}, \]
which is independent of the choice of universal family $\BF$.
For $x \in \Lambda$, consider the normalized descendents:
\[ B(x) := \pi_{M\ast}( u_v \cdot x^{\vee} ), \]
and let $B_k(x) = [ B(x) ]_{\deg=2k}$ its degree $2k$ component.
\begin{example} \label{example:B Hilb}
For $v=(1,0,1-d)$, the moduli space becomes the punctual Hilbert scheme: $M(v) = S^{[n]}$. % the Hilbert scheme of $n$ points on $S$.
%\begin{YC}Why call it  punctual Hilbert scheme? \end{YC}
%\begin{GO}Because you can also consider the Hilbert scheme of curves on a K3 surface \end{GO}
Then we have
\[ u_v = \exp\left( \frac{-\delta}{2d-2} \right) \ch( \CI_{\CZ} ) \sqrt{\td_S}, \]
where we let $\delta = \pi_{\ast} \ch_3( \CO_{\CZ} )$ (so that $-2 \delta$ is the class of the locus of non-reduced subschemes).

We define the standard descendents on the Hilbert scheme by
\[ \fG_d(\alpha) = \pi_{\Hilb\ast}( \pi_S^{\ast}(\alpha) \ch_d(\CO_{\CZ}) ) \in H^{\ast}(S^{[d]}), \]
where $\alpha \in H^{\ast}(S)$. One obtains that
\begin{align*}
B_1(\pt) & = - \frac{\delta}{2d-2}, \\
B_2(\pt) & = \frac{1}{2} \frac{\delta^2}{(2d-2)^2} - \fG_2(\pt). 
\end{align*}
For a divisor $D \in H^2(S)$ one finds
\begin{align*}
B_1(D) & = \fG_2(D), \\
B_2(D) & = \fG_3(D) - \frac{\delta}{2d-2} \fG_2(D).
\end{align*}
And for the unit,
\begin{align*}
B_1(1) & = - \frac{1}{2} \delta, \\
B_2(1) & = \frac{3}{4} \frac{\delta^2}{2d-2} - \fG_2(\pt) - \fG_4(1).
\end{align*}
\end{example}

\begin{example} \label{example:B Mbeta}
Let $\beta \in \Pic(S)$ be an effective class of square $\beta \cdot \beta = 2d-2$.
For the Mukai vector $v=(0,\beta,1)$ the moduli space is $M(v) = M_{S,\beta}$.
Let $\BF$ be the normalized universal family, i.e. which satisfies $\det \dR \pi_{M\ast} \BF = \CO$.
For $\alpha \in H^{\ast}(S)$, we define as before the descendents
\[ \sigma_d(\alpha) = \pi_{M\ast}( \pi_S^{\ast}(\alpha) \ch_d(\BF) ). \]
By the normalization condition and GRR, we have:
\begin{equation} c_1(\dR \pi_{M\ast} \BF) = \sigma_3(1) + 2 \sigma_1(\pt) = 0. \label{sigma13vanish} \end{equation}
Moreover, by a direct computation, one also has:
\begin{align*}
B_1(\pt) & = \sigma_1(\pt), \\
\theta_{\BF}(v) & = -\sigma_2(\beta) + \sigma_1(\pt),
\end{align*}
Using the vanishing $c_1( \dR \pi_{M\ast} \BF) = 0$ again yields
\[
B_1(1+\pt) = 
\left[ \ch( \dR \pi_{M\ast} \BF ) \exp\left( \frac{\theta_{\BF}(v)}{2d-2} \right) \right]_{\deg 2}
= \frac{\theta_{\BF}(v)}{2d-2}
= \frac{1}{2d-2} ( \sigma_1(\pt) - \sigma_2(\beta) ). \]
This shows
\begin{align*} 
\sigma_2(\beta) & = \sigma_1(\pt) - (2d-2) B_1(1 + \pt), \\
\theta_{\BF}(v) & = (2d-2) B_1(1 + \pt).
\end{align*}
By rewriting the $B$'s in terms of the $\sigma$'s using the formulae above
and then inverting the relation, we obtain for all $D \in H^2(S)$ by a straightforward calculation the following:
\begin{align*}
\sigma_1(\pt) & = B_1(\pt), \\
\sigma_2(D) & = - (D \cdot \beta) B_1(1 + \pt) - B_1(D),
\end{align*}
and
%From this it follows that $B_1(D) = - B_1(1+\pt) (D \cdot \beta) - \sigma_2(D)$.
\begin{align*}
\sigma_2(\pt) & = B_2(\pt) - B_1(\pt) B_1(1 + \pt), \\
\sigma_3(D) & = - B_2(D) + B_1(D) B_1(1 + \pt) + \frac{1}{2} (D \cdot \beta) B_1(1+\pt)^2, \\
\sigma_4(1) & = B_2(1 - \pt) + 2 B_1(\pt) B_1(1 + \pt) - \frac{1}{2} B_1(1+\pt)^2 \\
& = B_2(1 - \pt) - \frac{1}{2} B_1( 1 + \pt ) B_1( 1 - 3 \pt).
\end{align*}
\end{example}
For later, we also record some pairings
with respect to the Beauville-Bogomolov-Fujiki form:
\begin{lemma} \label{lemma:intersection pairing}
\[ \sigma_3(1) \cdot \sigma_1(\pt) = 0, \quad \sigma_1(\pt) \cdot \sigma_2(D) = D \cdot \beta,
\quad \sigma_3(1) \cdot \sigma_2(D) = -2 (D \cdot \beta). \]
\end{lemma}
\begin{proof}
By Eqn.~\eqref{sigma13vanish} and since \eqref{lattice iso} is an isometry 
and moreover $B_1( - )|_{v^{\perp}} = \theta_{\BF}|_{v^{\perp}}$ we have
\[ \sigma_3(1) \cdot \sigma_1(\pt) = -2 \sigma_1(\pt) \cdot \sigma_1(\pt) = -2 B_1(\pt) \cdot B_1(\pt) = 0. \]
Similarly,
\[ \sigma_1(\pt) \cdot \sigma_2(D) = B_1(\pt) \cdot ( - (D \cdot \beta) B_1(1 + \pt) - B_1(D) )
= - (D \cdot \beta) (\pt \cdot 1) = D \cdot \beta. \]
The last one uses again Eqn.~\eqref{sigma13vanish}.
\end{proof}
%where the last one follows from $B(1) = \exp( B_1(1 + \pt) ) \pi_{\ast}( \ch(\BF) \sqrt{\td} )$ and then
%\[ B_2(1) = \frac{1}{2} B_1(1 + \pt)^2 - B_1(1 + \pt) B_1(\pt) + \sigma_4(1) + \sigma_2(\pt). \]
\begin{comment}
One obtains that:
\begin{align*}
B_1(\pt) & = \sigma_1(\pt), \\
B_1(D) & = -\sigma_2(D) \quad \text{for all } D \in H^2(S) \text{ with } D \cdot \beta = 0, \\
B_1(1) 
% & = \left[ \exp( \theta_{\BF}(0,\beta,1)/(2d-2) ) (1 + \sigma_3(1) + \sigma_1(\pt)) \right]_{\deg=2}\\
& = \frac{\theta_{\BF}(v)}{2d-2}  + \sigma_3(1) + \sigma_1(\pt) \\
& = \frac{\theta_{\BF}(v)}{2d-2}  - \sigma_1(\pt) \\
& = \frac{1}{2d-2}( -\sigma_2(\beta) + \sigma_1(\pt) ) - \sigma_1(\pt).
\end{align*}
Inverting we obtain:
\begin{align*}
%\sigma_2(\beta) & = -(2d-2) B_1(1 + \pt) + B_1(\pt) \\
%& = -(2d-2) B_1(1 + \pt) - B_1(\beta) \\
\sigma_2(D) & = - (D \cdot \beta) B_1(1 + \pt) - B_1(D) \ \text{ for all } D \in H^2(S), \\
\theta_{\BF}(v) & = (2d-2) B_1( 1 + \pt ).
\end{align*}
\end{comment} 
Using the descendents $B_k(x)$, one allows to move between any two moduli spaces of stable sheaves on $S$
just by specifying a Mukai lattice isomorphism $g : \Lambda \otimes \BQ \to \Lambda \otimes \BQ$.
We give the details in the case of our interest, 
see \cite{Markman, OUniversality} for the general case.

We want to connect the moduli spaces
\[ M_{S,\beta} \,\, \rightsquigarrow \,\, S^{[d]}. \]
Define the isomorphism $g : \Lambda \otimes \BQ \to \Lambda \otimes \BQ$ by
\[
1 \mapsto (0,0,1), \quad \pt \mapsto (1,-\beta,d-1), \quad (0,D,0) \mapsto (0, D, - (D \cdot \beta) ), \]
for all $D \in H^2(S,\BZ)$.
The isomorphism was constructed so that
\[
(0, \beta, 1) \mapsto (1,0,1-d), \quad 1 \mapsto (0,0,1), \quad (2d-2,\beta,0) \mapsto (0,\beta,0),
\quad g|_{ \{1, \beta, \pt \}^{\perp}} = \id,
\]
which shows that it is a lattice isomorphism.
Then one has:
\begin{thm}$($Markman \cite[Thm.~1.2]{Markman}, reformulation as in \cite[Thm.~4]{OUniversality}$)$ \label{thm:Markman} For any $k_i \geqslant   0$, $\alpha_i \in H^{\ast}(S)$ and any polynomial $P$,
\[
\int_{ M_{S,\beta} } P( B_{k_i}(\alpha_i) , c_j( T_{M_{S,\beta}} ) ) 
=
\int_{ S^{[d]} } P( B_{k_i}(g \alpha_i) , c_j( T_{S^{[n]}}  ) ). 
\]
\end{thm}

\subsection{Proof of Proposition~\ref{prop:desc M_Sbeta}}
Let $\beta \in H_2(S,\BZ)$ be an effective curve class
with $\beta^2 = 2d-2$.
We begin with the first evaluation. The strategy is to use Theorem~\ref{thm:Markman}
and the formulae given in Examples~\ref{example:B Hilb},~\ref{example:B Mbeta}
to move between the standard descendents and Markman's $B$-classes.
We obtain:
\begin{align*}
& \int_{M_{S,\beta}} c_{2d-2}(T_{M_{S,\beta}}) \sigma_2(\pt) \\
= & \int_{M_{S,\beta}} c_{2d-2}(T_{M_{S,\beta}}) ( B_2(\pt) - B_1(\pt) B_1(1 + \pt) ) \\
= & \int_{S^{[d]}} c_{2d-2}(T_{S^{[d]}}) ( B_2(1 - \beta + (d-1)\pt) - B_1(1 - \beta + (d-1)\pt) B_1( 1 - \beta + d \pt) ). 
\end{align*}
Observe that:
\begin{align*}
B_1(1 - \beta + (d-1)\pt) & = - \delta - \fG_2(\beta), \\
B_1( 1 - \beta + d \pt ) & = - \delta - \fG_2(\beta) - \frac{ \delta}{2d-2}, \\
B_2(1 - \beta + (d-1)\pt) & = \frac{\delta^2}{2d-2} - d \fG_2(\pt) - \fG_4(1) 
- \left(\fG_3(\beta) - \frac{\delta}{2d-2} \fG_2(\beta) \right),
\end{align*}
Using the vanishing in Proposition~\ref{lem on vanishing} (i) we find that
\begin{align*}
& \int_{S^{[d]}} c_{2d-2}(T_{S^{[d]}}) B_2(1 - \beta + (d-1)\pt) \\
= & \int_{S^{[d]}} c_{2d-2}(T_{S^{[d]}}) \left(  \frac{\delta^2}{2d-2} - d \fG_2(\pt) - \fG_4(1) \right)\\
= & - C(c_{2d-2}(T_{S^{[d]}} )) 
- \int_{S^{[d]}} c_{2d-2}(T_{S^{[d]}}) (d \fG_2(\pt) + \fG_4(1)),
\end{align*}
as well as
\[\int_{S^{[d]}} c_{2d-2}(T_{S^{[d]}}) B_1(1 - \beta + (d-1)\pt) B_1( 1 - \beta + d \pt)
= - C(c_{2d-2}(T_{S^{[d]}} )). \]
Let us write
\[ A_d = \int_{M_{S,\beta}} c_{2d-2}(T_{M_{S,\beta}}) \sigma_2(\pt),\]
which is well-defined since the above shows that the right hand side only depends on $d$.
Taking generating series
and using the evaluations
of descendents on $S^{[d]}$
(in particular, the expression \eqref{sstwet} and the differential equation \eqref{diffeqn}), we conclude
\begin{align*}
\sum_{d} A_d q^d 
& =
- \sum_{d \geqslant   0} q^d \int_{S^{[d]}} c_{2d-2}(T_{S^{[d]}}) (d \fG_2(\pt) + \fG_4(1)) \\
& = 
- q \frac{d}{dq}\left( M(q) ( 1/24 + G_2 ) \right)
+ M(q) \left( 2 q \frac{d}{dq} G_2 + \left( \frac{1}{24} + G_2 \right)(1 + 24 G_2) \right) \\
& = M(q) q \frac{d}{dq} G_2,
\end{align*}
where we denote $M(q) = \prod_{n \geqslant   1} (1-q^n)^{-24}$.
This proves the first evaluation (after shifting the generating series by $q$).

For the second case, one argues similarly, and obtains
\[
\int_{M_{S,\beta}} c_{2d-2}(T_{M_{S,\beta}}) \sigma_3(D)
=
- (D \cdot \beta) \left(
\int_{S^{[d]}} c_{2d-2}(T_{S^{[d]}}) \fG_2(\pt) + C(c_{2d-2}(T_{S^{[d]}})) 
\right).
\]
In the third case, one obtains
\begin{equation*}
\int_{M_{S,\beta}} c_{2d-2}(T_{M_{S,\beta}}) \sigma_4(1)
=
\int_{S^{[d]}} c_{2d-2}(T_{S^{[d]}}) ( \fG_4(1) + (d-1) \fG_2(\pt)). \qed
\end{equation*}

\section{Product of $K3$ surfaces}\label{sect on prod}
In this section, we consider the product of two $K3$ surfaces $S$ and $T$:
\[ X = S \times T. \]
If the curve class $\beta \in H_2(S \times T, \BZ)$
is of non-trivial degree over both $S$ and $T$, then
one can construct two linearly independent cosections,
which imply that the reduced invariants of $X$ in this class vanish.\footnote{Of course, one may work with $2$-reduced invariants but the moduli spaces becomes more difficult to handle. We leave the study of the $2$-reduced theory to a future work.}
Because of that we always take  $\beta$ in the image of the natural inclusion
\[ \iota_{\ast} : H_2(S,\BZ) \hookrightarrow H_2(X,\BZ), \]
where $\iota : S \times \{ t \} \hookrightarrow X$ is the inclusion of a fiber.
In \S \ref{subsec:GW K3xK3}, we first discuss the computations of GW/GV invariants.
Then we completely determine all $\DT_4$ invariants. By comparing them, we prove Conjecture \ref{conj on DT4/GV} for $X=S\times T$. 

\subsection{Gromov-Witten invariants} \label{subsec:GW K3xK3}
For $\beta\in H_2(S,\mathbb{Z})\subseteq H_2(X,\mathbb{Z})$,
by the product formula in Gromov-Witten theory \cite{Bproduct}, the reduced virtual classes satisfy
\begin{equation} \label{product formula}
[ \Mbar_{g,n}(X,\beta) ]^{\text{vir}}
=
\begin{cases}
[ \Mbar_{0,n}(S,\beta) ]^{\text{vir}} \times [T] & \text{ if } g=0 \\
[ \Mbar_{1,n}(S,\beta) ]^{\text{vir}} \times ( c_2(T) \cap [T]) & \text{ if } g=1 \\
\quad \quad\quad 0 & \text{ if } g \geqslant   2.
\end{cases}
\end{equation}

The Gromov-Witten theory of $K3$ surfaces in low genus is well-known.

In genus $0$, one defines BPS numbers $n_{0,\beta}(S)$ by the multiple cover formula
\begin{equation}\label{multiple cover formu}\deg[\overline{M}_{0,0}(S,\beta)]^{\rm{vir}}
=
\sum_{\begin{subarray}{c}k\geqslant   1, k|\beta  \end{subarray}}\frac{1}{k^{3}}\cdot n_{0,\beta/k}(S). 
\end{equation}
By the Yau-Zaslow formula proven by Klemm, Maulik, Pandharipande and Scheidegger \cite{KMPS},
the invariant $n_{0,\beta}(S)$ only depends on the square $\beta^2$.
By the evaluation for primitive curve classes due to Bryan and Leung \cite{BL},
one then has
\begin{equation}\label{def of N0} n_{0,\beta}(S) = N_{0}\left(\frac{\beta^2}{2}\right), \end{equation}
where
\begin{align}\label{defi of N0}\sum_{l\in\mathbb{Z}}N_{0}(l)\, q^l&=\frac{1}{q} \prod_{n\geqslant   1}\frac{1}{(1-q^n)^{24}} \\
&=q^{-1} + 24 + 324 q + 3200 q^2 + 25650 q^3+\cdots. \nonumber \end{align}
In genus $1$, by Pandharipande-Yin \cite[pp.~12,~(8)]{PY}, we have the multiple cover formula 
\begin{equation}\label{py form}\int_{[\overline{M}_{1,1}(S, \beta)]^{\vir}}\mathrm{ev}^*(\pt)
=\sum_{\begin{subarray}{c}k\geqslant   1, k|\beta  \end{subarray}}k\cdot N_{1}\left(\frac{\beta^2}{2k^2}\right), \end{equation}
where $N_{1}(l)$ is defined as in \eqref{equ py formula} of the last section, that is
\begin{align*}
\sum_{l \in \BZ} N_{1}(l)\,q^l &=
\left(\frac{1}{q} \prod_{n\geqslant   1}\frac{1}{(1-q^n)^{24}}\right)
\left(q \frac{d}{dq}G_2(q)\right) \\
&=1+30q+480q^2+5460q^3+ 49440 q^4 + 378420 q^5 + \cdots. 
\end{align*}
We remark that although genus $2$ Gromov-Witten invariants are zero \eqref{product formula}, the corresponding Gopakumar-Vafa invariants are 
nontrivial (Proposition \ref{prop on GW for prod}).

\subsection{Gopakumar-Vafa invariants}
%We determine the Gopakumar-Vafa invariants.
Let $\gamma,\gamma'\in H^{4}(X)$ be cohomology classes and
\begin{align*}\gamma &=A_1\cdot 1\otimes \pt+D_1\otimes D_2+A_2\cdot \pt\otimes 1, \\
\gamma' &=A'_1\cdot 1\otimes \pt+D'_1\otimes D'_2+A'_2\cdot \pt\otimes 1
\end{align*}
be their decompositions under the K\"unneth isomorphism:
$$H^{4}(X)\cong (H^0(S)\otimes H^4(T))\oplus (H^2(S)\otimes H^2(T))\oplus (H^4(S)\otimes H^0(T)). $$
Fix also a curve class
$$\alpha=\theta_1\otimes \pt+\pt\otimes \theta_2\in H^6(X) \cong (H^2(S)\otimes H^4(T))\oplus (H^4(S)\otimes H^2(T)).$$

\begin{prop}\label{prop on GW for prod}
For any effective curve class $\beta\in H_2(S,\mathbb{Z})\subseteq H_2(X,\mathbb{Z})$,
we have
%the genus $0$ Gopakumar-Vafa invariants are:
\begin{align*}
n_{0,\beta}(\gamma, \gamma') &=(D_1\cdot\beta)\cdot (D_1'\cdot\beta)\cdot\int_T(D_2\cdot D_2')\cdot N_{0}\left(\frac{\beta^2}{2}\right), \\
n_{0,\beta}(\alpha)&=(\theta_1\cdot \beta)\,N_{0}\left(\frac{\beta^2}{2}\right). 
\end{align*}
If $\beta$ is primitive, we have
\begin{align*}
n_{1, \beta}(\gamma) = 24\,A_2\,  N_1\left(\frac{\beta^2}{2}\right), \quad n_{2,\beta}= N_2\left( \frac{\beta^2}{2} \right),
\end{align*}
where $N_1(l)$ is defined as in \eqref{equ py formula} and 
\begin{align}\label{defi of N2}\sum_{l\in\mathbb{Z}}N_{2}(l)\, q^l
&=\left(\frac{1}{q} \prod_{n\geqslant   1}\frac{1}{(1-q^n)^{24}}\right) \left( 24 q \frac{d}{dq} G_2 - 24 G_2 - 1 \right) \\
&=
72q + 1920q^{2} + 28440q^{3} + 305280q^{4} + 2639760q^{5} + 19450368q^{6} + \cdots.
%q^{-1} + 24 + 324 q + 3200 q^2 + 25650 q^3+\cdots. 
\nonumber \end{align}
In particular, Conjecture \ref{conj on integrality} holds for $X=S\times T$. 
%where $n_{0,h}(S)$ $($$h\in \mathbb{Z}_{\geqslant  0}$$)$ fits into the generating series 
%\begin{equation}\label{equ on yz for}\sum_{h\geqslant  0}n_{0,h}(S)\, q^h=\prod_{n\geqslant   1}\frac{1}{(1-q^n)^{24}}.  \end{equation}
\end{prop}
%\begin{YC}Is there any relation between $N_{2}(l)$ and genus 2 counts on K3 ? \end{YC}\begin{GO} Good question, I don't know. \end{GO}
\begin{proof} By the divisor equation,
%and since the reduced virtual class is of dimension $0$ on a K3 surface,
we have
\begin{align*}
\blangle \tau_{0}(\gamma)\tau_{0}(\gamma') \brangle^{\GW}_{0,\beta}
&=
(D_1\cdot\beta)\cdot (D_1'\cdot\beta)\cdot\int_T(D_2\cdot D_2')\cdot \deg([\overline{M}_{0, 0}(S, \beta)]^{\rm{vir}}), \\
\blangle \tau_{0}(\alpha)\brangle^{\GW}_{0,\beta}
&=
(\theta_1\cdot\beta)\cdot \deg([\overline{M}_{0, 0}(S, \beta)]^{\rm{vir}}).
\end{align*}
The genus $0$ formula hence follows from Eqn.~\eqref{multiple cover formu} and the Yau-Zaslow formula \eqref{def of N0}.
In genus $1$, the
product formula \eqref{product formula} and Eqn.~\eqref{py form} imply that for any effective class $\beta \in H_2(S,\BZ)$ we have:
\begin{equation} \label{eq:genus1 K3xK3}
\blangle \tau_0(\gamma) \brangle^{\GW}_{1,\beta}=A_2\,e(T)\,\int_{[\overline{M}_{1,1}(S, \beta)]^{\vir}}\mathrm{ev}^*(\pt)
=A_2 \,e(T)\sum_{\begin{subarray}{c}k\geqslant   1, k|\beta  \end{subarray}}k\cdot N_{1}\left(\frac{\beta^2}{2k^2}\right).
\end{equation}
Moreover, by the first part we have
\[ \blangle \tau_0(\gamma) \tau_0(c_2(X)) \brangle^{\GW}_{0,\beta} = 0. \]
Hence taking these formulae for primitive $\beta$ yields the result.

For the genus 2 Gopakumar-Vafa invariant, let $\beta$ be primitive.
Observe that we have
\begin{gather*}
\blangle \varnothing \brangle^{\GW}_{2,\beta}= 0, 
\quad n_{1,\beta}(c_2(X)) = 24^2 N_1(\beta^2/2), 
\quad \blangle \tau_0(c_2(X))^2 \brangle^{\GW}_{0,\beta} = 0.
\end{gather*}
The nodal invariant is computed as follows:
\begin{align*}
N_{\mathrm{nodal},\beta}
& = \frac{1}{2} \left[ \blangle \tau_0(\Delta_X) \brangle^{\GW}_{0,\beta}
- \blangle \tau_1(c_2(T_X)) \brangle^{\GW}_{0,\beta} \right] \\
& = \frac{1}{2} \left[ 24 \int_{[\Mbar_{0,2}(S,\beta)]^{\vir}} (\ev_1 \times \ev_2)^{\ast}(\Delta_S)
- 24 \int_{[ \Mbar_{0,1}(S,\beta) ]^{\vir}} \psi_1 \right] \\
& = \frac{1}{2} \left[ 24 (\beta \cdot \beta) N_{0}(\beta^2/2) 
+ 2 \cdot 24 N_0(\beta^2/2)
\right].
\end{align*}
If $\beta_h$ is a primitive curve class of square $\beta_h^2 = 2h-2$, we conclude:
\begin{align*}
\sum_{h \geqslant  0} N_{\mathrm{nodal},\beta_h} q^{h-1}
& = \frac{1}{2} \left[ 48 q \frac{d}{dq}\left( \frac{1}{\Delta(q)} \right) + 48 \frac{1}{\Delta(q)} \right] \\
& = 24^2 G_2(q) \frac{1}{\Delta(q)} + 24 \frac{1}{\Delta(q)},
\end{align*}
where we used
$\Delta(q) = q \prod_{n \geqslant  1} (1-q^n)^{24}$
and the identity (ref. Eqn.~\eqref{diffeqn}):
\[ q \frac{d}{dq}\left(\frac{1}{\Delta(q)} \right) =  \frac{24G_2(q)}{\Delta(q)}. \]
Using the definition of $n_{2,\beta}$, we conclude that:
\begin{align*}
\sum_{h \geqslant  0} n_{2,\beta_h} q^{h-1}
& =
\left( 24 \frac{1}{\Delta(q)} q \frac{d}{dq} G_2(q) \right)
- \left( 24 G_2(q) \frac{1}{\Delta(q)} + \frac{1}{\Delta(q)} \right) \\
& = \frac{1}{\Delta(q)} \left( 24 q \frac{d}{dq} G_2(q) - 24 G_2(q) - 1 \right). 
\end{align*}
This is exactly the desired result.
\end{proof}
%As an immediate consequence, we have integrality of all genus Gopakumar-Vafa invariants.
%\begin{cor}\label{cor on inte on prod}
%For any effective curve class $\beta\in H_2(S,\mathbb{Z})\subseteq H_2(X,\mathbb{Z})$, 
%\end{cor}
 
We will also need the following later (in the appendix):
\begin{lemma} \label{lemma:tau1 K3xK3}
For any effective curve class $\beta\in H_2(S,\mathbb{Z})\subseteq H_2(X,\mathbb{Z})$, we have 
\[\blangle \tau_1(\gamma) \brangle_{0,\beta}^{\GW}
=
-2A_1 \sum_{\begin{subarray}{c}k\geqslant   1, k|\beta  \end{subarray}}\frac{1}{k^{3}}\cdot N_0\left( \frac{(\beta/k)^2}{2} \right).
\]
\end{lemma}
\begin{proof} 
Consider a divisor $D=\pr_1^{\ast}(\alpha) \in H^2(X)$ with $d:=\alpha \cdot \beta \neq 0$.
By Lemma \ref{div equ on GW} and Eqn.~\eqref{product formula} we have
\begin{align*}
%\mathrm{GW}_{0, \beta}(\gamma;\psi)
\blangle \tau_1(\gamma) \brangle_{0,\beta}^{\GW}
&=-\frac{2}{d}\blangle \tau_0(\gamma\cdot D) \brangle_{0,\beta}^{\GW}  \\
&=-\frac{2}{d}\left(A_1\,\blangle\tau_0(\alpha\otimes\pt) \brangle_{0,\beta}^{\GW}+(D_1\cdot\alpha)\,\blangle\tau_0(\pt\otimes D_2) \brangle_{0,\beta}^{\GW}\right) \\
&=-\frac{2}{d}A_1\,\int_{[\overline{M}_{0,1}(S, \beta)]^{\vir}}\mathrm{ev}^{\ast}(\alpha) \\
&=-2A_1\,\deg[\overline{M}_{0,0}(S,\beta)]^{\rm{vir}}.
\end{align*}
By Eqn.~\eqref{multiple cover formu} and the Yau-Zaslow formula \eqref{def of N0}, we obtain the claim.  
\end{proof}

\subsection{$\mathrm{DT_4}$ virtual classes}

The moduli space $M_{\beta}$ of one dimensional stable sheaves on $X$ satisfies (e.g.~\cite[Lem.~2.2]{CMT1}):
\begin{align}\label{iso of mod on prod of k3}M_{\beta}\cong M_{S,\beta}\times T, \end{align}
where $M_{S,\beta}$ is the moduli space of one dimensional stable sheaves $F$ on $S$ with $[F]=\beta$ and $\chi(F)=1$.
By a result of Mukai \cite{M}, $M_{S,\beta}$ is a smooth projective holomorphic symplectic variety of dimension $\beta^2+2$. 
In order to determine the $\DT_4$ virtual class of $M_{\beta}$, we first recall:
\begin{defi}$($\cite[Ex.~16.52,~pp.~410]{Sw}, \cite[Lem.~5]{EG}$)$
Let $E$ be a $\mathrm{SO}(2n,\mathbb{C})$-bundle with a non-degenerate symmetric bilinear form $Q$ on a connected scheme $M$. 
Denote $E_+$ to be its positive real form\,\footnote{This means a real half dimensional subbundle such that $Q$ is real and positive definite on it. By
homotopy equivalence $\mathrm{SO}(m,\mathbb{C})\sim \mathrm{SO}(m,\mathbb{R})$, it exists and is unique up to isomorphisms.}.
The half Euler class of $(E,Q)$ is 
$$e^{\frac{1}{2}}(E,Q):=\pm\,e(E_+)\in H^{2n}(M,\mathbb{Z}), $$
where the sign depends on the choice of orientation of $E_+$. 
\end{defi}

\begin{defi}$($\cite{EG},~\cite[Def.~8.7]{KiP}$)$ \label{def of red vir clas}
Let $E$ be a $\mathrm{SO}(2n,\mathbb{C})$-bundle with a non-degenerate symmetric bilinear form $Q$ on a connected scheme $M$. 
An isotropic cosection of $(E,Q)$ is a map 
$$\phi: E\to \oO_M, $$
such that the composition 
$$\phi\circ \phi^{\vee}: \oO_M\to E^{\vee}\stackrel{Q}{\cong} E \to \oO_M$$
is zero. If $\phi$ is furthermore surjective, we define the (reduced) half Euler class: 
$$e_{\mathrm{red}}^{\frac{1}{2}}(E,Q):=e^{\frac{1}{2}}\left((\phi^{\vee}\oO_M)^{\perp}/(\phi^{\vee}\oO_M),\bar{Q}\right)\in H^{2n-2}(M,\mathbb{Z}), $$
as the half Euler class of the isotropic reduction.
Here $\bar{Q}$ denotes the induced non-degenerate symmetric bilinear form on $(\phi^{\vee}\oO_M)^{\perp}/(\phi^{\vee}\oO_M)$. 
\end{defi}
We show reduced half Euler classes are independent of the choice of surjective isotropic cosection. 
\begin{lem}\label{lem on indep of cosec}
Let $E$ be a $\mathrm{SO}(2n,\mathbb{C})$-bundle with a non-degenerate symmetric bilinear form $Q$ on a connected scheme $M$ and 
$$\phi: E\to \oO_M $$
be a surjective isotropic cosection.
Then we can write the positive real form $E_+$ of $E$ as 
$$E_+=\eE_+\oplus \underline{\mathbb{R}}^2$$
such that 
$$e_{\mathrm{red}}^{\frac{1}{2}}(E,Q)=\pm\,e(\eE_+). $$
Moreover, it is independent of the choice of surjective cosection. 

In particular, when $E=\oO^{\oplus2} \oplus V$ such that $Q=\begin{pmatrix}
0 & 1 \\
1 & 0
\end{pmatrix} \oplus Q|_{V}$, we have  
$$e_{\mathrm{red}}^{\frac{1}{2}}(E,Q)=\pm\,e^{\frac{1}{2}}(V,Q|_{V}).$$
\end{lem}
\begin{proof}
Let $E_-:=\sqrt{-1}\cdot E_+$, then $E=E_+\oplus E_-$. Since $\phi$ is surjective, $\phi^{\vee}$ determines 
a trivial subbundle $\oO_M$ of $E$. In the diagram: 
\begin{align*} \xymatrix{
 \oO_M  \ar[r]^{\phi^{\vee}\quad \quad } \ar[dr]_{ }    & E=E_+\oplus E_- \ar[d]^{\pi_\pm}  \\
  &  E_\pm,    }
\end{align*}
for $v\in \phi^{\vee}(\oO_M)$, write $v=v_++v_-$ based on above decomposition.  
The isotropic condition gives 
$$0=Q(v,v)=Q(v_+,v_+)+2\,Q(v_+,v_-)+Q(v_-,v_-). $$
If $v_+=0$, then $Q(v_-,v_-)=0$ which implies $v_-=0$ as $Q$ on $E_-$ is negative definite. Therefore 
the composition $\pi_{\pm}\circ\phi^{\vee}$ determines a trivial subbundle $\underline{\mathbb{R}}\subset E_\pm$. 

We write $(\phi^{\vee}\oO_M)^{\perp}=V_+\oplus V_-$ for $V_{\pm}=E_{\pm}\bigcap\,(\phi^{\vee}\oO_M)^{\perp}$, which fits in the diagram
\begin{align*} \xymatrix{
 \phi^{\vee}(\oO_M)  \ar[r]^{\subset \quad } \ar[d]_{=}  & (\phi^{\vee}\oO_M)^{\perp}  \ar[d]_{=} \ar[r]^{\quad \subset} & E \ar[d]_{=}  \\
\underline{\mathbb{R}}\oplus \underline{\mathbb{R}} & V_+\oplus V_- &  E_+\oplus E_-.    }
\end{align*}
Then $\rank_{\mathbb{R}}V_++\rank_{\mathbb{R}}V_-=4n-2$ and $\rank_{\mathbb{R}}V_{\pm}\leqslant\rank_{\mathbb{R}}E_{\pm}$. 
As $(\phi^{\vee}\oO_M)^{\perp}/(\phi^{\vee}\oO_M)$ has 
an induced non-degenerate symmetric bilinear form, so 
$$\rank_{\mathbb{R}}V_+=\rank_{\mathbb{R}}V_-=2n-1. $$
Let $\eE_+:=V_+/\underline{\mathbb{R}}$, by the metric $Q|_{V_+}$ on $V_+$, we may write 
\begin{align}\label{deco V+}V_+=\eE_+\oplus \underline{\mathbb{R}}. \end{align}
Under the identification $E^{\vee}\stackrel{Q}{\cong} E$, we have 
\begin{align*}
\Ker(\phi)&=\big\{v\in E\,|\,\phi(Q(v,-))=0 \big\}  \\
&=\big\{v\in E\,|\,Q(v,\phi^{\vee}\oO_M)=0 \big\} \\
&=(\phi^{\vee}\oO_M)^{\perp}.
\end{align*}
Therefore $E/(\phi^{\vee}\oO_M)^{\perp}\stackrel{\phi}{\cong}\oO_M$. By using the metric $Q|_{E_+}$ on $E_+$, we may write 
$$E_+=V_+\oplus \underline{\mathbb{R}}.$$
Combining with Eqn.~\eqref{deco V+}, we have 
$$E_+=\eE_+\oplus \underline{\mathbb{R}}^2. $$
By definition, the reduced half Euler class is the Euler class of $\eE_+$.

Given two surjective cosections $\phi_1, \phi_2$, if $\phi^{\vee}_1\oO_M=\phi^{\vee}_2\oO_M$, then the bundle $\eE_+$ they determine are the same, so are the reduced half Euler classes. 
%If $\phi^{\vee}_1\oO_M\neq \phi^{\vee}_2\oO_M$, it is easy to see the corresponding $\eE_+$ has a trivial subbundle 
%$\underline{\mathbb{R}}$, so both reduced half Euler classes vanish. 
If $\phi^{\vee}_1\oO_M\neq \phi^{\vee}_2\oO_M$, we divide into two cases: 
(1) when $\phi^{\vee}_2\oO_M\subseteq (\phi^{\vee}_1\oO_M)^{\perp}$ (which automatically implies $\phi^{\vee}_1\oO_M\subseteq (\phi^{\vee}_2\oO_M)^{\perp}$),
it is easy to see the corresponding $\eE_+$ has a trivial subbundle 
$\underline{\mathbb{R}}$, so both reduced half Euler classes vanish. 
(2) when $\phi^{\vee}_2\oO_M\nsubseteq (\phi^{\vee}_1\oO_M)^{\perp}$ (hence also $\phi^{\vee}_1\oO_M\nsubseteq (\phi^{\vee}_2\oO_M)^{\perp}$), then we have 
$$E\cong (\phi^{\vee}_1\oO_M)^{\perp}\oplus \phi^{\vee}_2\oO_M\cong (\phi^{\vee}_2\oO_M)^{\perp}\oplus \phi^{\vee}_1\oO_M. $$
Taking quotient by $\phi^{\vee}_1\oO_M\oplus \phi^{\vee}_2\oO_M$, we obtain 
$$(\phi^{\vee}_1\oO_M)^{\perp}/\phi^{\vee}_1\oO_M \cong (\phi^{\vee}_2\oO_M)^{\perp}/\phi^{\vee}_2\oO_M, $$
whose half Euler classes are the same. 
Therefore we know the reduced half Euler class
is independent of the choice of surjective isotropic cosection. 
The last statement when $E=\oO^{\oplus2} \oplus V$ such that 
$Q=\begin{pmatrix}
0 & 1 \\
1 & 0
\end{pmatrix} \oplus Q|_{V}$
follows from this.
\end{proof}
Recall a $\mathrm{Sp}(2r,\mathbb{C})$-bundle (or symplectic vector bundle) is a complex vector bundle of rank $2r$ 
with a non-degenerate anti-symmetric bilinear form.
One class of quadratic vector bundles is given by tensor product of two symplectic vector bundles $V_1, V_2$. 
Their half Euler classes
can be computed using Chern classes of $V_1,V_2$. 
For our purpose, we restrict to the following case.
\begin{lem}\label{lem on compu of half euler class}
Let $(V_1,\omega_1)$, $(V_2,\omega_2)$ be a $\mathrm{Sp}(2r,\mathbb{C})$ $($resp.~$\mathrm{Sp}(2,\mathbb{C})$-bundle$)$ on a connected scheme $M$. 
Then
$$(V_1\otimes V_2,\omega_1\otimes \omega_2)$$ 
defines a $SO(4r,\mathbb{C})$-bundle whose half Euler class satisfies 
$$e^{\frac{1}{2}}(V_1\otimes V_2,\omega_1\otimes \omega_2)=\pm\,\big(e(V_1)-c_{2r-2}(V_1)\cdot e(V_2)\big). $$
\end{lem}
\begin{proof}
Consider the universal $\mathrm{Sp}(2r,\mathbb{C})$-bundle $\mathcal{V}_1$ (resp.~$\mathrm{Sp}(2,\mathbb{C})$-bundle $\mathcal{V}_2$) 
on the classifying space $\mathrm{BSp}(2r,\mathbb{C})$ (resp.~$\mathrm{BSp}(2,\mathbb{C})$). 
Their tensor product gives a $\mathrm{SO}(4r,\mathbb{C})$-bundle on $\mathrm{BSp}(2r,\mathbb{C})\times \mathrm{BSp}(2,\mathbb{C})$,
whose half Euler class is denoted by $e^{\frac{1}{2}}(\mathcal{V}_1\otimes \mathcal{V}_2)$. 

By the property of half Euler class (e.g.~\cite[Prop.~2]{EG}):
\begin{align*}e^{\frac{1}{2}}(\mathcal{V}_1\otimes \mathcal{V}_2)^2=e(\mathcal{V}_1\otimes \mathcal{V}_2)
=c_{2r}(\mathcal{V}_1)^2-2c_{2r}(\mathcal{V}_1)c_{2r-2}(\mathcal{V}_1)c_2(\mathcal{V}_2), \end{align*}
where we use the fact that the odd Chern classes of $\mathcal{V}_i$ vanish in the second equality. 
Note that above expression is the same as the square of $c_{2r}(\mathcal{V}_1)-c_{2r-2}(\mathcal{V}_1)c_2(\mathcal{V}_2)$. 
Since $H^*(\mathrm{BSp}(2r,\mathbb{C})\times \mathrm{BSp}(2,\mathbb{C}))$ is the tensor product of two polynomial rings 
(e.g.~\cite[Thm.~16.10]{Sw}), hence it is an integral domain. Therefore 
\begin{align*}
e^{\frac{1}{2}}(\mathcal{V}_1\otimes \mathcal{V}_2)=\pm\,(c_{2r}(\mathcal{V}_1)-c_{2r-2}(\mathcal{V}_1)c_2(\mathcal{V}_2)).
\end{align*}
Since this construction is universal, we are done. 
\end{proof}
Finally, we can determine the (reduced) virtual class of $M_{\beta}$. 
\begin{thm}\label{thm on vir clas}
For certain choice of orientation, we have 
\begin{equation}\label{vir class StimesT}
[M_{\beta}]^{\mathrm{vir}}=e(M_{S,\beta})\cdot[T]-e(T)\cdot c_{\beta^2}(M_{S,\beta}). \end{equation}
\end{thm}
\begin{proof}
Under the isomorphism \eqref{iso of mod on prod of k3}: 
\begin{align*}M_{\beta}\cong M_{S,\beta}\times T, \end{align*}
a universal family $\mathbb{F}$ of $M_{\beta}$ satisfies 
\begin{align}\label{equ of univ sheaf on prod}\mathbb{F}=\mathbb{F}_S\boxtimes \oO_{\Delta_T}, \end{align}
where $\mathbb{F}_S$ is a universal sheaf of $M_{S,\beta}$ and $\Delta_T$ denotes the diagonal in $T\times T$.

Then the obstruction sheaf of $M_{\beta}$:
$$\eE xt_{\pi_{M_{\beta}}}^2(\mathbb{F},\mathbb{F})\cong \eE xt_{\pi_{M_{S,\beta}}}^2(\mathbb{F}_S,\mathbb{F}_S)\oplus
\eE xt_{\pi_{M_{S,\beta}}}^1(\mathbb{F}_S,\mathbb{F}_S)\boxtimes T_T\oplus \eE xt_{\pi_{M_{S,\beta}}}^0(\mathbb{F}_S,\mathbb{F}_S)\boxtimes\wedge^2T_T$$
is a vector bundle with two trivial subbundles $\eE xt_{\pi_{M_{S,\beta}}}^2(\mathbb{F}_S,\mathbb{F}_S)$, and $\eE xt_{\pi_{M_{S,\beta}}}^0(\mathbb{F}_S,\mathbb{F}_S)\boxtimes\wedge^2T_T$. By Lemmata \ref{lem on indep of cosec}, \ref{lem on compu of half euler class}, we are done.
\end{proof}

\subsection{$\DT_4$ invariants and proof of conjectures}
In this section, we determine all $\DT_4$ invariants of $S \times T$.
%needed to prove Conjecture \ref{conj on DT4/GV}.
Let $\gamma,\gamma'\in H^{4}(X)$ be cohomology classes and decompose them as
$$\gamma=A_1\cdot 1\otimes \pt+D_1\otimes D_2+A_2\cdot \pt\otimes 1, $$
$$\gamma'=A'_1\cdot 1\otimes \pt+D'_1\otimes D'_2+A'_2\cdot \pt\otimes 1, $$
according to the K\"unneth decomposition:
$$H^{4}(X)\cong (H^0(S)\otimes H^4(T))\oplus (H^2(S)\otimes H^2(T))\oplus (H^4(S)\otimes H^0(T)). $$
Fix also a divisor class
\[ \theta=\theta_1+\theta_2\in H^2(X) \cong H^2(S)\oplus H^2(T), \]
and a curve class
$$\alpha=\theta_1\otimes \pt+\pt\otimes \theta_2\in H^6(X) \cong (H^2(S)\otimes H^4(T))\oplus (H^4(S)\otimes H^2(T)).$$

\begin{thm} \label{thm:K3xK3 DT invariants}
Let $\beta\in H_2(S,\mathbb{Z})\subseteq H_2(X,\mathbb{Z})$ be any effective curve class.
With respect to the choice of orientation \eqref{vir class StimesT}, we have
\begin{align}
\blangle\tau_0(\alpha)\brangle^{\DT_4}_{\beta}& =(\theta_1\cdot\beta)\, N_0\left( \frac{\beta^2}{2} \right), \tag{i} \\[5pt]
\blangle\tau_1(\gamma)\brangle^{\DT_4}_{\beta}& =A_1\,N_0\left( \frac{\beta^2}{2} \right)-A_2\,e(T)\,N_{1}\left(\frac{\beta^2}{2}\right), \tag{ii} \\[5pt]
\blangle\tau_2(\theta) \brangle^{\DT_4}_{\beta} & = ( \theta_1 \cdot \beta)\,N''\left( \frac{\beta^2}{2} \right), \tag{iii} \\[5pt]
\blangle\tau_3(1) \brangle^{\DT_4}_{\beta} & = N''\left( \frac{\beta^2}{2} \right), \tag{iv} \\[5pt]
\blangle\tau_0(\gamma),\tau_0(\gamma') \brangle^{\DT_4}_{\beta}& =(D_1\cdot\beta)\cdot (D_1'\cdot\beta)\cdot \left(\int_T D_2\cup D_2'\right)\cdot N_0\left( \frac{\beta^2}{2} \right), \tag{v} \\[5pt]
\blangle\tau_0(\gamma), \tau_1(\theta)\brangle^{\DT_4}_{\beta}&= 
(D_1\cdot\beta)\left(\int_T D_2\cup\theta_2\right)N_0\left( \frac{\beta^2}{2} \right)
 -24A_2 (\theta_1 \cdot \beta) N_1\left( \frac{\beta^2}{2} \right), \tag{vi} \\[5pt]
\blangle\tau_0(\gamma), \tau_2(1)\brangle^{\DT_4}_{\beta}&= 0,  \tag{vii} \\[5pt] 
\blangle\tau_1(\theta), \tau_2(1)\brangle^{\DT_4}_{\beta}&= 48 (\theta_1 \cdot \beta) N_1\left( \frac{\beta^2}{2} \right),  \tag{viii}  
\end{align}
where $N_0(l)$, $N_1(l)$ are defined in Eqns.~\eqref{defi of N0}, \eqref{equ py formula} respectively and 
\begin{align*} \sum_{l \in \BZ} N''(l)\,q^l & = 
\frac{1}{q} \prod_{n \geqslant   1} (1-q^n)^{-24} \left( 24q \frac{d}{dq} G_2(q)+24G_2(q)-1\right) \\
& = -2 q^{-1} + 720q + 14720q^{2} + 182340q^{3} + 1715328q^{4} + \cdots. 
%&= -2 + 720 q^2 + 14720 q^3 + 182340 q^4 + 1715328 q^5 + 13376960 q^6+\cdots
\end{align*}
\end{thm}

The above theorem immediately implies the following:

\begin{cor}\label{verify conjs}
Conjecture \ref{conj on DT4/GV} holds for the product $X = S \times T$
and $\beta\in H_2(S,\mathbb{Z})\subseteq H_2(X,\mathbb{Z})$.
%Moreover, equality \eqref{equ on impri DT4/GV} governing the imprimitive genus $1$ case also holds.
\end{cor}
\begin{proof}
This follows by inspection using Theorem~\ref{thm:K3xK3 DT invariants} on $\DT_4$ invariants and Proposition~\ref{prop on GW for prod} and Lemma~\ref{lemma:tau1 K3xK3} for the GV/GW invariants respectively.
\end{proof}

Another remarkable consequence of Theorem~\ref{thm:K3xK3 DT invariants}
is that all $\DT_4$ invariants of $S \times T$ depend upon the curve class $\beta$
only via the square $\beta^2$ and not the divisibility.
More precisely, given pairs $(S,\beta)$ and $(S',\beta')$ of a $K3$ surface and an effective curve class
such that
$\beta^2 = \beta'^2$, let
\[ \varphi : H^2(S,\BR) \to H^2(S',\BR) \]
be any {\em real} isometry such that $\varphi(\beta) = \beta'$.
Extend $\varphi$ to the full cohomology by setting $\varphi(1) = 1$ and $\varphi(\pt_S) = \pt_{S'}$ where $\pt_S \in H^4(S,\BZ)$ is the point class.

\begin{cor} With respect to the choice of orientation \eqref{vir class StimesT}, we have
\[
\blangle\tau_{k_1}(\gamma_1), \ldots, \tau_{k_n}(\gamma_n) \brangle^{\DT_4, S \times T}_{\beta}
=
\blangle\tau_{k_1}( (\varphi \otimes \id)\gamma_1), \ldots, \tau_{k_n}( (\varphi \otimes \id) \gamma_n) \brangle^{\DT_4, S' \times T}_{\beta'}
\]
for any $\gamma_i \in H^{\ast}(X)$ and $k_i \geqslant   0$.
\end{cor}
This raises the question whether a similar independence of the divisibility holds for
Donaldson-Thomas invariants of holomorphic symplectic 4-folds more generally.

\subsection{Proof of Theorem~\ref{thm:K3xK3 DT invariants}}
We split the proof in two parts.
\begin{proof}[Proof of Theorem~\ref{thm:K3xK3 DT invariants} Part (i), (v)]
%Here we prove for the case $\gamma, \gamma'\in H^{4}(X, \mathbb{Z})$, other cases follow from easier versions of the same argument.
We begin with part (v).
By Eqn.~\eqref{equ of univ sheaf on prod}, the primary insertion becomes
$$\tau_0(\gamma)=(D_1\cdot\beta)\otimes D_2+A_2\,\pi_{M_{S,\beta}*}(\pi_S^*\pt\cdot \ch_1(\mathbb{F}_S))\otimes 1, $$
where $\pi_S$, $\pi_{M_{S,\beta}}$ are projections to each factor of $S\times M_{S,\beta}$. Therefore 
$$\tau_0(\gamma)\cdot\tau_0(\gamma')=(D_1\cdot\beta)\cdot (D_1'\cdot\beta)\otimes (D_2\cdot D_2')+A_2A_2'\,\left(\pi_{M_{S,\beta}*}\left(\pi_S^*\pt\cdot \ch_1(\mathbb{F}_S)\right)\right)^2\otimes 1+\mathrm{others}, $$
where ``others'' lie in $H^2(M_{S,\beta})\otimes H^2(T)$. 
Combining with Theorem \ref{thm on vir clas}, we get 
\begin{align}\label{equ on tau0 prod}
\blangle\tau_0(\gamma),\tau_0(\gamma') \brangle^{\DT_4}_{\beta}=&(D_1\cdot\beta)\cdot (D_1'\cdot\beta)\cdot e(M_{S,\beta})\int_T(D_2\cdot D_2') \\  \nonumber
&-A_2\,A_2'\,e(T)\,\int_{M_{S,\beta}}\left(\pi_{M_{S,\beta}*}\left(\pi_S^*\pt\cdot \ch_1(\mathbb{F}_S)\right)\right)^2\cdot c_{\beta^2}(M_{S,\beta}). 
\end{align}
There exists a Hilbert-Chow map 
$$\mathrm{HC}: M_{S,\beta}\to | \beta|=\mathbb{P}^{\frac{1}{2}\beta^2+1}, $$ 
to the linear system $|\beta|$ and $\ch_1(\mathbb{F}_S)=(\id_S\times\mathrm{HC})^*[\mathcal{C}]$, where $\mathcal{C}$ is the universal curve of the 
linear system:
\begin{align*} \xymatrix{
    \mathcal{C}  \hookrightarrow  S\times |\beta| \ar[d]_{p} \ar[r]^{\quad\quad q} & |\beta| \\
  S. &  &  }
\end{align*}
Since $[\mathcal{C}]=p^*\beta+q^*h$ for the hyperplane class $h$ of $|\beta|$, we have 
\begin{align*}
\pi_{M_{S,\beta}*}\left(\pi_S^*\pt\cdot \ch_1(\mathbb{F}_S)\right)&=\pi_{M_{S,\beta}*}\left(\pi_S^*\pt\cdot (\id_S\times\mathrm{HC})^*[\mathcal{C}]\right)\\ 
&=\mathrm{HC}^*q_*([\cC]\cdot p^*\pt) \\
&=\mathrm{HC}^*(h).
\end{align*}
By Theorem \ref{fujiki result}, we have  
\begin{align*}
& \quad \int_{M_{S,\beta}}\left(\pi_{M_{S,\beta}*}\left(\pi_S^*\pt\cdot \ch_1(\mathbb{F}_S)\right)\right)^2\cdot c_{\beta^2}(M_{S,\beta}) \\
&=C(c_{\beta^2}(M_{S,\beta}))\cdot \mathsf{q}(\mathrm{HC}^*(h)) \\
&=C(c_{\beta^2}(M_{S,\beta}))\,\left(C(1)^{-1}\cdot \int_{M}\left(\mathrm{HC}^*(h^{\beta^2+2})\right)\right)^{\frac{1}{\beta^2+2}} \\
&=0.
\end{align*}
Therefore Eqn.~\eqref{equ on tau0 prod} becomes 
\begin{equation*}\blangle\tau_0(\gamma),\tau_0(\gamma') \brangle^{\DT_4}_{\beta}=(D_1\cdot\beta)\cdot (D_1'\cdot\beta)\cdot (D_2\cdot D_2')\cdot e(M_{S,\beta}).   \end{equation*}
Finally, since $M_{S,\beta}$ is deformation equivalent to $S^{[d]}$ ($\beta^2=2d-2$) (e.g.~\cite[Cor.~3.5,~pp.~136]{Y2}),
they have the same Euler numbers:
\[ e(M_{S,\beta}) = e(S^{[d]}), \]
which is given by $N_0(\beta^2/2)$ due to G\"{o}ttsche \cite{Gott1}. This proves (v). 

For (i), we similarly have 
$$\tau_0(\alpha)=\pi_{M_{S,\beta}*}(\pi_S^*\theta_1\cdot \ch_1(\mathbb{F}_S))\otimes \pt+\pi_{M_{S,\beta}*}(\pi_S^*\pt\cdot \ch_1(\mathbb{F}_S))\otimes \theta_2, $$
\begin{equation*}\blangle\tau_0(\alpha) \brangle^{\DT_4}_{\beta}=(\theta_1\cdot\beta)\cdot e(M_{S,\beta})=(\theta_1\cdot\beta)\cdot N_0(\beta^2/2). \qedhere  \end{equation*}
\end{proof}

\begin{proof}[Proof of Theorem~\ref{thm:K3xK3 DT invariants} Parts (ii-iv) and (vi-viii)]
We first express the $\DT_4$ descendent invariants as integrals on $M_{S,\beta}$. 
Let $\mathbb{F}_S^{\mathrm{norm}}$ be the normalized universal sheaf on $M_{S,\beta}\times S$, i.e. 
$$\det(\pi_{M_{S,\beta}*}\mathbb{F}_S^{\mathrm{norm}})\cong\oO_{M_{S,\beta}}, $$
where $\pi_{M_{S,\beta}}: M_{S,\beta}\times S\to M_{S,\beta}$ is the projection. 
By Eqn.~\eqref{equ of univ sheaf on prod}, we have 
\[ \mathbb{F}=\mathbb{F}_S\boxtimes \oO_{\Delta_T}. \]
Hence the family $\mathbb{F}$ is normalized if and only if $\mathbb{F}_S$ is so.
Moreover, for the diagonal embedding $\Delta: T\to T\times T$, by GRR, we have
$$\ch(\oO_{\Delta})=\Delta_*(1-2\pt). $$
We obtain
\begin{align*}\tau_1(\gamma)&=A_1\,\pi_{M_{S,\beta}*}(\ch_2(\mathbb{F}_S^{\mathrm{norm}}))\otimes \pt+ A_2\,\pi_{M_{S,\beta}*}(\pi_S^*\pt\cdot\ch_2(\mathbb{F}_S^{\mathrm{norm}}))\otimes 1\\
&\quad\,+\pi_{M_{S,\beta}*}(\pi_S^*(D_1)\cdot\ch_2(\mathbb{F}_S^{\mathrm{norm}}))\otimes D_2.
\end{align*}
By base change to a point, we have  
$$\pi_{M_{S,\beta}*}(\ch_2(\mathbb{F}_S^{\mathrm{norm}}))=1. $$
Combining with Theorem \ref{thm on vir clas}, we obtain that
\[ \blangle\tau_1(\gamma)\brangle^{\DT_4}_{\beta}=A_1\,e(M_{S,\beta})-A_2\,e(T)\,\int_{M_{S,\beta}}c_{\beta^2}(M_{S,\beta})\cdot\pi_{M_{S,\beta}*}\left(\ch_2(\mathbb{F}_S^{\mathrm{norm}}\right)\pi_S^*\pt).
\]
Part (ii) now follows from Proposition~\ref{prop:desc M_Sbeta}(i).

Similarly, for (iii) we have 
$$\ch_5(\mathbb{F}_{\mathrm{norm}})=\ch_3(\mathbb{F}_S^{\mathrm{norm}})\cdot \Delta_*1-2\ch_1(\mathbb{F}_S^{\mathrm{norm}})\cdot \Delta_*(\pt). $$
Hence  
\begin{align*} 
\tau_2(\theta)&=\pi_{M*}(\ch_5(\mathbb{F}_{\mathrm{norm}})\pi_X^*\theta) \\
&=\pi_{M_{S,\beta}*}(\ch_3(\mathbb{F}_S^{\mathrm{norm}})\pi_{S}^*\theta_1)+\pi_{M_{S,\beta}*}(\ch_3(\mathbb{F}_S^{\mathrm{norm}}))\boxtimes\theta_2 -2\pi_{M_{S,\beta}*}(\ch_1(\mathbb{F}_S^{\mathrm{norm}})\pi_{S}^*\theta_1)\boxtimes \pt \\
&=\pi_{M_{S,\beta}*}(\ch_3(\mathbb{F}_S^{\mathrm{norm}})\pi_{S}^*\theta_1)+\pi_{M_{S,\beta}*}(\ch_3(\mathbb{F}_S^{\mathrm{norm}}))\boxtimes\theta_2 -2(\theta_1\cdot \beta)\boxtimes \pt,
\end{align*}
where the last equality is by base change to a point. 
Using Theorem \ref{thm on vir clas}, we obtain
\[ \blangle\tau_2(\theta) \brangle^{\DT_4}_{\beta}=-2(\theta_1\cdot \beta)\,e(  M_{S,\beta})-24\int_{M_{S,\beta}}c_{\beta^2}(M_{S,\beta})\cdot
\pi_{M_{S,\beta}*}\left(\ch_3(\mathbb{F}_S^{\mathrm{norm}}\right)\pi_S^*\theta_1). \]
Thus with Proposition~\ref{prop:desc M_Sbeta}, we obtain that
\[ \blangle\tau_2(\theta) \brangle^{\DT_4}_{\beta} = (\theta_1 \cdot \beta) \left( -2 N_0\left( \frac{\beta^2}{2} \right) + 24 N'\left( \frac{\beta^2}{2} \right) \right). \]
For part (iv), one similarly establishes:
\begin{align*}
\blangle\tau_3(1) \brangle^{\DT_4}_{\beta} 
& =-2\,e(  M_{S,\beta})-24\int_{M_{S,\beta}}c_{\beta^2}(M_{S,\beta})\cdot \pi_{M_{S,\beta}*}(\ch_4(\mathbb{F}_S^{\mathrm{norm}})) \\
& = -2 N_0\left( \frac{\beta^2}{2} \right) + 24 N'\left( \frac{\beta^2}{2} \right),  
\end{align*}

For (vi), we compute using Lemma~\ref{lemma:intersection pairing} that
\begin{align*}
\blangle\tau_0(\gamma), \tau_1(\theta) \brangle^{\DT_4}_{\beta}&=(D_1\cdot\beta)(D_2\cdot\theta_2)\,e( M_{S,\beta}) \\
&\quad -24A_2\int_{M_{S,\beta}}c_{\beta^2}(M_{S,\beta})\cdot \pi_{M_{S,\beta}*}\left(\ch_1(\mathbb{F}_S^{\mathrm{norm}}\right)\pi_S^*\pt)\cdot \pi_{M_{S,\beta}*}\left(\ch_2(\mathbb{F}_S^{\mathrm{norm}}\right)\pi_S^*\theta_1)  
 \\
 &=(D_1\cdot\beta)(D_2\cdot\theta_2)\,e( M_{S,\beta}) 
 -24A_2 (\theta_1 \cdot \beta)\,C(c_{\beta^2}(T_{M_{S,\beta}}) ).
\end{align*}
Since $M_{S,\beta}$ and $S^{[d]}$ are deformation equivalent they share the same Fujiki constants:
\[ C(c_{\beta^2}(T_{M_{S,\beta}})) = C( c_{2d-2}( T_{S^{[d]}} ) ) = N_1( \beta^2/2 ), \]
where $\beta^2 = 2d-2$. This implies the claim.
Finally for (vii) and (viii), we similarly find:
\begin{align*}
\blangle\tau_0(\gamma), \tau_2(1) \brangle^{\DT_4}_{\beta} 
& =-24A_2\int_{M_{S,\beta}}c_{\beta^2}(M_{S,\beta})\cdot \pi_{M_{S,\beta}*}\left(\ch_1(\mathbb{F}_S^{\mathrm{norm}}\right)\pi_S^*\pt)\cdot 
\pi_{M_{S,\beta}*}\left(\ch_3(\mathbb{F}_S^{\mathrm{norm}}\right))  
 \\ & = 0,  \\
%\end{align*}
%\begin{align*}
\blangle\tau_1(\theta), \tau_2(1) \brangle^{\DT_4}_{\beta} 
& =-24\int_{M_{S,\beta}}c_{\beta^2}(M_{S,\beta})\cdot 
\pi_{M_{S,\beta}*}\left(\ch_2(\mathbb{F}_S^{\mathrm{norm}}\right)\pi_S^*\theta_1)\cdot 
\pi_{M_{S,\beta}*}\left(\ch_3(\mathbb{F}_S^{\mathrm{norm}}\right))  \\
& = -24 \cdot (-2) \cdot (\theta_1 \cdot \beta)\,C( c_{\beta^2}(T_{M_{S,\beta}})) \\
& = 48 (\theta_1 \cdot \beta) N_1( \beta^2/2 ). \qedhere
\end{align*}
\end{proof}

\section{Cotangent bundle of $\p^2$}
We consider the geometry $X = T^{\ast} \p^2$.
There is a natural identification of curve classes:
\[ H_2(X,\BZ) = H_2(\p^2, \BZ) = \BZ [ \ell ], \]
where $\ell \subset \p^2$ is a line.

\subsection{GW and GV invariants}
Let $H \in H^2(T^{\ast} \p^2)$ be the pullback of hyperplane class.
We identify $H_2(T^{\ast} \p^2, \BZ) \equiv \BZ$ by its degree against $H$.

\begin{prop} \label{prop:TstartP2}
\begin{gather*}
\blangle \tau_0(H^2), \tau_0(H^2) \brangle^{\GW}_{0, d} =  \frac{(-1)^{d-1}}{d}, \quad 
\blangle \tau_0(H^2) \brangle^{\GW}_{1, d} =  \frac{(-1)^{d-1}}{8} d, \quad 
\blangle \varnothing \brangle^{\GW}_{2, d} =  \frac{(-1)^{d-1}}{128} d^3.
\end{gather*}
\end{prop}
%Consider the reduced Gromov-Witten invariants of $T^{\ast} \p^2$:
%\begin{align*}
%N_{0,d} & = \blangle \tau_0(H^2) \tau_0(H^2) \brangle_{0, d} \\
%N_{1,d} & = \blangle \tau_0(H^2) \brangle_{1, d} \\
%N_{2,d} & = \blangle 1 \brangle_{2, d}
%\end{align*}
%where we identified $H_2(T^{\ast} \p^2, \BZ) \equiv \BZ$ and $H$ is the pullback of the hyperplane class.
%
%\begin{lemma}
%\begin{align*}
%N_{0,d} &= (-1)^{d-1} \frac{1}{d} \\
%N_{1,d} &= (-1)^{d-1} \frac{1}{8} d \\
%N_{2,d} &= (-1)^{d-1} \frac{1}{128} d^3
%\end{align*}
%\end{lemma}
\begin{proof}
This follows by a direct calculation using Graber-Pandharipande virtual localization formula \cite{GP}. We refer to \cite[\S 3]{PZ} for a computation with parallel features.
\end{proof}
Based on Definition \ref{def of g=0 GV inv}, \ref{defn g1 GV primitive}, \ref{defn g2 GV primitive}, we then obtain the following:
\begin{cor}\label{cor on inte on local p2}
\begin{align*}
n_{0,d}(H^2,H^2)&=
\left\{\begin{array}{rcl} 1   &\mathrm{if} \,\, d=1, \\ 
 -1  &\mathrm{if} \,\,  d=2,  \\
  0    & \,\,  \mathrm{otherwise}. 
\end{array} \right.  \\
n_{1,1}(H^2)&=0,  \quad n_{2,1}=0. 
\end{align*}
In particular, Conjecture \ref{conj on integrality} holds for $T^{\ast} \p^2$. 
\end{cor}
\begin{proof}
In genus $0$ and $1$, this follows from a direct calculation using the definition and that
\[ c_2(T_X)|_{\p^2} = -3 H^2. \]
In genus $2$, it remains to determine the nodal invariant $N_{\mathrm{nodal}, \beta}$.
In $H^{\ast}(\p^2 \times \p^2)$, we have
\[ \Delta_{X}|_{\p^2 \times \p^2} = \Delta_{\p^2} \cdot \pr_{1}^{\ast}( c_2(\Omega_{\p^2})) = 3 \pr_1^{\ast}(H^2) \pr_2^{\ast}(H^2). \]
Using Lemma~\ref{div equ on GW} and Eqn.~\eqref{tau31} we find that
\begin{align*}
N_{\mathrm{nodal},1} 
& = \frac{1}{2} \left[ 3 - \left(\blangle \tau_1(c_2(T_X)) \brangle^{\GW}_{0,1} + \blangle \tau_3(1) \brangle^{\GW}_{0,1} \right) \right]  \\
& = \frac{1}{2} \left[ 3 - (-3 + 6) \right] = 0.
\end{align*}
The vanishing $n_{2,1} =0$ follows now from a direct calculation.
\end{proof}
\subsection{$\DT_4$ invariants}
%When we consider curve classes on this $\mathbb{P}^2$, we may work with the local model $T^*\mathbb{P}^2$.Denote $[\ell]\in H_2(T^*\mathbb{P}^2)\cong H_2(\mathbb{P}^2)$ to be the class of a line. 

Let $M_{T^*\mathbb{P}^2,d}$ (resp.~$M_{\mathbb{P}^2,d}$) be the moduli scheme of compactly supported one dimensional stable sheaves $F$ on $T^*\mathbb{P}^2$ 
(resp.~$\mathbb{P}^2$) with $[F]=d[\ell]$ ($d\geqslant   1$) and $\chi(F)=1$. 
\begin{lem}
Let $\iota: \mathbb{P}^2\to T^*\mathbb{P}^2$ be the zero section. Then the pushforward map
\begin{align}\label{equ on iso of TP2} \iota_*: M_{\mathbb{P}^2,d}\to M_{T^*\mathbb{P}^2,d} \end{align}
is an isomorphism. 
\end{lem}
\begin{proof}
The map $\iota_{\ast}$ is obviously injective. 
We show that $\iota_{\ast}$ is also surjective.
As $T^{\ast}\mathbb{P}^2$ admits a birational contraction 
$T^{\ast}\mathbb{P}^2 \to Y$ which contracts the zero section 
$\mathbb{P}^2 \hookrightarrow T^{\ast}\mathbb{P}^2$
to $0 \in Y$ and $Y$ is affine, 
any one dimensional sheaf on $T^{\ast}\mathbb{P}^2$
is set theoretically supported on the zero 
section. It is enough to show that any one dimensional stable sheaf $F$
on $T^{\ast}\mathbb{P}^2$
is scheme theoretically supported on the zero section.

Recall the following fact as stated in \cite[Lem.~2.2]{CMT1}: let $g \colon Z \to T$ be a morphism of $\mathbb{C}$-schemes, and 
take a closed point $t \in Z$. 
Let $Z_t \subset Z$ be the scheme theoretic fiber of $g$ at $t$. 
Suppose that $F \in \Coh(Z)$
is set theoretically supported on $Z_t$ and satisfies $\End(F)=\mathbb{C}$. Then 
$F$ is scheme theoretically supported on $Z_t$. 
%Indeed for any element $u \in \oO_{Y, t}$ vanishing at $t$, 
%the map $g^{\ast}u \colon F \to F$ is nilpotent as 
%$F$ is supported on $Z_t$, hence it is zero as $\End(F)=\mathbb{C}$. 

It should be well-known (and easy) that the scheme theoretic fiber of 
$T^{\ast}\mathbb{P}^2 \to Y$ at $0 \in Y$ 
is the reduced zero section, then surjectivity of $\iota_{\ast}$ follows from the above fact. As 
we cannot find its reference, we give another argument here. Consider the closed embedding 
$T^{\ast}\mathbb{P}^2 \subset \mathcal{O}_{\mathbb{P}^2}(-1)^{\oplus 3}$
induced by the Euler sequence on $\mathbb{P}^2$. 
Note that $\mathcal{O}_{\mathbb{P}^2}(-1)^{\oplus 3}$ is an open subscheme 
of $[\mathbb{C}^6/\mathbb{C}^{\ast}]$,
where $\mathbb{C}^{\ast}$ on $\mathbb{C}^6$ 
by 
$$t(x_1, x_2, x_3, y_1, y_2, y_3)=(tx_1, tx_2, tx_3, t^{-1}y_1, t^{-1}y_2, t^{-1}y_3), $$
and corresponds to $(x_1, x_2, x_3) \neq (0, 0, 0)$. 
The stack $[\mathbb{C}^6/\mathbb{C}^{\ast}]$
admits a good moduli space 
\begin{align}\notag
    [\mathbb{C}^6/\mathbb{C}^{\ast}] \to T:=\mathrm{Spec}\,\mathbb{C}[x_i, y_i]^{\mathbb{C}^{\ast}}
    =\mathrm{Spec}\, \mathbb{C}[x_iy_j : 1\leqslant i, j \leqslant 3].
    \end{align}
    One can easily calculates that the scheme theoretic fiber of the above 
    morphism restricted to 
    $(x_1, x_2, x_3) \neq (0, 0, 0)$ is $(y_1=y_2=y_3=0)$. 
    It follows that the 
    scheme theoretic fiber of $T^{\ast}\mathbb{P}^2 \subset \oO_{\mathbb{P}^2}(-1)^{\oplus 3} 
    \to T$ at $0 \in T$ is the reduced zero section $\mathbb{P}^2$. As $T$ is affine, any  one dimensional stable sheaf is set
    theoretically supported on the (scheme theoretic) fiber of $0 \in T$. Using the above fact, it is also scheme theoretically supported on it. 
   Therefore $\iota_{\ast}$ is surjective. 
    
    Since $M_{\mathbb{P}^2, d}$ is smooth and 
    $\iota_{\ast}$ is bijective on closed points, 
    it remains to show that $\iota_{\ast}$ induces an isomorphisms on 
    tangent spaces. 
    For a one dimensional stable sheaf $F$ on $\mathbb{P}^2$, 
    the tangent space of $M_{T^{\ast}\mathbb{P}^2, d}$ at $\iota_{\ast}F$ is 
    \begin{align}\notag
    \Ext_{T^{\ast}\mathbb{P}^2}^1(\iota_{\ast}F, \iota_{\ast}F) \cong
    \Ext_{\mathbb{P}^2}^1(F, F) \oplus \Hom(F, F \otimes T^*\mathbb{P}^2). 
    \end{align}
    By the Euler sequence and stability, we have 
    $$\Hom(F, F \otimes T^*\mathbb{P}^2) \subset 
    \Hom(F, F \otimes \mathcal{O}_{\mathbb{P}^2}(-1)^{\oplus 3})=0. $$
    Therefore $\iota_{\ast}$ induce an isomorphism of tangent spaces. 
\end{proof}
Then the following result is straightforward. 
\begin{lem}
Under the isomorphism \eqref{equ on iso of TP2}, we have 
$$[M_{T^*\mathbb{P}^2,d}]^{\vir}=\mathrm{PD}\left(e_{\mathrm{red}}^{\frac{1}{2}}\left(\eE xt^1_{\pi_M}(\mathbb{F},\mathbb{F}\boxtimes T^*\mathbb{P}^2),Q\right)\right)\in H_{4}(M_{\mathbb{P}^2,d},\mathbb{Z}). $$
Here $\mathrm{PD}$ denotes the Poincar\'e dual, $e_{\mathrm{red}}^{\frac{1}{2}}$ is the reduced half Euler class as in Definition \ref{def of red vir clas},
$\mathbb{F}$ denotes a universal sheaf of $M_{\mathbb{P}^2,d}$ and $\pi_M: M_{\mathbb{P}^2,d}\times \mathbb{P}^2 \to M_{\mathbb{P}^2,d}$ is the projection.
\end{lem}
%\begin{proof}\end{proof}
\begin{prop}\label{prop on dt4 tp2}
For certain choice of orientation, we have 
%$$\blangle\tau_0(H^2), \tau_0(H^2) \brangle^{\DT_4}_{1}=1, \quad \blangle\tau_0(H^2), \tau_0(H^2) \brangle^{\DT_4}_{2}=-1, \quad \blangle\tau_0(H^2), \tau_0(H^2) \brangle^{\DT_4}_{3}=0,$$
%$$\blangle\tau_1(H^2) \brangle^{\DT_4}_{1}=-\frac{1}{2}, \quad \blangle\tau_1(H^2) \brangle^{\DT_4}_{2}=\frac{1}{2}, \quad \blangle\tau_1(H^2) \brangle^{\DT_4}_{3}=0, $$
%$$\blangle\tau_2(H) \brangle^{\DT_4}_{1}=-\frac{1}{4}, \quad \blangle\tau_2(H)\brangle^{\DT_4}_{2 }=-\frac{1}{4}, \quad \blangle\tau_2(H) \brangle^{\DT_4}_{3}=0, $$
%$$\blangle\tau_3(1) \brangle^{\DT_4}_{1}=-\frac{1}{8}, \quad \blangle\tau_3(1)\brangle^{\DT_4}_{2 }=\frac{1}{8}, \quad \blangle\tau_3(1) \brangle^{\DT_4}_{3}=0, $$
$$\blangle\tau_0(H^2), \tau_0(H^2) \brangle^{\DT_4}_{[\ell]}=1, \quad \blangle\tau_0(H^2), \tau_0(H^2) \brangle^{\DT_4}_{2[\ell]}=-1, \quad \blangle\tau_0(H^2), \tau_0(H^2) \brangle^{\DT_4}_{3[\ell]}=0,$$
$$\blangle\tau_1(H^2) \brangle^{\DT_4}_{[\ell]}=-\frac{1}{2}, \quad \blangle\tau_1(H^2) \brangle^{\DT_4}_{2[\ell]}=\frac{1}{2}, \quad \blangle\tau_1(H^2) \brangle^{\DT_4}_{3[\ell]}=0, $$
$$\blangle\tau_2(H) \brangle^{\DT_4}_{[\ell]}=-\frac{1}{4}, \quad \blangle\tau_2(H)\brangle^{\DT_4}_{2[\ell]}=-\frac{1}{4}, \quad \blangle\tau_2(H) \brangle^{\DT_4}_{3[\ell]}=0, $$
$$\blangle\tau_3(1) \brangle^{\DT_4}_{[\ell]}=-\frac{1}{8}, \quad \blangle\tau_3(1)\brangle^{\DT_4}_{2[\ell]}=\frac{1}{8}, \quad \blangle\tau_3(1) \brangle^{\DT_4}_{3[\ell]}=0. $$
In particular, for $X=T^*\mathbb{P}^2$, we have 
\begin{itemize}
\item Conjecture \ref{conj on DT4/GV} \eqref{g=0 part} holds when $d\leqslant 3$. 
\item Conjecture \ref{conj on DT4/GV} \eqref{g=1 part}, \eqref{g=2 part} hold. 
 \end{itemize}
\end{prop}
\begin{proof}
We present the proof of $d=2$ case (the $d=1$ case follows similarly). The support map 
$$M_{\mathbb{P}^2,2}\stackrel{\cong}{\to} |\oO_{\mathbb{P}^2}(2)|\cong\mathbb{P}^5, \quad F\mapsto \mathrm{supp}(F) $$
is an isomorphism. The normalized universal sheaf satisfies $\mathbb{F}_{\mathrm{norm}}=\oO_{\mathcal{C}}$ for the universal $(1,2)$-divisor
$\mathcal{C}\hookrightarrow \mathbb{P}^5\times \mathbb{P}^2$.  
Let $\pi_M: M_{\mathbb{P}^2,2}\times \mathbb{P}^2\to M_{\mathbb{P}^2,2}$ be the projection. Bott's formula implies 
$$ \dR\mathcal{H}om_{\pi_M}(\oO,\oO(-\mathcal{C})\boxtimes T^*\mathbb{P}^2)\cong \oO_{\mathbb{P}^5}(-1)[-2]^{\oplus 3}, $$
$$ \dR\mathcal{H}om_{\pi_M}(\oO,\oO(\mathcal{C})\boxtimes T^*\mathbb{P}^2)\cong \oO_{\mathbb{P}^5}(-1)^{\oplus 3}, $$
$$ \dR\mathcal{H}om_{\pi_M}(\oO,\oO\boxtimes T^*\mathbb{P}^2)\cong \oO_{\mathbb{P}^5}[-1] .$$
Therefore, we have 
\begin{align*}
&\quad \, \dR\mathcal{H}om_{\pi_M}(\oO_{\mathcal{C}},\oO_{\mathcal{C}}\boxtimes T^*\mathbb{P}^2)[1]\\
&\cong \dR\mathcal{H}om_{\pi_M}(\oO(-\mathcal{C})\to \oO,(\oO(-\mathcal{C})\to \oO) \boxtimes T^*\mathbb{P}^2)[1] \\
&\cong \oO_{\mathbb{P}^5}(-1)^{\oplus 3}\oplus \oO_{\mathbb{P}^5}(1)^{\oplus 3} \oplus \oO_{\mathbb{P}^5}  \oplus \oO_{\mathbb{P}^5}.  
\end{align*}
By Grothendieck-Verdier duality, it is easy to see 
$$\oO_{\mathbb{P}^5}(-1)^{\oplus 3}\oplus \oO_{\mathbb{P}^5}$$ 
is a maximal isotropic subbundle of $\dR\mathcal{H}om_{\pi_M}(\oO_{\mathcal{C}},\oO_{\mathcal{C}}\boxtimes T^*\mathbb{P}^2)[1]$. 
Hence the reduced virtual class satisfies 
$$[M_{T^*\mathbb{P}^2,2}]^{\vir}=\pm e(\oO_{\mathbb{P}^5}(-1)^{\oplus 3})\cap [\mathbb{P}^5] \in H_4(\mathbb{P}^5). $$
Let $h\in H^2(\mathbb{P}^5)$ denote the hyperplane class. It is straightforward to check 
$$\tau_0(H^2)=[h], \,\, \tau_1(H^2)=-\frac{1}{2}h^2, \,\, \tau_2(H)=\frac{1}{4}h^2, \,\, \tau_3(1)=-\frac{1}{8}h^2. $$ 
By integration again the virtual class, we have the desired result for $d=2$ case.

The $d=3$ case can be computed by a torus localization as in \cite{CKM1, CKM2}. 
One sees that for any torus fixed point, the reduced obstruction space
has a trivial factor\footnote{We thank Sergej Monavari for his observation and help on this.} which implies the vanishing of (reduced) invariants.  
\end{proof}

\section{Hilbert scheme of two points on a $K3$ surface} \label{sec:Hilb K3}
Let $S$ be a $K3$ surface.
There are three fundamental conjectures which govern the Gromov-Witten invariants of the Hilbert scheme of points $S^{[n]}$:
\begin{enumerate}
\item[(i)] Multiple cover conjecture (proposed in \cite{O2}, and proven partially in \cite{QuasiK3}) which expresses Gromov-Witten invariants
for imprimitive curve classes as an explicit linear combination of primitive invariants,
\item[(ii)] Quasi-Jacobi form property (proposed in \cite{O1, HilbHAE}),
\item[(iii)] Holomorphic anomaly equation (proposed in \cite{HilbHAE}, see also \cite{O3} for a progress report).
\end{enumerate}

For the Hilbert scheme of two points $S^{[2]}$
these conjectures have been established in genus $0$ by \cite{O1, QuasiK3, HilbHAE}.
Together with \cite{O1} they yield a complete evaluation of all genus $0$ Gromov-Witten invariants of $S^{[2]}$,
that is for all curve classes and all insertions.
% and proven there in genus $0$
%proven in \cite{O1} in genus $0$)
%, proven in \cite{QuasiK3} for genus 0)
%In particular, the works \cite{O1, QuasiK3} yield a full evaluation of all genus $0$ Gromov-Witten invariants of $S^{[2]}$,
%that is for all divisibility and all insertions.
We consider here the case of genus $1$ and genus $2$ Gromov-Witten invariants of $S^{[2]}$ for primitive curve classes.
%in order to determine the higher genus Gopakumar-Vafa invariants.
The strategy is to assume both the quasi-Jacobi form property (ii) and the holomorphic anomaly equation (iii).
Under this assumption, the natural generating series of genus $1$ and $2$ Gromov-Witten invariants
are given in terms of Jacobi forms and are determined up to finitely many coefficients.
Using our earlier computations in ideal geometries we are able to uniquely fix these finitely many coefficients.
Modulo the above conjectures, this leads to a complete evaluation of Gopakumar-Vafa invariants for $S^{[2]}$ in all genera.

%of Gromov-Witten invariants of $S^{[2]}$ are determined up to finitely many coefficients.

\subsection{Quasi-Jacobi forms}\label{sect on qj form}
To state the result we will work with quasi-Jacobi forms.
We refer to \cite{Lib, IOP} for an introduction to quasi-Jacobi forms,
and to \cite[App.~B]{O1} for the variable conventions that we follow here.
We work here entirely on the level of $(q,y)$-series. We need the following series:
\begin{gather*}
E_k(q) = 1 - \frac{2k}{B_k} \sum_{n \geqslant   1} \sum_{d|n} d^{k-1} q^n, 
\quad \quad
\Delta(q) = q \prod_{n \geqslant    1} (1-q^n)^{24}, \\
\Theta(y,q) = (y^{1/2} + y^{-1/2}) \prod_{m \geqslant   1} \frac{ (1 + yq^m) (1 + y^{-1}q^m)}{ (1-q^m)^2 }, \\
\wp(y,q) = \frac{1}{12} - \frac{y}{(1+y)^2} + \sum_{d \geqslant   1} \sum_{m|d} m ((-y)^m - 2 + (-y)^{-m}) q^{d}.
\end{gather*}
Sometimes it will also be convenient to use the following alternative convention of Eisenstein series:
%We also recall the Eisenstein series $G_k(q)=- \frac{B_k}{2 \cdot k} + \sum_{n \geqslant   1} \sum_{d|n} d^{k-1} q^n$
%which satisfy
\[ G_k(q) = - \frac{B_k}{2 \cdot k} E_k = - \frac{B_k}{2 \cdot k} + \sum_{n \geqslant   1} \sum_{d|n} d^{k-1} q^n. \]
The algebra of quasi-Jacobi forms is then the subring of
\[ \BC\left[ \Theta, \, \frac{1}{\Theta} y\frac{d}{d y} \Theta,\, G_2,\, G_4,\, \wp,\, y \frac{d}{dy} \wp \right] \]
consisting of all series which define holomorphic functions $\BC \times \BH \to \BC$ in $(z,\tau)$ where $y = e^{2 \pi i (z+1/2)}$ and $q = e^{2 \pi i \tau}$.
%under the 
%the functions above and their derivatives with respect to the differential operators $q \frac{d}{dq}$ and $y \frac{d}{dy}$,
%satisfying several conditions (e.g. holomorphicity).
A key fact is that the generator $G_2(q)$ is algebraically independent in the algebra of quasi-Jacobi forms from the other generators.
Hence for any quasi-Jacobi form $F(y,q)$ we can speak of its `holomorphic anomaly', which is defined by $\frac{d}{dG_2} F(y,q)$, see \cite{IOP}.

\subsection{Curve classes}\label{sect on curv clas}
Since $X:=S^{[2]}$ is irreducible hyperk\"ahler, recall from Section~\ref{sect on fujiki}
the integral, even, non-degenerate Beauville-Bogomolov-Fujiki form
\[ \mathsf{q}: H^2(X, \BZ) \to \BZ. \]
Since $\mathsf{q}$ is non-degenerate, we obtain an inclusion of finite index
\[ H^2(X,\BZ) \hookrightarrow H^2(X,\BZ)^{\ast} \cong H_2(X,\BZ), \quad D \mapsto (D, - ), \]
where we write $( - , - )$ for the induced inner product on $H^2(X,\BZ)$.
By extending $\mathsf{q}$, we hence obtain a $\BQ$-valued non-degenerate quadratic form
\[ \mathsf{q} : H_2(X,\BZ) \to \BQ, \quad \beta \mapsto (\beta, \beta). \]
Given a class $\beta \in H_2(X,\BZ)$, we write
\[
h_{\beta} = (\beta, - ) \in (H_2(X,\BQ)^{\ast}) \cong H^2(X,\BQ)
\]
for its dual with respect to the Beauville-Bogomolov-Fujiki form $(- , -)$. We have
\[ (h_{\beta}, h_{\beta}) = (\beta, \beta). \]
Let also
\[ c_{BB} \in H^2(X) \otimes H^2(X) \]
be the inverse of the Beauville-Bogomolov-Fujiki form,
i.e. the image of $\mathsf{q} \in H^{2}(X)^{\ast} \otimes H^{2}(X)^{\ast}$
under the natural isomorphism $H^2(X,\BQ)^{\ast} \cong H^2(X,\BQ)$ induced by $\mathsf{q}$.

We will also require the following definition:
\begin{defi} \label{defn:beta coefficient}
Let $F(y,q)$ be a quasi-Jacobi form of index $1$ which satisfies the transformation law of Jacobi forms for the elliptic transformation $z \mapsto z+\tau$ 
(in generators this means it is independent of $\frac{1}{\Theta} y\frac{d}{d y} \Theta$; we will only encounter such kind here).

%or equivalently, when writing in the standard generator does not involve the generator $\mathsf{A}$).
For any class $\beta \in H_2(X,\BZ)$, the $\beta$-coefficient of $F(y,q)$,
\[ F_{\beta} \in \BQ, \]
is defined to be the coefficient of $q^d y^k$ for any $d, k \in \BZ$ such that $(\beta,\beta) = 2d-k^2/2$.
%and given a class $\beta \in H_2(S^{[2]})$ of square $\beta^2 = 2m - s/2$ where $s \in \{0,1\}$.
%Then the coefficient of $F(y,q)$ determined by $\beta$ is
%\[ F(y,q)_{\beta} := \mathrm{q^{m} y^{s}}( F(y,q)). \]
\end{defi}
\begin{rmk}
The choice of $d,k$ is not unique, but the coefficient $F_{\beta}$ is independent of the choice
by the elliptic transformation law of Jacobi forms \cite{EZ}.
\end{rmk}

\subsection{Gromov-Witten invariants}
We first recall the genus $0$ Gromov-Witten invariants of $S^{[2]}$
which are completely determined by the following two quasi-Jacobi forms:
\begin{gather*}
F(y,q) := \frac{\Theta(y,q)^2}{\Delta(q)}, \\
G(y,q) := \frac{\Theta(y,q)^2}{\Delta(q)} ( -\wp(y,q) + \frac{1}{12} E_2(q)).
\end{gather*} 
The first coefficients read:
\begin{align*}
F(y,q) & = \left( y^{-1} + 2 + y \right) q^{-1} + 
\left( 2 y^{-2} + 32 y^{-1} + 60 + 32 y + 2 y^2 \right) \\
&\quad  + (y^{-3} + 60 y^{-2} + 555 y^{-1} + 992 + 555 y + 60 y^2 + 1 y^2)q + \cdots,  \\
G(y,q) & = q^{-1} + (4 y + 30 + 4 y^{-1})
+ (30y^{-2} + 120 y^{-1} + 504 + 120 y + 30 y^3) q + \cdots.
\end{align*}
The following completely determines all primary Gromov-Witten invariants
of $X=S^{[2]}$ in primitive curve classes (see \cite{QuasiK3} for the imprimitive case):
\begin{thm}[\cite{O1, O2}] \label{thm:Hilb genus 0}
Let $\beta \in H_2(X,\BZ)$ be a primitive curve class. We have
\begin{align*}
\ev_{\ast} [\Mbar_{0,1}(X,\beta)]^{\vir} & = G_{\beta} h_{\beta}, \\
\ev_{\ast} \left( \psi_1 \cdot [\Mbar_{0,1}(X,\beta) ]^{\vir} \right) & = \frac{1}{2} F_{\beta} h_{\beta}^2 - \frac{1}{15} \left( G_{\beta} + \frac{1}{4} (\beta, \beta) F_{\beta} \right) c_2(X), \\
\ev_{\ast} \left( \psi_1^2 \cdot [\Mbar_{0,1}(X,\beta) ]^{\vir} \right) & = -3 F_{\beta} \cdot \beta, \\
\ev_{\ast} \left( \psi_1^3 \cdot [\Mbar_{0,1}(X,\beta) ]^{\vir} \right) & = 6 F_{\beta} \, [\pt],
\end{align*}
as well as:
\begin{align*}
& \ev_{\ast}[\Mbar_{0,2}(X,\beta)]^{\vir}= 
 \frac{1}{4} F_{\beta} (h_{\beta}^2 \otimes h_{\beta}^2) + 
G_{\beta} \big(h_{\beta} \otimes \beta + \beta \otimes h_{\beta} + (h_{\beta} \otimes h_{\beta}) \cdot c_{BB} \big)  \\
& \ + \left( -\frac{1}{30} (h_{\beta}^2 \otimes c_2(X) + c_2(X) \otimes h_{\beta}^2) + \frac{1}{900} (\beta, \beta) c_2(X) \otimes c_2(X) \right) \left( G_{\beta} + \frac{1}{4} (\beta, \beta) F_{\beta} \right).
\end{align*}
\end{thm}
Modulo conjectures we have the following evaluation of genus $1,2$ Gromov-Witten invariants:

\begin{thm} \label{thm:Hilb genus 12}
Assume Conjectures A and C of \cite{HilbHAE}. Then 
for any primitive curve class $\beta \in H_2(X,\BZ)$, in genus $1$, we have:
\[ \ev_{\ast} [\Mbar_{1,1}(X,\beta)]^{\vir} =  \frac{1}{2} \CA_{\beta} h_{\beta}^2 + \CB_{\beta} c_2(T_X), \]
where
\begin{align*}
\CA & = 
\frac{\Theta^2}{\Delta} \left( \frac{1}{4} \wp E_2 + \frac{3}{32} E_2^{2} + \frac{1}{96} E_4 \right), \\
\CB & = \frac{\Theta^2}{\Delta} \left(  -\frac{5}{46} \wp^{3} + \frac{5 \wp E_2^{2}}{384} + \frac{5 E_2^{3}}{1536} - \frac{ \wp E_4}{2944} + \frac{5 E_2 E_4 }{4608}
+ \frac{5}{184} \left( y \frac{d}{dy} \wp \right)^2 - \frac{5  E_6}{39744} \right).
\end{align*}
In genus 2, we have
\[ \blangle \varnothing \brangle^{\GW}_{2, \beta} = I_{\beta}, \]
where
\[
I(y,q) = \frac{\Theta^2}{\Delta}
\left(
\frac{5 \wp E_2^{3}}{384} + \frac{25 E_2^{4}}{6144} + \frac{5 \wp E_2 E_4}{384} + \frac{7 E_2^{2} E_4}{3072} - \frac{13 E_4^{2}}{18432} - \frac{\wp E_6}{96} + \frac{E_2 E_6}{1152}
\right).
\]
%$A_{\beta}$ is the $\beta$-coefficient of
%\begin{multline*} \frac{1}{\Theta^2 \Delta} \left( \frac{1}{4} \wp E_2 + \frac{3}{32} E_2^{2} + \frac{1}{96} E_4 \right) =
%( 1/8 y^1 + 1/8 y^{-1}) q^-1 \\
%+ (1/8 y^3 + 315/8 y + 160 + 315/8 y^{-1} + 1/8 y^{-3})q + \ldots
%\end{multline*}
%and $B_{\beta}$ is the $\beta$-coefficient of 
%\[ ... \]
\end{thm}
\vspace{5pt}
The first coefficients of $\CA$ and $\CB$ and $I$ are as follows:
\begin{align*}
\CA(y,q) & = \frac{(y + y^{-1}) }{8} q^{-1} + \left(\frac{1}{8} y^3 + \frac{315}{8} y + 160 +  \frac{315}{8} y^{-1} + \frac{1}{8}  y^{-3}\right)q + \cdots, \\
\CB(y,q) & = \frac{(y + y^{-1} )}{192} q^{-1} + 1 \\
& \quad + \left( \frac{1}{192} y^{-3} + y^{-2} + \frac{385}{64} y^{-1} + \frac{110}{3} + \frac{385}{64} y + y^2 + \frac{1}{192} y^3 \right) q + \cdots, \\
I(y,q) & = \frac{(y + y^{-1})}{128} q^{-1} - \frac{15}{2} \\
& \quad + \left( \frac{1}{128} y^{-3} - \frac{15 }{2}y^{-2} - \frac{11445}{128}y^{-1} - 485 - \frac{11445}{128} y - \frac{15}{2}y^{2} + \frac{1}{128}y^{3} \right) q + O(q^2).
\end{align*}

%CB.q_expansion(3)
%((-1/192*s^4 - 1/192)/s^2)*q^-1 + 1 + ((-1/192*s^12 + s^10 - 385/64*s^8 + 110/3*s^6 - 385/64*s^4 + s^2 - 1/192)/s^6)*q + ((110/3*s^8 - 400/3*s^6 + 540*s^4 - 400/3*s^2 + 110/3)/s^4)*q^2 + O(q^3)

%sage: F2.q_expansion(5)
%((-1/128*s^4 - 1/128)/s^2)*q^-1 - 15/2 + ((-1/128*s^12 - 15/2*s^10 + 11445/128*s^8 - 485*s^6 + 11445/128*s^4 - 15/2*s^2 - 1/128)/s^6)*q + ((-485*s^8 + 4240*s^6 - 13770*s^4 + 4240*s^2 - 485)/s^4)*q^2 + ((11445/128*s^12 - 13770*s^10 + 12278385/128*s^8 - 244860*s^6 + 12278385/128*s^4 - 13770*s^2 + 11445/128)/s^6)*q^3 + ((-15/2*s^16 + 4240*s^14 - 244860*s^12 + 1426608*s^10 - 3155645*s^8 + 1426608*s^6 - 244860*s^4 + 4240*s^2 - 15/2)/s^8)*q^4 + O(q^5)

\subsection{Proof of Theorem~\ref{thm:Hilb genus 12}: Holomorphic anomaly equations}
The global Torelli theorem for hyperk\"ahler varieties
implies that the Hilbert scheme $S^{[2]}$ has a large monodromy group,
we refer to \cite{Markman_Survey} for an introduction.
In our case, as \cite[\S 2.7]{OSY} or \cite{O2} the monodromy implies that for a primitive curve class, we have
\begin{gather} \ev_{\ast} [\Mbar_{1,1}(X,\beta)]^{\vir} =  \frac{1}{2} \CA_{\beta} h_{\beta}^2 + \CB_{\beta} c_2(T_X), \label{ABdef} \\
 \blangle \varnothing \brangle^{\GW}_{2, \beta} = I_{\beta}, \notag \end{gather}
for some constants $\CA_{\beta}, \CB_{\beta}, I_{\beta} \in \BQ$ which only depend on the square $(\beta, \beta)$ of the class.

%\[ \ev_{\ast} [\Mbar_{2,0}(X,\beta)] = C_{\beta} \pt \]
%where $h_{\beta} = (\beta, - )$ is the dual to $\beta$ with respect to the Beauville-Bogomolov form.
To determine these constants, we can work with an elliptic $K3$ surface $S \to \p^1$ with section.
The Hilbert scheme in this case has an induced Lagrangian fibration
$S^{[2]} \to \p^2$ with section.
Let $B,F$ be the section and fiber class of $S$ respectively, and let $A \in H_2(S^{[2]},\BZ)$ be the class
of the locus of non-reduced subschemes supported at a single point.
There exists a natural isomorphism
\[ H_2(S^{[2]},\BZ) = H_2(S,\BZ) \oplus \BZ A \]
given by the Nakajima basis \cite[\S 0.2]{O1}.
For $h \geqslant   0$ and $k \in \BZ$, we consider the classes 
\[ \beta_{h,k} = B + hF + kA, \]
which are of square
\[ (\beta_{h,k}, \beta_{h,k}) = 2h-2- \frac{k^2}{2}. \]
The set of these squares contains all possible squares of curve classes $\beta \in H_2(X,\BZ)$,
we see that any $(X,\beta)$ can be deformed to $(S^{[2]}, \beta_{h,k})$ for some $h,k$.
We form the generating series
\[
\F_{g}(\gamma_1, \ldots, \gamma_n) = \sum_{d \geqslant   0} \sum_{k \in \mathbb{Z}}  \blangle \tau_0(\gamma_1), \ldots, \tau_0(\gamma_n)  \brangle^{S^{[2]}}_{g, B+dF+kA} q^{d-1} y^k.
\]
The $\F_{g}$'s are conjectured to be quasi-Jacobi forms and that their formal derivatives $\frac{d}{dG_2} \F_g$ are 
determined by a holomorphic anomaly equation \cite{HilbHAE}.
%If $g=0$ then this is known and explicitly determined in \cite{O1, HAE}.

Below we will freely use the language of Nakajima operators 
\[ \Fq_i(\alpha) : H^{\ast}(S^{[m]}) \to H^{\ast}(S^{[m+i]}) \]
for all $i \in \BZ$ and $\alpha \in H^{\ast}(S)$, where we follow the conventions of \cite{NOY}.
Given $\gamma_1, \ldots, \gamma_k \in H^{\ast}(S)$ and $n_1, \ldots, n_k \geqslant    1$, we will write
\[ \gamma_1[n_1] \cdots \gamma_k[n_k] := 
\Fq_{n_1}(\gamma_1) \cdots \Fq_{\ell_k}(n_k) 1 \ \in H^{\ast}(S^{[\sum_i n_i]}) , \]
where the unit $1 \in H^{\ast}(S^{[0]})$ is also sometimes called the \textit{vacuum}.

%which also gives a basic overview of the definitions and results in the field.
%We also use the defintiions of quasi-Jacobi forms (see \cite[Appendix B]{O1}).
%We need the following definition:
\begin{proof}[Proof  of Theorem~\ref{thm:Hilb genus 12}: Genus 1 case]
By \cite[Conj.~C]{HilbHAE}, we have for any $\gamma \in H^4(X)$ the following holomorphic-anomaly equation:
\begin{align*}
\frac{d}{dG_2} \F_1( \gamma ) &= \F_0( \gamma , U ) - 2 \F_1( \lambda_1 ; U(\gamma) ) \\
& = y \frac{d}{dy} \F_0( \gamma, F[2] )
+ 2 \F_0( \gamma , F[1]W[1] + 1_S[1] \pt[1] ) 
\\ & \quad \quad 
+ 2 q \frac{d}{dq} \F_0( \gamma, F[1]^2 ) 
- 2 \F_1(\lambda_1; U(\gamma)),
\end{align*}
where 
$$W = B+F$$ 
and
\begin{align*}
U 
& = - \frac{1}{4} \Fq_2 \Fq_{-2}( F_1 + F_2 ) - \Fq_1 \Fq_{-1}( F_1 + F_2 ) \\
& = -\frac{1}{4} \Fq_2 \Fq_{-2}( F_1 + F_2 ) + \Fq_1 \Fq_1 \Fq_{-1} \Fq_{-1}( (F_1 + F_4) \Delta_{23} ) \\
& = -\frac{1}{4} \Fq_2 \Fq'_{2}( F_1 + F_2 ) (1 \otimes 1) + \Fq_1 \Fq_1 \Fq'_{1} \Fq'_{1}( (F_1 + F_4) \Delta_{23} ) (1 \otimes 1),
\end{align*}
which is viewed here both as a morphism $H^{\ast}(S^{[2]}) \to H^{\ast}(S^{[2]})$
and by Poincar\"e duality in the last line as a class in $H^{\ast}(S^{[2]} \times S^{[2]})$
%is the Lefschetz dual 
(we let $\Fq_i'$ denote the Nakajima operator acting on the second factor of $H^{\ast}(S^{[n]}) \otimes H^{\ast}(S^{[m]})$).
Moreover, $\F_g(\lambda_1 ; \ldots )$ stands for the (obvious) generating series
where we integrate also over the tautological class $\lambda_1$, see \cite{HilbHAE}.

We consider the invariant $\F_1( \Fq_1(F)^2 1)$. Using the holomorphic anomaly equation above,
the known results in genus 0 (Theorem~\ref{thm:Hilb genus 0}) and 
the discussion in \cite{O1, OP} on how to reduce the series $\F_1(\lambda_1; ... )$ to genus $0$ invariants, we have:
\begin{align*} \frac{d}{d G_2} \F_{1}( \Fq_1(F)^2 ) 
%& = F_0( [1,F]^2, U ) \\
%& = y \frac{d}{dy} F_0( [1,F]^2, [2,F] ) \\
%& + 2( F_0([1,F]^2, [1,1][1,pt] + [1,F][1,W]) + q \frac{d}{dq} F_0( [1,F]^2, [1,F]^2) ) \\
& = \frac{\Theta^2}{\Delta} \left( -6 \wp + 108 G_2 \right).
\end{align*}
Integrating with respect to $G_2$ yields
\[ \F_{1}( \Fq_1(F)^2 ) = \frac{\Theta^2}{\Delta}\left( a E_4 - 6 \wp G_2 + 54 G_2^2 \right), \]
where $\wp^2$ does not appear, because it would yield the only pole on the left hand side
(contradicting Conjecture A of \cite{HilbHAE} or also monodromy invariance).
By Proposition~\ref{prop:TstartP2}, we have
\[ \mathrm{Coeff}_{q^{-1} y^{-1}}( \F_{1}( \Fq_1(F)^21 ) ) = \blangle \tau_0( F[1]^2) \brangle_{g=1, B-A}^{\GW} = \blangle \tau_0(H^2) \brangle^{\GW,T^{\ast}\p^2}_{g=1,1} = \frac{1}{8}. \]
Solving for $a$ one finds $a = 1/96$, and hence
\begin{align*}
 \F_{1}( \Fq_1(F)^2 1 ) 
& = \frac{\Theta^2}{ \Delta} \left( \frac{1}{4} \wp E_2 + \frac{3}{32} E_2^{2} + \frac{1}{96} E_4 \right).
%& = ( 1/8 y^1 + 1/8 y^{-1}) q^-1 + (1/8 y^3 + 315/8 y + 160 + 315/8 y^{-1} + 1/8 y^{-3})q + O(q^2).
\end{align*}
Similarly, we have
\begin{align*}
\frac{d}{dG_2} \F_1( c_2(X)) 
%&= F_0( c_2(X), U ) - 2 F_1( \lambda_1 ; U(c_2(X)) ) \\
%& = y \frac{d}{dy} F_0( c_2(X), [2,F] )
%+ 2 F_0( c_2(X), [1,F][1,W] + [1,1][1, \pt] ) 
%\\ & \quad \quad 
%+ 2 q \frac{d}{dq} F_0( c_2(X), F[1]^2 ) 
%- 60 F_1(\lambda_1) \\
%& = \frac{\Theta^2}{\Delta} \left( -3210 \wp E_2 + \frac{585}{4} E_2^{2} - \frac{175}{4} E_4 \right)
& = \frac{\Theta^2}{\Delta} \left( -105 \wp E_2 + \frac{135}{8} E_2^{2} - \frac{5}{8} E_4 \right),
\end{align*}
where we used that $U(c_2(X)) = 30 \Fq_1(F) \Fq_1(1) 1$. This yields
\begin{align*}
\F_1( c_2(X))
%& = \frac{\Theta^2}{\Delta} \left( -3210 \cdot \left( \frac{-1}{48} \right) \wp E_2^2 + \frac{585}{4} \left( \frac{-1}{72} \right) E_2^{3} - \frac{175}{4} \frac{-1}{24} E_2 E_4 \right) \\
& = \frac{\Theta^2}{\Delta} \left( \frac{35}{16} \wp E_2^{2} - \frac{15}{64} E_2^{3} + \frac{5}{192} E_2 E_4 + a E_4 \wp + b E_6 \right),
\end{align*}
where, since there are no poles on the left hand side,
the poles in $(D_z \wp )^2$ and $\wp^3$ cancel and give the Eisenstein series $E_6$.
By Proposition~\ref{prop:TstartP2} and since the pair $(S^{[2]}, B+F+A)$ is deformation equivalent to $(S^{[2]},A)$ and we have seen in 
Lemma~\ref{lemma:GW ERCF genus 12}
%Section~\ref{subsubsec:Embedded rational curve family}
that the genus $1$ invariants vanishing in this case, we have:
\[ \blangle c_2(X) \brangle^{\GW}_{1,B-A} = -3 \blangle \tau_0(H^2) \brangle^{\GW,T^{\ast}\p^2}_{g=1,1} = -\frac{3}{8}, \]
\[ \blangle c_2(X) \brangle^{\GW}_{1,B+F+A} = 0. \]
Solving with these conditions for $a$ and $b$, we obtain
%Solving yields $a = -47/16$ and $b = -5/48$ so we get
%\begin{multline*}
\[
 \F_1( c_2(X)) 
= \frac{\Theta^2}{\Delta} \left( \frac{35 \wp E_2^{2}}{16} - \frac{15 E_2^{3}}{64} - \frac{47 \wp E_4}{16} + \frac{5 E_2 E_4}{192} - \frac{5  E_6}{48} \right)  
=
-3/8 (y^{-1} + y) q^{-1} + 828 + O(q).
\]
Finally, by Lemma~\ref{lemma:curve intersections} below and the definition of $\CA_{\beta}, \CB_{\beta}$
in \eqref{ABdef}, the functions
\[ \CA = \sum_{d,k} \CA_{\beta_{h,k}} q^{h-1} y^k,
\quad \quad
\CB = \sum_{d,k} \CB_{\beta_{h,k}} q^{h-1} y^k
\]
satisfy:
% with $\beta_{h,k} = B+hF+kA$ and similarly for $\CB$ are 
%and $\CB = \sum_{d,k} \CB_{\beta_{d,k}}$ are 
%constrained by:
\begin{align*}
 \F_{1}( \Fq_1(F)^2 1 ) = \CA, \quad \F_1( c_2(X)) = 30 \left( q \frac{d}{dq} - \frac{1}{4} \left( y \frac{d}{dy} \right)^2 \right) \CA + 828 \CB.
\end{align*}
This proves the claim by solving for $\CA$ and $\CB$.

We remark that determining $\F_{1}( \Fq_1(F)^2 1 )$ only required a single geometric constraint, namely the computation for class $B-A$.
However, the formula also matches the vanishings obtained from computations in the ideal geometry 
(which applies to classes $\beta \in \{ B, B+F+A \}$).
For $\F_1(c_2(X))$ the system is likewise overdetermined: we only used $2$ of the $3$ available constraints.
\end{proof}
%
%We remark that we only used twone (for the first case) and
%
%We only use a single computation here, namely that of $B-A$.
%However, the formula also matches the vanishings obtained from the computations in the ideal geometry 
%(which applies to classs $\beta \in \{ B, B+F+A \}$).
%The vanishing in class $B+F$ is not automatic.

%(-3/8 y^{-3} + 828 y^{-2} + 46935/8 y^{-1} + 35160 + 46935/8 y^1 + 828 y^2 - 3/8 y^3) q + O(q^2)
%\end{multline*}
%
%
%\vspace{50pt}
%We consider $\gamma = [1, \pt][1,1]$. 
%We get
%\[
%\frac{d}{dG_2} F_1( [1, \pt][1,1] )
%=
%\frac{\Theta^2}{\Delta} \left(
%-3 A^{2} \wp - \frac{9}{4} A^{2} E_2 - 3 \wp^{2} - \frac{423}{4} \wp E_2 + \frac{39}{8} E_2^{2} - 3 A \wp^{\bullet} - \frac{17}{12} E_4
%\right)
%\]
%So
%\[
%F_1( [1,\pt][1,1] )
%=
%\frac{\Theta^2}{\Delta}
%\left( 
%\frac{1}{8} A^{2} \wp E_2 + \frac{3}{64} A^{2} E_2^{2} + \frac{1}{8} \wp^{2} E_2 + \frac{141}{64} \wp E_2^{2} - \frac{13}{192} E_2^{3} + \frac{1}{8} A E_2 \wp^{\bullet} + \frac{17}{288} E_2 E_4 \right) + (...)
%\]
%where $(...)$ is a constant in $E_2$.
%Similarly,
%\[
%\frac{d}{dA} F_1( [1,\pt][1,1] )
%= -\frac{1}{2} F_1( [2,F] ) = \frac{1}{2} y \frac{d}{dy} F_1( [1,F]^2 )
%\]
%which gives

\begin{proof}[Proof  of Theorem~\ref{thm:Hilb genus 12}: Genus 2 case]
%$c_2 \cdot \Fq_1(\pt) \Fq_1(1)1 = 27$
Using Lemma~\ref{lemma:curve intersections} below,
the standard intersections
\[ c_2(X) \cdot \Fq_1(\pt) \Fq_1(1) 1 = 27, \quad \quad c_2(X) \cdot \Fq_1(W) \Fq_1(F)1 = 3 \]
and the genus $1$ part of Theorem~\ref{thm:Hilb genus 12},
the holomorphic anomaly equation of \cite{HilbHAE} reads:
%(with $c_2 \cdot \Fq_1(\pt) \Fq_1(1) = 27$, $c_2 \cdot \Fq_1(W) \Fq_1(F) = 3$) 
%one has
\[ \frac{d}{dG_2} \F_2 = \F_1(U) = 3 \left( 2 q \frac{d}{dq} - \frac{1}{2} \left( y \frac{d}{dy} \right)^2 \right) \CA + 60 \CB. \]
%\begin{align*} \frac{d}{dG_2} \F_2 & = \F_1(U) \\
%& = y \frac{d}{d y} \F_1( \Fq_2(F)1 ) + 2 \F_1\Big( \Fq_1(F) \Fq_1(W) 1 + \Fq_1(\pt) \Fq_1(1) 1 \Big) + 2 q \frac{d}{dq} \F_1( \Fq_1(F)^2 1) \\
%& = 3 \left( 2 q \frac{d}{dq} - \frac{1}{2} \left( y \frac{d}{dy} \right)^2 \right) \CA + 60 \CB.
%\end{align*}
%where we used that
%\begin{align*}
%\F_1( \Fq_1(F)^2 1) & = \CA.
%\F_1( [2,F] ) & = - y \frac{d}{dy} F_1( [1,F]^2 ) = - y \frac{d}{dy} \CA \\
%\F_1( \Fq_1(1) \Fq_1(\pt) ) & = \frac{1}{2} ( 2 q \frac{d}{dq} - \frac{1}{4} \left( y \frac{d}{dy} \right)^2 ) \CA + 27 \CB  \\
%\F_1( \Fq_1(F) \Fq_1(W) ) & = \frac{1}{2} ( 2 q \frac{d}{dq} - \frac{1}{4} \left( y \frac{d}{dy} \right)^2 ) \CA + 3 \CB  \\
%\F_1( \Fq_1(1) \Fq_1(\pt) + \Fq_1(F) \Fq_1(W) ) & = ( 2 q \frac{d}{dq} - \frac{1}{4} \left( y \frac{d}{dy} \right)^2 ) \CA + 30 \CB \\
%\end{align*}
Integration with respect to $G_2$ yields
\[
\F_2 = \frac{\Theta^2}{\Delta}
\left(
\frac{5}{384} \wp E_2^{3} + \frac{25}{6144} E_2^{4} + \frac{5}{384} \wp E_2 E_4 + \frac{7}{3072} E_2^{2} E_4 + \frac{1}{1152} E_2 E_6
%-\frac{11 \wp^{3} E_2}{23} + \frac{5 \wp E_2^{3}}{384} + \frac{25 E_2^{4}}{6144} + \frac{203 \wp E_2 E_4}{8832} + \frac{7 E_2^{2} E_4}{3072} + \frac{11 E_2 ( y \frac{d}{dy} \wp)^{2}}{92} + \frac{25 E_2 E_6}{79488}
+ a E_4^2 + b \wp E_6 \right)
\]
for some $a,b \in \BC$. Here we used that $F_2$ is determined up to the functions $\wp^4, \wp (y \frac{d}{dy} \wp)^2, \wp^2 E_4, E_4$ and that the poles in the first of these functions have to cancel which replaces them with a $E_6 \wp$ term and then that $\wp^2 E_4$ can also not appear because of holomorphicity.
Finally, using the following evaluations (ref.~Proposition~\ref{prop:TstartP2}, Lemma~\ref{lemma:GW ERCF genus 12}):
\[ \blangle \varnothing \brangle^{\GW}_{2, B-A} = \frac{1}{128}, \quad \blangle \varnothing \brangle^{\GW}_{2, B+F+A} = 0 \]
yields that $a=-13/18432$ and $b = -1/96$ and thus 
\[
\F_2(\varnothing) = 
\frac{\Theta^2}{\Delta}
\left(
\frac{5 \wp E_2^{3}}{384} + \frac{25 E_2^{4}}{6144} + \frac{5 \wp E_2 E_4}{384} + \frac{7 E_2^{2} E_4}{3072} - \frac{13 E_4^{2}}{18432} - \frac{\wp E_6}{96} + \frac{E_2 E_6}{1152}
%-\frac{11 \wp^{3} E_2}{23} + \frac{5 \wp E_2^{3}}{384} + \frac{25 E_2^{4}}{6144} + \frac{203 \wp E_2 E_4}{8832} + \frac{7 E_2^{2} E_4}{3072} + \frac{11 E_2 \wp^{\bullet}^{2}}{92} - \frac{13 E_4^{2}}{18432} - \frac{\wp E_6}{96} + \frac{25 E_2 E_6}{79488}
\right).
\]
This implies the result by monodromy invariance
(we even have one more condition to spare, namely the vanishing of $\blangle \varnothing \brangle_{2,B}^{\GW}$).
\end{proof}

\begin{lemma} \label{lemma:curve intersections}
Let $\widetilde{\beta}_{d,k} = W + dF + kA$ %where $W = B+F$ 
and let $h_{d,k} = \widetilde{\beta}_{d,k}^{\vee}$ be the dual. Then
\begin{alignat*}{2}
& h_{d,k}^2 \cdot \Fq_1(F)^2 1 = 2, & \quad \quad \quad & h_{d,k}^2 \cdot \Fq_1(\pt) \Fq_1(1) 1 = 2d - k^2/4, \\
& h_{d,k}^2 \cdot \Fq_2(F) 1 = -2k, & & h_{d,k}^2 \cdot \Fq_1(W) \Fq_1(F) 1 = 2d - k^2/4,
\end{alignat*}
\[ h_{d,k}^2 \cdot c_2(X) = 30 (2d - k^2 / 2). \]
\end{lemma}
\begin{proof}
Let $\delta = c_1( \CO_S^{[2]}) = -\frac{1}{2} \Delta_{S^{[2]}}$ and $D(\alpha) = \Fq_1(\alpha) \Fq_1(1)1$. We have
\[ h_{d,k} = \widetilde{\beta}_{d,k}^{\vee} = D(W) + d D(F) - \frac{k}{2} \delta. \]
This yields, for example
\[ \int h_{d,k}^2 \cdot \Fq_2(F)1 = -k \int \Fq_2(F)1 \cdot \delta \cdot D(W) = -2k. \]
The other cases are similar (use that $\Fq_1(W) \Fq_1(F) \cdot \delta^2 = \Fq_1(\pt) \Fq_1(1) \cdot \delta^2 = -1$).
For the last expression we use the Fujiki constant $C(c_2) = 30$.
\end{proof}

\subsection{Genus $1$ Gopakumar-Vafa invariants}
A hyperk\"ahler variety $X$ is of {\em $K3^{[2]}$-type}
if it is deformation equivalent to the Hilbert scheme $S^{[2]}$ for a $K3$ surface $S$.
For any primitive curve class $\beta \in H_2(X,\BZ)$, 
we define the \textit{genus 1 Gopakumar-Vafa class} 
\[ n_{1,\beta} \in H^4(X,\BQ) \]
by
\[ \int_X n_{1,\beta} \cup \gamma = n_{1,\beta}(\gamma), \quad  \forall\,\, \gamma \in H^4(X,\BQ), \]
where $n_{1,\beta}(\gamma)$ is given in Definition \ref{defn g1 GV primitive}.
In an ideal geometry (ref.~\S \ref{subsec:ideal geometry primitive}), $n_{1,\beta}$ is the class of the surface swept out by the elliptic curves in class $\beta$.

Our discussion above leads to the following formula. Define
\begin{align*}
\CA' & = \frac{\Theta^2}{\Delta} \left( -\frac{1}{4} \wp - \frac{5}{48} E_2 \right) 
= -\frac{(y + y^{-1})}{8} q^{-1} + 6 + O(q), \\
\CB' & = \frac{\Theta^2}{\Delta} \left( -\frac{1}{96} \wp E_2 - \frac{1}{256} E_2^{2} - \frac{1}{2304} E_4 \right)
= -\frac{(y + y^{-1})}{192} q^{-1} + O(q),
\end{align*}
and recall the series $\CA, \CB$ from Theorem~\ref{thm:Hilb genus 12}.
\begin{thm} \label{thm: Hilb GV genus 1} Assume Conjectures A and C of \cite{HilbHAE}. 
For any hyperk\"ahler variety $X$ of $K3^{[2]}$ type and
for any primitive curve class $\beta \in H_2(X,\BZ)$, we have
\[ n_{1, \beta} = \frac{1}{2} a_{\beta} h_{\beta}^2 + b_{\beta} c_2(T_X), \]
where $a_{\beta} = \CA_{\beta} + \CA'_{\beta}$ and $b_{\beta} = \CB_{\beta} + \CB'_{\beta}$.
\end{thm}
\begin{proof}
Since the Chern class $c_2(X)$ is monodromy invariant, we can write
%By monodromy invariance of the Chern class $c_2(X)$ we can write
\[
\frac{1}{24} \ev_{1 \ast}( \ev_2^{\ast}(c_2(X)) [\Mbar_{0,2}(X,\beta)]^{\text{vir}} )
=
\frac{1}{2} \CA'_{\beta} h_{\beta}^2 + \CB_{\beta}' c_2(T_X).
\]
for some $\CA'_{\beta}, \CB'_{\beta}$.
Using Theorem~\ref{thm:Hilb genus 0}, one computes that these are precisely the $\beta$-coefficients
of the functions $\CA', \CB'$ defined above.
The claim now follows from Theorem~\ref{thm:Hilb genus 12} and the definition of genus $1$ Gopakumar-Vafa invariants.
\end{proof}

The integrality conjecture for Gopakumar-Vafa invariants
(Conjecture~\ref{conj on integrality}) would imply that $n_{1,\beta} \in H^4(X,\BQ)$ is an integral class.
We give the following criterion:

\begin{lemma} \label{lemma:integrality}
$n_{1,\beta}$ is integral, i.e. lies in $H^4(X, \BZ)$ if and only if
the following holds:
\begin{enumerate}
\item[(i)] If $(\beta,\beta) \in 2 \BZ$, then $a_{\beta}$ is an even integer and $3 b_{\beta} \in \BZ.$
\item[(ii)] If $(\beta, \beta) = 2d - \frac{1}{2}$, then
$a_{\beta}$, $24 b_{\beta}$, $\frac{1}{8} a_{\beta} - 3 b_{\beta}$ all lie in $\BZ$.
%$a_{\beta} \in \BZ$ and $\frac{(d+1)}{2} a_{\beta} \in \BZ$ and $24 b_{\beta} \in \BZ$
%and $\frac{1}{8} a_{\beta} - 3 b_{\beta} \in \BZ$.
\end{enumerate}
\end{lemma}
\begin{proof}
Using deformation invariance (e.g.~\cite[Cor.~2]{O2}), we may work with $X=S^{[2]}$ for an elliptic $K3$ surface $S$ with $\Pic(S)$ generated by
the class of a section $B$ and the fiber class $F$,
Moreover, we can use the curve class 
\[ \beta := \widetilde{\beta}_{d,k} = W + d F + k A, \]
for $d \geqslant   -1$ and $k \in \{ 0, 1 \}$.
With the notation of Lemma~\ref{lemma:curve intersections}, we then have:
\[
n_{1,\beta} =
\frac{1}{2} a_{\beta} D(W)^2 + d a_{\beta} D(W) D(F) + \frac{1}{2} a_{\beta} d^2 D(F)^2 + 
-k a_{\beta} D(W) \delta - kd D(F) \delta  + \frac{k^2}{4} \delta^2.
\]
By the main result of \cite{Novario}, a basis for the Hodge classes
\[ H^{2,2}(S^{[2]}, \BZ) = H^{4}(S^{[2]},\BZ) \cap H^{2,2}(S^{[2]},\BC) \]
is given by the 7 classes
\begin{gather*} D(W)^2,\ D(W) D(F),\ D(F)^2,\ \delta^2 \\
e_x := \frac{1}{2} \left( D(x)^2 + D(x) \delta \right) \text{ for } x \in \{ W, F \} \\
V := \frac{1}{24} c_2(T_{S^{[2]}}) + \frac{1}{8} \delta^2.
\end{gather*}
The class $n_{1,\beta}$ has the following expansion in this integral basis:
\begin{align*}
n_{1,\beta} & =
\frac{k+1}{2} a_{\beta} D(W)^2 - k a_{\beta} e_W \\
& \quad + \frac{d (d+k)}{2} a_{\beta} F^2 - kd a_{\beta} e_F  \\
& \quad + d a_{\beta} D(W) D(F) \\
& \quad + 24 b_{\beta} V + \left( \frac{k^2}{8} a_{\beta} - 3 b_{\beta} \right) \delta^2.
\end{align*}
If $(\beta, \beta) \in \BZ$ then $k=0$, so integrality of $n_{1,\beta}$
impies (by the first summand) that $a_{\beta} \in 2\BZ$ and by the last summand that $3 b_{\beta} \in \BZ$,
and this is clearly sufficient.
If $(\beta, \beta) = 2d-\frac{1}{2}$, we have $k=1$, which gives $a_{\beta}, 24 b_{\beta}, \frac{1}{8} a_{\beta} - 3 b_{\beta} \in \BZ$ and this is clearly sufficient.
\end{proof}
The criterion of the lemma can be easily checked using a computer program. We obtain:
\begin{cor}\label{cor on chc int 1}
Under the assumptions of Theorem~\ref{thm: Hilb GV genus 1},
$n_{1,\beta}$ is integral for all $(\beta,\beta) \leqslant 100$.
\end{cor}

\begin{example}[A real life example]
Let $F(Y) \subset \mathrm{Gr}(2,6)$ be the Fano variety of lines on a very general cubic 4-fold $Y \subset \p^5$. 
%Let $\beta \in H_2(X,\BZ)$ be the minimal curve class,
%which is of degree $3$ against the Pl\"ucker polarization and satisfies $(\beta,\beta) = 3/2$.
Let $\CU \subset \CO_{\Gr}^{\otimes 6}$ be the universal subbundle on $\Gr(2,6)$ and set
\[ g = c_1(\CU^{\vee}), \quad c=c_2(\CU^{\vee}). \]
The unique primitive curve class is $\beta = \frac{1}{2} g^{\vee}$ and is of square $(\beta,\beta)=3/2$ since $(g,g)=6$.
The basic geometry of these classes is discussed in \cite{Ottem}, in particular we have
\[ c_2(X) = 5 g^2 - 8 c. \]
Theorem~\ref{thm: Hilb GV genus 1} implies that the surface in $F(Y)$ swept out by elliptic curves in class $\beta$ has class:
\[ n_{1,\beta} = 35 (g^2 - c). \]
This is indeed integral and effective $($the surface of lines meeting a given line is $\frac{1}{3} (g^2 - c)$$)$.
\end{example}

\subsection{Genus $2$ Gopakumar-Vafa invariants}
Since we can control now all Gromov-Witten invariants for $S^{[2]}$
in arbitrary genus (for primitive classes),
it is also straightforward to compute genus $2$ Gopakumar-Vafa invariants
(see \cite{NO} for the computation of the nodal invariants):

\begin{thm} \label{thm: Hilb GV genus 2} Assume Conjectures A and C of \cite{HilbHAE}. 
For any hyperk\"ahler variety $X$ of $K3^{[2]}$ type and
for any primitive curve class $\beta \in H_2(X,\BZ)$, we have
\[ n_{2,\beta} = \widetilde{I}_{\beta}, \]
where
\begin{multline*} \widetilde{I}(y,q) = 
\frac{\Theta^2}{\Delta}\Bigg[ 
\frac{5}{384} \wp E_2^{3} + \frac{25}{6144} E_2^{4} + \frac{35}{384} \wp E_2^{2} - \frac{5}{512} E_2^{3} + \frac{5}{384} \wp E_2 E_4 + \frac{7}{3072} E_2^{2} E_4 \\
- \frac{71}{64} \wp E_2 + \frac{27}{512} E_2^{2} - \frac{47}{384} \wp E_4 + \frac{5}{4608} E_2 E_4 - \frac{13}{18432} E_4^{2} - \frac{1}{96} \wp E_6 \\
+ \frac{1}{1152} E_2 E_6 + \frac{9}{8} \wp - \frac{5}{32} E_2 - \frac{23}{1536} E_4 - \frac{5}{1152} E_6 + \frac{1}{8} \Bigg].
\end{multline*}
\end{thm}
Using a computer program, we immediately obtain: 
\begin{cor}\label{cor on chc int 2}
Under the assumptions of Theorem~\ref{thm: Hilb GV genus 2},
$n_{2,\beta}$ is integral for all $(\beta,\beta) \leqslant 138$.
\end{cor}

\subsection{Genus $0$ Gopakumar-Vafa invariants}\label{pf of inte of g0 for hilb2}
For completeness, we also give a proof of the integrality of genus $0$ Gopakumar-Vafa invariants
discussed in the introduction.

\begin{proof}[Proof of Theorem~\ref{thm:Hilb 2 genus 0 integrality}]
Inverting the definition of genus $0$ Gopakumar-Vafa invariants, we have
\[
n_{0,\beta}(\gamma_1, \ldots, \gamma_n) = \sum_{k | \beta} \mu(k) k^{-3+n} \blangle \tau_0(\gamma_1) \cdots \tau_0(\gamma_n) \brangle^{\GW}_{0,\beta},
\]
where $\mu(k)$ is the M\"obius function.
Consider also the ``BPS invariants" introduced in \cite{O2}:
\[
\widetilde{n}_{0,\beta}(\gamma_1, \ldots, \gamma_n) = \sum_{k | \beta} \mu(k) k^{-3+n} (-1)^{[\beta] + [\beta/k]} \blangle \tau_0(\gamma_1) \cdots \tau_0(\gamma_n) \brangle^{\GW}_{0,\beta}.
\]
Then it is straightforward to show that (see \cite[Def.~1]{O2} for the notation $[-]$):
\begin{enumerate}
\item[(i)] If $\left[ \frac{\beta}{\div(\beta)} \right] = 0$, then
$n_{0,\beta}(\gamma_1, \ldots, \gamma_n) = \widetilde{n}_{0,\beta}(\gamma_1, \ldots, \gamma_n)$,
\item[(ii)] If $\left[ \frac{\beta}{\div(\beta)} \right] = 1$, then
\[
n_{0,\beta}(\gamma_1, \ldots, \gamma_n)
=
\begin{cases}
\widetilde{n}_{0,\beta}(\gamma_1, \ldots, \gamma_n) & \text{if } \div(\beta) \text{ is odd or } 4 | \div(\beta), \\
\widetilde{n}_{0,\beta}(\gamma_1, \ldots, \gamma_n) - 
\widetilde{n}_{0,\beta/2}(\gamma_1, \ldots, \gamma_n) & \text{if } \div(\beta) \text{ is even but} \div(\beta/2) \text{ is odd}.
\end{cases}
\]
\end{enumerate}
Hence it suffices to show that $\widetilde{n}_{0,\beta}(\gamma)$ is integral for any effective curve class $\beta \in H_2(X,\BZ)$.
As conjectured in \cite{O2} and proven in \cite{QuasiK3}, the invariant
 $\widetilde{n}_{0,\beta}(\gamma)$ only depends on
\[ \mathsf{q}(\beta), \quad \left[ \beta / \div(\beta) \right], \quad \text{and} \quad (\beta, \gamma). \]
Hence we may assume that $\beta$ is primitive. But here the result follows
since for a very general pair $(X,\beta)$, where $X$ is a hyperk\"ahler variety of $K3^{[2]}$-type,
it is well-known that $\Mbar_{0,1}(X,\beta)$ is an algebraic space (there are no non-trivial automorphisms) of expected dimension (e.g.~\cite[\S 1.1]{OSY}), therefore
$$\ev_{\ast}[\Mbar_{0,1}(X,\beta)]^{\vir} = \ev_{\ast}[\Mbar_{0,1}(X,\beta)]$$
is integral (the same argument also shows the integrality of $n_{0,\beta}(\gamma_1,\ldots,\gamma_n)$ for arbitrary number of markings if $\beta$ is primitive).
\end{proof}

\appendix
\section{Imprimitive curve classes}\label{append}
Let $X$ be a holomorphic symplectic $4$-fold.
We consider here the Gromov-Witten, Gopakumar-Vafa and $\DT_4$ invariants of $X$ in a (possibly imprimitive) curve class $\beta \in H_2(X,\BZ)$.
As discussed in Section~\ref{sect on heur arg}, the ideal geometry of curves in this case is very difficult to control.
Moreover, there are only very few geometries where
both the GW and $\DT_4$ invariants can be completely computed,
and these geometries do not reflect the general structure of the GW/GV/$\DT_4$ invariants.
As such, the general definition of Gopakumar-Vafa invariants $n_{g,\beta}$ for $g>0$ and imprimitive $\beta$ is not clear at this point.
Nevertheless, in this section we define genus $1$ Gopakumar-Vafa invariants for imprimtive curve classes
in the two geometries where all the invariants can be controlled,
and then prove a GV/$\DT_4$ relation.

\subsection{Genus $1$ Gopakumar-Vafa invariants}
There are two geometries where we know all GW and $\DT_4$ invariants:
\begin{itemize}
\item[(i)] The Embedded Rational Curve family of Section~\ref{sec:embedded rational curve family} %(with assumptions (i-iv))
\item[(ii)] The product of two $K3$ surfaces $S \times T$
for all curve classes which lie in $H_2(S,\BZ)$.
\end{itemize}
These geometries are special because
all primary GW and $\DT_4$ invariants with insertion $c_2(T_X)$ vanish.
This implies that the genus $1$ Gopakumar-Vafa invariants do not have any contributions from genus $0$ curves.
Based on a computation in the ideal geometry following Section~\ref{subsubsec:contributions gv genus 1 from genus 1},
one expects that
\begin{align}\label{equ of pot def of g=1}\blangle \tau_0(\gamma) \brangle^{\GW}_{1,\beta}``="\sum_{\begin{subarray}{c} k|\beta  \end{subarray}}\sigma(k)\, n_{1, \beta/k}(\gamma)+ (\cdots), \end{align}
where $\sigma(k):=\sum_{l|k}l$ and $(\cdots)$ stands for contributions from genus 0 curves.
This suggests that for the geometries (i) and (ii), there 
should be no contributions in genus $0$. Hence we make the following adhoc definition in this case:

\begin{defi} \label{rmk g=1 GV}
Let $X$ be a holomorphic symplectic $4$-fold
and $\beta \in H_2(X,\BZ)$ be an effective curve class  of type (i) or (ii) above.
%(which we assume lies in $H_2(S,\BZ)$ for $X = S \times T$).
For any $\gamma \in H^4(X,\BZ)$, we define $n_{1, \beta}(\gamma)$ by:
\begin{equation} \blangle \tau_0(\gamma) \brangle^{\GW}_{1,\beta}=\sum_{\begin{subarray}{c}k\geqslant   1, k|\beta  \end{subarray}}\sigma(k)\, n_{1, \beta/k}(\gamma). \label{gv g1 imprimitive} \end{equation} 
%This will be useful, because in some cases (e.g. when `ideal' rational curve families are embedded) we can argue that there should be no contribution from genus 0 curves(?).
\end{defi}
We also introduce the following:
\begin{defi}\label{def of g=0 m inv}
For any $\gamma\in H^4(X, \BZ)$, we define $n_{0, \beta}(\gamma;\psi)\in \mathbb{Q}$ by the multiple cover formula: 
$$
%\mathrm{GW}_{0, \beta}(\gamma;\psi)
\blangle \tau_1(\gamma) \brangle^{\GW}_{0,\beta}
=\sum_{\begin{subarray}{c}k\geqslant   1, k|\beta  \end{subarray}}\frac{1}{k^{3}}\, n_{0, \beta/k}(\gamma;\psi). $$
\end{defi}
%One can determine $\blangle \tau_1(\gamma) \brangle^{\GW}_{0,\beta}$ (hence also $n_{0, \beta}(\gamma;\psi)$) from $g=0$ primary GW invariants:
We then can prove the following generalization of the genus one part of Conjecture~\ref{conj on DT4/GV}:

\begin{prop} \label{prop:GV genus 1 imprimitive}
Let $X$ be a holomorphic symplectic $4$-fold,
and $\beta \in H_2(X,\BZ)$ be an effective curve class  of type (i) or (ii) above.
For certain choice of orientation, we have
\begin{align}\label{equ on impri DT4/GV}
\blangle\tau_1(\gamma) \brangle^{\DT_4}_{\beta}=-\frac{1}{2}n_{0,\beta}(\gamma;\psi)-\sum_{\begin{subarray}{c}k\geqslant   1, k|\beta  \end{subarray}}n_{1, \beta/k}(\gamma), 
\quad \forall\,\, \gamma \in H^4(X,\BZ).
\end{align}
\end{prop}
\begin{rmk}
The second part in the RHS of the above equality is also consistent with the ideal geometry computation (see Remark \ref{rmk on higher k}). 
\end{rmk}

\subsection{Proof of Proposition~\ref{prop:GV genus 1 imprimitive}: Embedded rational curve family}
Let $X$ be a holomorphic symplectic $4$-fold and let $\beta \in H_2(X,\BZ)$
be a curve class which satisfy conditions (i-iii) of Section~\ref{sec:embedded rational curve family}.
Then by Lemma~\ref{lemma:GW ERCF genus 12}, all genus $1$ GV invariants $n_{1,\beta}(\gamma)$ vanish.
Moreover, by Lemma \ref{lemma:descendent GW on ERCF} and with the notation of that section, we have
\[
n_{0, d\beta}(\gamma;\psi)
=
\begin{cases}
\int_{\CC} j^{\ast}(\gamma)\,c_1(\omega_p) & \text{ if } d=1, \\
\quad \quad  \quad \,\, 0 & \text{ if } d>1.
\end{cases}
\]
Similarly, by Lemma~\ref{lemma:DT4 ERCF}, for certain choice of orientation we have
\[ \blangle\tau_1(\gamma)\brangle^{\DT_4}_{\beta} =
\begin{cases} -\frac{1}{2}\int_{\CC} j^*(\gamma)\cdot c_1(\omega_p) & \text{ if } d=1, \\
\quad \quad  \quad \,\, 0 & \text{ if } d>1.
\end{cases}
\]
This implies the claim. \qed

\subsection{Proof of Proposition~\ref{prop:GV genus 1 imprimitive}: $K3\times K3$}
Let $X = S \times T$ and $\beta \in H_2(S,\BZ)$ be an effective curve class.
Consider a cohomology class $\gamma \in H^4(X,\BZ)$ with K\"unneth decomposition
$$\gamma=A_1\cdot 1\otimes \pt+D_1\otimes D_2+A_2\cdot \pt\otimes 1. $$
The claim follows from the following two lemmata:

\begin{lemma}\label{lem on g=1 gw on prod}
$n_{0, \beta}(\gamma;\psi)=-2 A_1\cdot N_{0}\left(\frac{\beta^2}{2}\right)$.
%where $n_{0,h}(S)$ is given in \eqref{equ on yz for}.
\end{lemma}
\begin{proof} 
This follows from Lemma~\ref{lemma:tau1 K3xK3} and the definition.
\end{proof}

\begin{lemma} \label{lemma:g1bpsK3xK3}
We have
\begin{equation}\sum_{\begin{subarray}{c}k\geqslant   1, k|\beta  \end{subarray}} n_{1, \beta/k}(\gamma) = A_2\, e(T) N_1(\beta^2/2). \label{n1neweval} \end{equation}
\end{lemma}
\begin{proof}
Recall that by \eqref{eq:genus1 K3xK3} we have
$$\blangle \tau_0(\gamma) \brangle^{\GW}_{1,\beta}
%=A_2\,e(T)\,\int_{[\overline{M}_{1,1}(S, \beta)]^{\vir}}\mathrm{ev}^*(\pt)
=A_2 \,e(T)\sum_{\begin{subarray}{c}k\geqslant   1, k|\beta  \end{subarray}}k\cdot N_{1}\left(\frac{\beta^2}{2k^2}\right).$$
Hence by Eqn.~\eqref{gv g1 imprimitive}, $n_{1,\beta}(\gamma)$'s are the unique (recursively defined) integers which satisfy the relation
\begin{equation} \label{xx}
\sum_{k|\beta} \sum_{l | k} l\, n_{1, \beta/k}(\gamma)
=A_2\, e(T) \sum_{k|\beta}k\cdot N_{1}\left(\frac{\beta^2}{2k^2}\right).
\end{equation}
We show that the integers $n_{1,\beta}(\gamma)$ defined by Eqn.~\eqref{n1neweval} satisfy the relation \eqref{xx}.
This then completes the proof.
Indeed, using Eqn.~\eqref{n1neweval}, the right hand side of Eqn.~\eqref{xx} becomes
\begin{align*}
A_2 \,e(T) \sum_{k|\beta}k\cdot N_{1}\left(\frac{\beta^2}{2k^2}\right)
& = \sum_{k| \beta} k \left( \sum_{a | \beta/k} n_{1, \beta/ka}(\gamma)  \right) \\
\text{(set $m:=ka$)}\quad & = \sum_{m| \beta} \sum_{k|m} k n_{1,\beta/m},
\end{align*}
which is precisely the left hand side of Eqn.~\eqref{xx}. Hence Eqn.~\eqref{n1neweval} holds.
\end{proof}

We conclude the proof of Proposition~\ref{prop:GV genus 1 imprimitive}:
By Theorem~\ref{thm:K3xK3 DT invariants}, we have
\[ \blangle\tau_1(\gamma)\brangle^{\DT_4}_{\beta} =A_1\,N_0\left( \frac{\beta^2}{2} \right)-A_2\,e(T)\,N_{1}\left(\frac{\beta^2}{2}\right), \]
which is precisely the right hand side of \eqref{equ on impri DT4/GV} by the two lemmata above. \qed

\providecommand{\bysame}{\leavevmode\hbox to3em{\hrulefill}\thinspace}
\providecommand{\MR}{\relax\ifhmode\unskip\space\fi MR }
\providecommand{\MRhref}[2]{%
 \href{http://www.ams.org/mathscinet-getitem?mr=#1}{#2}}
\providecommand{\href}[2]{#2}

\end{document}